\documentclass{JTPaper}

\title{A General Connected Sum Formula for the Families Bauer-Furuta Invariant}
\author{Joshua Tomlin}
\address{School of Mathematical Sciences, University of Adelaide, Adelaide SA 5005, Australia}
\email{joshua.tomlin@adelaide.edu.au} 

\begin{document}

\pagenumbering{arabic}
\maketitle

\begin{abstract}
    The Bauer-Furuta invariant of a family of smooth 4-manifolds is a stable cohomotopy refinement of the families Seiberg-Witten invariant and is constructed from a finite dimensional approximation of the Seiberg-Witten monopole map. We prove a general formula for the families Bauer-Furuta invariant of a fibrewise connected sum, extending Bauer's non-parameterised formula \cite{BF2}. In a subsequent paper \cite{Part2}, we will use this formula to derive a general connected sum formula for the families Seiberg-Witten invariant which incorporates both the families blow-up formula of Liu \cite{FamiliesBlowupFormula} and the gluing formula of Baraglia-Konno \cite{BaragliaSWGluingFormula}.
\end{abstract}

    \section{Introduction}

The Bauer-Furuta invariant \cite{BF1} of a 4-manifold is a stable cohomotopy refinement of its integer valued Seiberg-Witten invariant. Specifically, it is the equivariant stable cohomotopy class of a finite dimensional approximation of the Seiberg-Witten monopole map. This approach takes a new perspective of studying the monopole map, rather than its moduli space of solutions. It is possible to recover the Seiberg-Witten invariant from the Bauer-Furuta invariant, hence techniques from algebraic topology can be used to circumvent laborious analytical arguments.

In subsequent work, Bauer derived a formula \cite{BF2} for the Bauer-Furuta invariant of a connected sum of 4-manifolds. His idea was to analyse behaviour of monopoles on a 4-manifold with an $n$-component separating neck $N(L) = \coprod_n S^3 \times [-L,L]$ of varying length $2L$. He showed that given a 4-manifold with a separating neck, ends of the necks can be permuted without changing the Bauer-Furuta class of the monopole map. The key insight was that monopoles decay exponentially towards the middle of the neck, hence stretching the neck could be used to control the dynamics in the middle. 

Since Donaldson's suggestion in 1996 \cite{DonaldsonResearchReport}, there has been much interest in studying the Seiberg-Witten equations of 4-manifold families. Several authors including Li-Liu, Nakamura and Ruberman have generalised Seiberg-Witten theory to the families setting \cite{FamiliesSWFirstDef, NakamuraFSW, RubermanFSW}. This body of work involves wall crossing formulas, non-existence of positive scalar curvature metrics, and a particularly noteworthy families blow-up formula \cite{FamiliesBlowupFormula} due to Liu. One striking application of families Seiberg-Witten theory applied to mapping tori is the construction of 4-manifolds with diffeomorphisms that are continuously isotopic to the identity, but not smoothly isotopic \cite{SmoothIsotopyObstruction}.

Since 2019, Baraglia has contributed to the theory of families Seiberg-Witten invariants in several papers \cite{BaragliaObstruction, BaragliaSWGluingFormula, BaragliaConstraints}. In \cite{BaragliaSWGluingFormula}, Baraglia-Konno proved a connected sum formula for the families Seiberg-Witten invariant under some restrictive assumptions. These assumptions simplified the moduli space of one of the summands and avoided cases involving chambers. The overarching goal of this paper and upcoming work \cite{Part2} is to derive a completely general connected sum formula for families Seiberg-Witten invariants extending both Baraglia-Konno's formula and Liu's families blow-up formula. 

This is accomplished by first proving a similar result for the families Bauer-Furuta invariant. Szymik illustrated in \cite{BFExtFamilies} that the Bauer-Furuta invariant naturally extends to the families setting. In this paper, we prove the following families connected sum formula, generalising Bauer's formula for the unparameterised case.
\begin{theorem}\label{T:FamBFCSF}
    For $j \in \{1,2\}$, let $E_j \to B$ be a smooth family of closed, oriented 4-manifolds equipped with a \spinc structure $\sfrak_j$ on the vertical tangent bundle. Assume a section $i_j : B \to E_j$ exists with normal bundle $V_j$ and suppose that $\vphi : V_1 \to V_2$ is an orientation reversing isomorphism satisfying 
    \begin{align*}
        \vphi(i_1^*(\sfrak_{E_1})) \cong i_2^*(\sfrak_{E_2}).
    \end{align*} 
    Then the families Bauer-Furuta class of the fiberwise connected sum $E = E_1 \#_B E_2$ is
    \begin{align}
        [\mu_{E}] &= [\mu_{E_1}] \wedge_\Jcal [\mu_{E_2}].
    \end{align}
\end{theorem}
In 2021, Baraglia-Konno demonstrated how to recover the families Seiberg-Witten invariant from the families Bauer-Furuta invariant via a formulation of the families Seiberg-Witten invariant in equivariant cohomology \cite{BaragliaKonnoBFandSW}. In upcoming work \cite{Part2}, we will use this formulation and the above formula to prove a connected sum formula for the families Seiberg-Witten invariant.

\subsection*{Acknowledgements} The author thanks his advisor, David Baraglia, for his guidance and many helpful discussions.

    \section{Finite dimensional approximation}\label{Ch:FDA}

The Bauer-Furuta invariant is obtained from the stable homotopy class of an approximation of the Seiberg-Witten monopole map by finite dimensional subspaces. In \cite{BF1}, two methods of finite dimensional approximation are described, one method due to Schwarz \cite{Schwarz} and one due to Bauer-Furuta. The Bauer-Furuta method is useful for formally defining the invariant, while the Schwarz method is more useful for practical calculations. Bauer further clarifies their construction in \cite{RefinedBF} using Spanier-Whitehead spectra. We begin by reviewing these two constructions and showing that they are equivalent.

Let $X$ and $Y$ denote pointed topological spaces. We will assume that all maps $f : X \to Y$ are continuous and basepoint preserving. Denote by $[X, Y]$ the set of based homotopy classes of maps between $X$ and $Y$. Let $S^n$ denote the unit sphere in $\R \oplus \R^n$ with $\infty = (1,0) \in S^n$ as the basepoint. The $n$-th homotopy group of $X$ is 
\begin{align*}
    \pi_n(X) &= [S^n, X].
\end{align*}
The suspension functor $\Sigma X = S^1 \wedge X$ defines a map of homotopy groups
\begin{align*}
    \Sigma : \pi_{n}(\Sigma^n X) \to \pi_{n+1}(\Sigma^{n+1} X).
\end{align*} 
The Freudenthal suspension theorem \cite{FreudenthalSusThm} states that this map is an isomorphism for large enough $n$ and the $n$-th stable homotopy group is defined by
\begin{align*}
    \pi_n^s(X) &= \Colim_{\longrightarrow k} \pi_{n+k} (\Sigma^k X).
\end{align*}
In the stable range, the homotopy group $\pi_{n+k} (\Sigma^k X) = [S^{n+k}, \Sigma^k X]$ does not depend on the dimension of the domain and codomain, but only on the difference in dimensions. In the opposite fashion, the $n$-th cohomotopy set of $X$ is given by $\pi^n(X) = [X, S^n]$. The functor $\pi^n$ is now contravariant, but suspension still defines a map $\Sigma : \pi^n(X) \to \pi^{n+1}(\Sigma X)$. The $n$-th stable cohomotopy group of $X$ is defined as
\begin{align*}
    \pi^n_s(X) &= \Colim_{\longrightarrow k} \pi^{n+k} (\Sigma^k X).
\end{align*}
The stable cohomotopy groups define a generalised cohomology theory, and Brown's representability theorem \cite{BrownCohomTheories} guarantees that this cohomology theory is representable. The natural objects for representing stable cohomotopy groups are spectra, in particular, the sphere spectrum $\Sspec^n$ represents the above groups. In order to define the Bauer-Furuta invariant, it will be more convenient to work with spaces that are indexed by finite dimensional subspaces of an infinite dimensional Hilbert space. This is more general than indexing by the natural numbers and allows us to keep track of coordinates when taking suspensions. 

Let $G$ be a compact lie group. For our purposes, $G$ will always be a product of circles. A $G$-space is a pointed topological space $X$ with a continuous left action $G \times X \to X$ that fixes the basepoint. For two $G$-spaces $X$ and $Y$, let $[X,Y]^G$ denote the set of homotopy classes through equivariant pointed maps. The diagonal subgroup of $G \times G$ naturally defines a $G$-action on the smash product $X \wedge Y$. 
\begin{definition}\label{D:GUniverse}
    A $G$-universe $\Ucal$ is an infinite dimensional separable Hilbert space which $G$ acts on by isometeries. It is required that $\Ucal$ contains the trivial representation and that for any irreducible $G$-module $M$, $\Hom_G(M, \Ucal)$ is either zero or infinite dimensional.
\end{definition}
The above condition on $\Hom_G(M, \Ucal)$ guarantees that if we ever suspend by an irreducible representation $M$, then we can suspend by $M$ an arbitrary number of times. A $G$-universe is called complete if it contains a copy of every irreducible representation \cite{ESHTBook}.

For any subspace $U \subset \Ucal$ let $S_U$ denote the unit sphere in $\R \oplus U$, which has a natural basepoint $\infty = (1, 0) \in \R \oplus U$. If $U$ is finite dimensional, then $S_U$ is the one-point compactification of $U$. For any direct sum $V \oplus U$, we have $S_{V \oplus U} = S_V \wedge S_U$. We say that $U$ is a subrepresentation if it is $G$-invariant. In this case the $G$-action can be extended to $S_U \subset \R \oplus U$ by acting trivially on the $\R$ component. Since $G$ acts orthogonally, this fixes the basepoint of $S_U$.

\begin{definition}\label{D:GSpectrum}
    A $G$-spectrum $\Acal = \{\Acal_U\}$ (indexed by $\Ucal$) is a collection of $G$-spaces indexed by subrepresentations $U \subset \Ucal$. Additionally, for any subrepresentation $W \supset U$ with orthogonal decomposition $W = V \oplus U$, there is an equivariant structure homeomorphism
    \begin{align*}
        \sigma_{U,W} : S_V \wedge \Acal_U \to \Acal_W.
    \end{align*}
    The structure maps have the property that for any other subrepresentation $W' \supset W$ with $W' = V' \oplus W$ orthogonally, the following diagram commutes up to homotopy.
    \begin{equation}
        \begin{tikzcd}
            S_{V' \oplus V} \wedge \Acal_U \arrow{rr}{\sigma_{U, W'}} \arrow[swap]{d}{=} && \Acal_{W'}\\
            S_{V'} \wedge S_V \wedge \Acal_U \arrow{rr}{\id \wedge \sigma_{U, W}} && S_{V'} \wedge \Acal_W \arrow[swap]{u}{\sigma_{W,W'}}
        \end{tikzcd}    
    \end{equation}
\end{definition}
\begin{definition}\label{D:SpanierWhiteheadMorphisms}
    The set of morphisms $\Hom_\Ucal(\Acal, \Bcal)$ between two $G$-spectra $\Acal$ and $\Bcal$, both indexed by $\Ucal$, is 
    \begin{align*}
        \Hom_{G,\Ucal}(\Acal, \Bcal) &= \Colim_{U \subset \Ucal} [\Acal_U, \Bcal_U]^G.
    \end{align*}
    This colimit is taken over morphisms of the form 
    \begin{align*}
        [\Acal_U, \Bcal_U]^G \overset{\id_{S_V} \wedge -}{\longrightarrow} [&S_V \wedge \Acal_U, S_V \wedge \Bcal_U]^G = [\Acal_W, \Bcal_W]^G
    \end{align*}
    for $W = V \oplus U$ orthogonally. The identification of $[S_V \wedge \Acal_U, S_V \wedge \Bcal_U]$ with $[\Acal_W, \Bcal_W]$ is given by the structure maps $\sigma^\Acal_{U,W}$ and $\sigma^\Bcal_{U,W}$.
\end{definition}
From the above definition, we see that morphisms between spectra are only defined stably and up to homotopy. This means to define a $G$-spectrum $\Acal$ up to isomorphism, it is enough to specify $\Acal_U$ only for subrepresentations $U$ in an indexing set that is cofinal in the directed system of subrepresentations of $\Ucal$.

\begin{example}[Suspension Spectrum]\label{E:SuspensionSpectra}
    For any $G$-space $A$, define the suspension spectrum $\Sigma A$ by 
    \begin{align*}
        (\Sigma A)_U = S_U \wedge A.
    \end{align*}
    For $W = V \oplus U$ orthogonally, the structure map $\sigma_{U, W} : S_V \wedge (S_U \wedge A) \to S_W \wedge A$ is just the identity. Further, a map $f : A \to B$ induces a map $\Sigma f : \Sigma A \to \Sigma B$ of spectra by taking smash products with the identity. Thus $\Sigma$ embeds pointed topological spaces as a full subcategory inside the category of spectra. We write $\Sspec^n_G$ to denote the suspension spectrum of $S^n$. 
    
    More generally, for any finite dimensional subrepresentation $V \subset \Ucal$ define the suspension $\Sigma^V \Acal$ of a $G$-spectrum $\Acal$ by 
    \begin{align*}
        (\Sigma^V \Acal)_U &= S_V \wedge \Acal_U.
    \end{align*}
    The associated structure maps are the obvious ones induced by smash products with the identity.
\end{example}
\begin{example}[Desuspension]\label{E:Desuspension}
    Fix a finite dimensional subrepresentation $V \subset \Ucal$. For any subrepresentation $W$ containing $V$, write $W = V \oplus U$ orthogonally and define the desuspension $\Sigma^{-V}\Acal$ by 
    \begin{align*}
        (\Sigma^{-V}\Acal)_W &= \Acal_U.
    \end{align*}
    This defines $\Sigma^{-V}A$ up to isomorphism since the set of subrepresentations containing $V$ is cofinal in the directed system of subrepresentations of $\Ucal$. The set of morphisms between $\Sigma^{-V} \Acal$ and another $G$-spectrum $\Bcal$ is given by 
    \begin{align*}
        \Hom_\Ucal(\Sigma^{-V} \Acal, \Bcal) &= \Hom_\Ucal(\Acal, \Sigma^V \Bcal).
    \end{align*}
    That is, $\Sigma^{-V}$ is the left adjoint of $\Sigma^V$. 
\end{example}
\begin{example}[Smash product of spectra]\label{E:SmashProductSpectra}
    Let $\Acal$ be a $G_1$-spectrum indexed by $\Ucal$ and $\Bcal$ be a $G_2$-spectrum indexed by $\Vcal$. The smash product $\Acal \wedge \Bcal$ is a $G_1 \times G_2$-spectrum indexed by the universe $\Ucal \oplus \Vcal$ and, for subrepresentations $U \subset \Ucal$ and $V \subset \Vcal$,
    \begin{align*}
        (\Acal \wedge \Bcal)_{U \oplus V} &= \Acal_U \wedge \Bcal_V.
    \end{align*}
    Let $W_U = U' \oplus U$ and $W_V = V' \oplus V$ orthogonally. The structure map $\sigma_{U \oplus V, W_U \oplus W_V}$ is defined by the following diagram.
    \begin{equation}
        \begin{tikzcd}
            S_{U' \oplus V'} \wedge (\Acal \wedge \Bcal)_{U \oplus V} \arrow{rr}{\sigma_{U \oplus V, W_U \oplus W_V}} \arrow{d}{=} && (\Acal \wedge \Bcal)_{W_U \oplus W_V} \arrow{d}{=}\\
            (S_{U'} \wedge \Acal_U) \wedge (S_{V'} \wedge \Bcal_V) \arrow{rr}{\sigma_{U, W_U} \wedge \sigma_{V, W_V}} && \Acal_{W_U} \wedge \Bcal_{W_V}
        \end{tikzcd}
    \end{equation}
\end{example}
The motivating principle behind defining these objects is that spectra represent equivariant stable cohomology theories. In this case, let $B$ be a compact topological space and fix a universe $\Ucal$. Let $\lambda$ be an equivariant K-theory element $\lambda \in RO(B)$. Write $\lambda = E - F$ where $E$ and $F$ are honest finite dimensional vector bundles over $B$. Assume without loss generality that $F = B \times V$ is trivial with $V \subset \Ucal$ a subrepresentation. Let $TE$ be the Thom space of $E$ and define the Thom spectrum of $\lambda$ by 
\begin{align*}
    T\lambda = \Sigma^{-V} TE.
\end{align*}
\begin{definition}\label{D:EquivStableCohomotopyGroups}
    The $n$-th equivariant stable cohomotopy group of $B$ with coefficients in $\lambda$ is 
    \begin{align}
        \pi^n_{G, \Ucal}(B ; \lambda) &= \Hom_{G, \Ucal}(T\lambda, \Sspec^n) \nonumber\\
        &= \Colim_{U \perp V} [S_U \wedge TE, S_U \wedge S_V \wedge S^n]^G.
    \end{align}
\end{definition}

\subsection{Bauer-Furuta Approximation}\label{S:EquivApprox}

Fix a $G$-universe $\Ucal$. For simplicity, we will assume that $\Hom_G(M, \Ucal)$ is only non-zero for finitely many isomorphism classes of irreducible $G$-modules $M$. Now the isotypical decomposition of $\Ucal$ guarantees that any finite dimensional subspace $V \subset \Ucal$ is contained in a $G$-invariant subspace. Let $B$ be a finite CW complex, which implies that $B$ is compact and Hausdorff. We let $G$ act on $B$ trivially.

Let $H', H \to B$ be $G$-Hilbert bundles, by which we mean locally trivially fibre bundles over $B$ with standard fibre $\Ucal$ and fibre preserving, fibrewise orthogonal $G$-action. Fix an equivariant bundle map $l : H' \to H$ that is fibrewise linear Fredholm.
\begin{definition}\label{D:GBoundedFredholm}
    An equivariant bundle map $f : H' \to H$ is Fredholm (relative to $l$) if $c = f - l$ is continuous and compact. That is, $c$ maps disk bundles to precompact sets.
\end{definition}
A disk bundle $D \subset H$ is a subbundle where each fibre is a closed disk of constant finite radius. We say that a Fredholm map $f$ is bounded if the preimage of any disk bundle is contained in a disk bundle. For any subbundle $V \subset H$ we write $S(V)$ to denote the unit sphere of $V$ and set $S_V = S(\R \oplus V)$. The fibre $(S_V)_b$ over $b \in B$ is a sphere with natural choice of basepoint $\infty_b = (1,0) \in (S_V)_b$. Let $B_\infty \subset S_H$ denote the image of the section at infinity. We identify $H = S_H \setminus B_\infty$ through fibrewise stereographic projection. The boundedness condition for $f$ is equivalent to $f$ admitting a continuous, basepoint preserving extension $f : S_{H'} \to S_H$. Note that this extension is equivariant since $G$ acts orthogonally.

Kupier's theorem \cite{Kuiper} applied to the isotypical decomposition of $H$ implies that there is an equivariant trivialisation $H \to \Ucal \times B$ and all such trivialisations are homotopic. Fix a trivialisation and let $p : H \to \Ucal$ be projection onto the first factor. For $f : H' \to H$ bounded Fredholm, we will often abuse notation by writing $f : H' \to \Ucal$ to also denote $f$ composed with this projection. With this notation in mind, the extension $f : S_{H'} \to S_\Ucal$ factors through the Thom space $TH' = S_{H'}/B_\infty$.

Let $V \subset \Ucal$ be a closed subrepresentation with $p_V : \Ucal \to V$ the orthogonal projection. The orthogonal decomposition $\Ucal = V \oplus V^\perp$ identifies $S_{V} = S(\R \oplus V \oplus 0) \subset S_\Ucal$ and $S_{V^\perp} = S(\R \oplus 0 \oplus V^\perp) \subset S_{\Ucal}$. The spheres $S(V^\perp)$ and $S_{V}$ are disjoint subsets of $S_\Ucal$ and there is a deformation retraction $\rho_V : S_\Ucal \setminus S(V^\perp) \to S_{V}$ defined by
\begin{align}
    &\rho_V(t, v, v') = \frac{1}{\sqrt{t^2 + |v|^2}}(t, v, 0).\label{E:OrthCompDefRet}
\end{align}
This deformation retraction has the property that if $h \in \Ucal \setminus V^\perp$, then $\rho_V(h) = \lambda(h)p(h)$ for some positive and continuous function $\lambda : \Ucal \setminus V^\perp \to \R$.

For any finite dimensional subrepresentation $V \subset \Ucal$, set $V' = l\inv(V)$ and ~$\Vund = V \times B$. Let $p_V, p_{V'}$ be orthogonal projections onto $V$ and $V'$ respectively. We say that $V$ surjects onto $\coker l$ if for each $b \in B$, the projection $\pi : \Ucal \to \Ucal/ \im l_b$ is still surjective when restricted to $V_b$. In this case, $V + (\im l_b)^\perp$ spans $\Ucal$ for all $b \in B$ and $V' \to B$ is a vector bundle of rank $\dim V' = \dim V + \Ind l$. In particular, $V' - \Vund$ represents the virtual index bundle $\ind l$. 

Assume for the moment that the image of $\fhat|_{S_{V'}}$ is disjoint from the sphere $S(V^\perp) \subset S_\Ucal$. Composing with the above deformation retraction, we obtain a map $\rho_{V} \fhat|_{S_{V'}} : S_{V'} \to S_{V}$ which factors through the Thom space.
\begin{definition}
    The map $\vphi_f = \rho_{V} \fhat|_{S_{V'}} : TV' \to S_{V}$ is called the (Bauer-Furuta) finite dimensional approximation of $f$.
\end{definition}
Notice that this definition of finite dimensional approximation depends on the choice of subspace $V$ such that $\fhat|_{S_{V'}}$ is valued in $S_\Ucal \setminus S(V^\perp)$ and the choice of decomposition $f = l + c$.  We will show that such subspaces exist and that the stable homotopy class of $\vphi_f$ is independent of $V, l$ and $c$. 

For any finite dimensional subspace $W \supset V$, write $W$ as an orthogonal sum $W = U \oplus V$ with $U$ the orthogonal complement of $V$ inside $W$. Assuming that $V$ surjects onto $\coker l$, let $W' = \lhat\inv(W)$ with $W' = \tU \oplus V'$ where $\tU$ is the fibrewise orthogonal complement of $V'$ in $W'$. Notice that $l|_{\tU} : \tU \to \Uund$ is an isomorphism of vector bundles, hence $\tU$ is trivial and $T(\tU \oplus V') = S_U \wedge TV'$.
\begin{definition}\label{D:admissibleSubspace}
    A finite dimensional subrepresentation $V \subset \Ucal$ is admissible (with respect to $f$) if it satisfies the following three conditions:
    \begin{enumerate}
        \item $V$ surjects onto $\coker l$.
        \item For any finite dimensional subspace $W \supset V$, the image of $\fhat|_{S_{W'}} : S_{W'} \to S_\Ucal$ is disjoint from the unit sphere $S(W^\perp)$ in $W^\perp$. Consequently the deformation retract $\rho_W : S_\Ucal \setminus S(W^\perp) \to S_{W}$ defines a map 
        \begin{align*}
            \rho_W \fhat|_{S_{W'}} : TW' \to S_{W}.
        \end{align*}
        \item The maps $\rho_W \fhat|_{S_{W'}}$ and $\id \wedge \rho_V \fhat|_{S_{V'}}$ are homotopic under the identifications $TW' = S_{U} \wedge TV'$ and $S_W = S_U \wedge S_V$. 
        \begin{equation}
            \begin{tikzcd}
                TW' \arrow{rr}{\rho_W \fhat|_{S_{W'}}} \arrow{d}[swap]{=} && S_{W} \arrow{d}{=}\\
                S_{U} \wedge TV' \arrow{rr}{\id \wedge \rho_V \fhat|_{S_{V'}}} && S_{U} \wedge S_{V} 
            \end{tikzcd}  
        \end{equation}
    \end{enumerate}
\end{definition}
\begin{proposition}[\cite{BF1} Lemma 2.3]\label{P:admissibleSubspace}
    For any bounded Fredholm map $f = l + c : H' \to H$, there exists an admissible subrepresentation $V \subset \Ucal$.
\end{proposition}
\begin{proof}[Proof sketch.]
    To construct one such $V$, let $D \subset \Ucal$ be the closed unit disk in $\Ucal$. By the boundedness condition, $\fhat\inv(D)$ is contained in a closed disk bundle $D'_R \subset H'$ of radius $R$. Consequently, if $|h'| > R$, then $|f(h)| > 1$. Set $C$ to be the closure of $c(D'_R)$, which is compact. Let $0 < \eps \leq \frac{1}{4}$ and choose a finite covering of $C$ by balls of radius $\eps$ with centers $v_i$ for $i = 1,...,N$. By \cite[Proposition A5]{AtiyahFredholm} there is a finite dimensional subspace $V_0 \subset \Ucal$ with $(\im \lhat_b)^\perp \subset V_0$ for all $b \in B$. Let $V$ be a finite dimensional $G$-invariant subspace containing both $V_0$ and $\spann\{v_1, ..., v_n\}$, which can be obtained using isotypical decomposition. 
    
    By construction $V$ satisfies (1). Further, $V$ has the property that for any subspace $W \supset V$ and $h \in D_R'$, 
    \begin{align*}
        |(1 - p_W)c(h)| < \eps.
    \end{align*}
    Property (2) follows from this bound and the fact that $|f(h)| = 1$ implies $h \in D_R'$. Let $S'$ be the bounding sphere bundle of $D'_R$. Property (3) follows by defining a homotopy $h_t : D'_R \cap W' \to S_\Ucal \setminus S(W^\perp)$ between $\fhat|_{S_{W'}}$ and $\id \wedge \rho_V \fhat|_{S_{V'}}$ on the restricted domain $D'_R \cap W'$. This homotopy is constructed so that the image of $h_t|_{S'}$ does not intersect $W^\perp$ for any $t$. Thus $h_t|_{S'}$ is valued in $S_\Ucal \setminus (D \cap W^\perp)$, which is a contractible subset of $S_\Ucal \setminus S(W^\perp)$. Hence $h_t$ extends over the complementary disk $S_{W'} \setminus (D'_R \cap W')$ and composing with $\rho_W$ gives a homotopy between $\rho_W \fhat|_{S_{W'}}$ and $\id \wedge \rho_V \fhat|_{S_{V'}}$ on $S_{W'}$.
\end{proof}
\begin{definition}[\cite{BF1} Theorem 2.6]\label{T:EquivCohomotopyClass}
   Let $f = l + c : H' \to H$ be an equivariant, bounded Fredholm map and fix an equivariant trivialisation $H \cong \Ucal \times B$. The Bauer-Furuta class of $f$ is the stable homotopy class
    \begin{align*}
        [\vphi_{f}] \in \pi^0_{G,\Ucal}(B ; \Ind l)
    \end{align*}
    where $\vphi_{f} = \rho_V \fhat|_{S_{V'}} : TV' \to S_V$ for any choice of admissible subrepresentation $V \subset \Ucal$. This cohomotopy class is independent of $V$ and the presentation $f = l + c$.
\end{definition}
\begin{proof}
    Fix an admissible subrepresentation $V \subset \Ucal$ and recall that $\ind l = V' - \Vund$, hence the Thom spectrum $T(\ind l)$ is given by $T(\ind l) = \Sigma^{-V} TV'$. It follows that
    \begin{align*}
        \pi^0_{G, \Ucal}(B ; \ind l) &= \Hom(T(\ind l), \Sspec^0)\\
        &= \Colim_{U \subset V^\perp} [S_{U} \wedge TV', S_{U} \wedge S_{V}].
    \end{align*}
    Here $U \subset \Ucal$ is orthogonal to $V$ and the connecting morphisms are given by smash products with the identity. For any other admissible subrepresentation $W$, there is an admissible subrepresentation containing both $V$ and $W$. Hence property (3) implies that the Bauer-Furuta classes corresponding to $V$ and $W$ are stably homotopic, therefore $[\vphi_f] \in \pi^0_{G, \Ucal}(B ; \ind l)$ is well defined. 
    
    To see that $[\vphi_{f}]$ does not depend on the choice of decomposition $f = l + c$, let $f = l_i + c_i$ be two Fredholm decompositions for $i = 0,1$. Let $F_t = l_t + c_t$ for $l_t = (1-t)l_0 + tl_1$ and $c_t = (1-t)c_0 + tc_1$, noting that $F_t = f$ for all $t$. The maps $l_t$ are linear Fredholm and the maps $c_t$ are compact. Now $F$ is a Fredholm map over $B \times [0,1]$ which is certainly bounded. Applying finite dimensional approximation to $F$ gives a homotopy between finite dimensional approximations of $f$ using the two different presentations $f = l_0 + c_0$ and $f = l_1 + c_1$.
\end{proof}

\subsection{Schwarz approximation}\label{S:SchwarzApprox}

In \cite{Schwarz}, Schwarz details an alternative approach to finite dimensional approximation. Let $D' \subset H'$ be a closed disk bundle with boundary sphere bundle $S'$. Fix a trivialisation $H \cong \Ucal \times B$ and let $\Ccal_l(D', H)$ denote the set of continuous maps $f : D' \to \Ucal$ such that $c = f - l|_{D'}$ is compact and $f|_{S'}$ is non-vanishing. 
\begin{definition}\label{D:CompactHomotopy}
    Two Fredholm maps $f_0, f_1 : D' \to H$ are compactly homotopic (relative to $l$) if there is a homotopy $f_t = l + c_t$ with $c_t$ compact and $(f_t)|_{S'}$ non-vanishing for all $t \in [0,1]$. More generally, we say that two bounded Fredholm maps $f_0, f_1 : H' \to H$ are compactly homotopic if there exists a disk $D' \subset H'$ containing $f_0\inv(0) \cup f_1\inv(0)$ on which the restrictions $f_0|_{D'}$ and $f_1|_{D'}$ are compactly homotopic.
\end{definition}
Give $\Ccal_l(D', H)$ the uniform convergence topology so that $\pi_0 (\Ccal_l(D', H))$ is the set of compact homotopy classes relative to $l$. The homotopy class of a Fredholm map $f : H' \to H$ is dull since it is classified by $\ind l$ \cite{SpectralFlow}, but restricting to homotopies through $\Ccal_l(D', H)$ uncovers more interesting behaviour.

Let $f = l + c \in \Ccal_l(D', H)$. Suppose for now that $c(D')$ is contained in a finite dimensional subrepresentation $V \subset \Ucal$. Without loss of generality, we can assume that $(\im l_b)^\perp \subset V$ for all $b \in B$. Let $V' = l\inv(V)$, which is a vector bundle of rank $\dim V' = \dim V + \ind l$. Denote the restriction $f|_{D' \cap V'}$ by
\begin{align*}
    \psi_{f, V} &= f|_{D' \cap V'} : (D' \cap V', S' \cap V') \to (V, V \setminus \{0\})
\end{align*}
Let $W \supset V$ be a finite dimensional subrepresentation containing $V$ with $W = U \oplus V$ orthogonally. Let $W' = l\inv(W)$ so that $W' = \tU \oplus V$ orthogonally with $l|_{\Ut} : \Ut \to \Uund$ an isomorphism. For any map $g : (D' \cap V', S' \cap V') \to (V, V\setminus \{0\})$, define a suspension map
\begin{align*}
    &\Sigma^{\Ut}g : (D' \cap W', S' \cap W') \to (W, W\setminus \{0\}) \\
    &\Sigma^{\Ut}g(u + v) = l(u) + g(v).
\end{align*}
Note that for $w = u + v$, if $\Sigma^{\Ut}g(w) = 0$ then $g(v) = 0$ and $u = 0$, which implies that $w \notin S' \cap W'$. Let $[(A, B) ; (C,D)]$ denote the set of homotopy classes of maps from $(A, B)$ to $(C, D)$ where the homotopies are through maps of pairs. Then $\Sigma^{\Ut}$ descends to a map of homotopy classes
\begin{align}\label{E:SigmaMaps}
    \Sigma^{\Ut} : [(D' \cap V', S' \cap V') ; (V, V\setminus \{0\})] &\to [(D' \cap W', S' \cap W') ; (W, W \setminus \{0\})].
\end{align}
Define 
\begin{align}\label{E:PiDef}
    \Pi_l(D', H) = \Colim_{V \subset \Ucal} [(D' \cap V', S' \cap V') ; (V, V\setminus \{0\})]
\end{align}
where the colimit is taken over the maps given by (\ref{E:SigmaMaps}). Any map $g : (D' \cap V', S' \cap V') \to (V, V\setminus \{0\})$ defines a class $[g] \in \Pi_l(D', H)$ by suspension. The map $\psi_{f, V}$ depends on the choice of subrepresentation $V$, however the class of $[\psi_{f, V}] \in \Pi_l(D', H)$ does not. 
\begin{lemma}\label{L:PhiWellDefined1}
    For $f = l + c \in \Ccal_l(D', H)$, suppose that $V$ and $W$ are finite dimensional subrepresentations which both contain $c(D')$ and surject onto $\coker l$. Then $[\psi_{f, V}]$ and $[\psi_{f,W}]$ are equal classes of $\Pi_l(D', H)$.
\end{lemma}
\begin{proof}
    Assume without loss of generality that $V \subset W$. As before write $W = U \oplus V$ and $W' = \tU \oplus V'$ orthogonally with $l|_\Ut : \Ut \to \Uund$ an isomorphism. For any element $u + v \in W'$ with $u \in \Ut$ and $v \in V'$, we have 
    \begin{align*}
        f|_{S' \cap W'}(u+v) = l(u) + l(v) + c(u + v).
    \end{align*}
    Define a homotopy 
    \begin{align*}
        F_t(u + v) = l(u) + l(v) + (1-t)c(v) + tc(v + u).
    \end{align*} 
    This is a homotopy from $F_0 = \Sigma^{\Ut}\psi_{f, V}$ to $F_1 = \psi_{f, W}$. Additionally, $F_t$ is non-zero on $S' \cap W'$ for all $t \in [0,1]$. To see this, recall that $c(D') \subset V$, hence $F_t(u+v) = 0$ implies that $l(u) = 0$. It follows that $u = 0$ and $|v| = 1$. But $f|_{S' \cap W'}(v) = f|_{S' \cap V'}(v)$, which does not vanish. Thus the classes $[\psi_{f, V}]$ and $[\psi_{f,W}]$ are equal.
\end{proof}
\begin{lemma}\label{L:PhiWellDefined2}
    Suppose $f_t = l + c_t : [0,1] \to \Ccal_l(D', H)$ is a compact homotopy with $c_0(D') \cup c_1(D') \subset V$ for some finite dimensional subrepresentation $V \subset \Ucal$ that surjects onto $\coker l$. Then $\psi_{f_0, V}$ and $\psi_{f_1, V}$ are homotopic as maps of pairs.
\end{lemma}
\begin{proof}
    This follows immediately from Definition \ref{D:CompactHomotopy} since $f_t|_{S'}$ is non-vanishing, hence the restriction 
    \begin{align*}
        (f_t)|_{D' \cap V'} : (D' \cap V', S' \cap V') \to (V, V \setminus \{0\})
    \end{align*}
    is a map of pairs for all $t$ with $f_0 = \psi_{f_0, V}$ to $f_1 = \psi_{f_1, V}$.
\end{proof}
Not all elements $f = l + c \in \Ccal_l(D', H)$ are nice enough to have $c(D')$ contained in a finite dimensional subrepresentation, however it is true that every compact homotopy class has such a representative.
\begin{lemma}\label{L:FredholmDelta}
    For any $f \in \Ccal_l(D', H)$, there exists $\delta > 0$ such that $|f(h)| > \delta$ for all $h \in S'$.
\end{lemma}
\begin{proof}
    Fix $b \in B$ and suppose that there is a sequence $h_n \in S_b'$ with $|f(h_n)| \to 0$. By the weak compactness of $S_b'$, after passing to a subsequence it can be assumed that $h_n \to h$ weakly for some $h \in H_b'$. By the compactness of $c$, after passing to a further subsequence it can be assumed that $c(h_n) \to a$ strongly for some $a \in \Ucal$. Now $l(h_n) = f(h_n) - c(h_n) \to -a$ strongly. Since $l_b$ is Fredholm, its image is closed and $a = l(v)$ for some $v \in (\ker l_b)^\perp$. Write $h_n = x_n + y_n$ for $x_n \in \ker l_b$ and $y_n \in (\ker l_b)^\perp$. Now $l(h_n) = l(y_n) \to -l(v)$. Since $l_b$ is an isomorphism from $(\ker l_b)^\perp$ onto its image, it follows that $y_n \to -v$. Further $x_n = h_n - y_n \to h + v$ weakly, but $\ker l_b$ is finite dimensional so $x_n \to h + v$ strongly as well. Thus $h_n \to h$ strongly and $h \in S_b'$ since $S_b'$ is closed. However $f(h_n) \to 0$ implies that $f(h) = 0$, contradicting the assumption that $f|_{S_b'} \neq 0$. Since $B$ is compact, such a delta can be chosen simultaneously over all fibres.
\end{proof}
\begin{remark}\label{R:FredholmDelta}
    In fact, suppose that $f : H' \to H$ is a bounded Fredholm map with $f\inv(0) \cap S' = \emptyset$. The above argument can be extended to show that there is a $\delta > 0$ with $|f(h)| > \delta$ for every $h \in \overline{H' - D'}$. First choose a closed disk $E'$ such that $|f(h)| \geq 1$ for all $h \notin E'$, which we can assume contains $D'$. Now the argument in the lemma easily extends to the closed, bounded set $\overline{E' - D'}$.
\end{remark}
\begin{corollary}\label{C:fdRangeHomotopic}
    Every element $f = l + c_0 \in \Ccal_l(D', H)$ is compactly homotopic to a map $g = l + c_1 \in \Ccal_l(D', H)$ with $c_1(D')$ contained in a finite dimensional subrepresentation.
\end{corollary}
\begin{proof}
    From Lemma \ref{L:FredholmDelta}, choose $\delta > 0$ such that $|f(h)| > \delta$ for all $h \in S'$. Let $\eps = \frac{\delta}{2}$. Since $D'$ is bounded and $c$ is compact, the closure of $c(D')$ can be covered by finitely many balls of radius $\eps$ with centers $v_1,..., v_n$. Let $V$ be a finite dimensional subrepresentation that contains $\spann\{v_i\}$ and surjects onto $\coker l$. Set $V' = l\inv(V)$ and let $g = l + p_V c$. By construction, $|(1 - p_V)c(h)| < \eps$ for all $h \in D'$. Define a homotopy for $t \in [0,1]$ by 
    \begin{align*}
        F_t = l + (1-t)c + tp_V c.
    \end{align*}
    Notice that for $h \in S'$,
    \begin{align*}
        |F_t(h)| &= |l(h) + c(h) - t(1-p_V)c(h)|\\
        &\geq |f(h)| - t|(1-p_V)c(h)|\\
        &> \frac{\delta}{2}.
    \end{align*}
    Thus $F_t$ is a compact homotopy from $F_0 = f$ to $F_1 = g$.
\end{proof}
For any $f \in \Ccal_l(D', H)$, define $\psi_f = \psi_{g,V}$ for some choice of $g$ compactly homotopic to $f$ with $V$ a finite dimensional subrepresentation that contains $c(D')$ and surjects onto $\coker l$. The map $f \mapsto [\psi_f]$ identifies $\pi_0(\Ccal_l(D', H))$ with a subset of $\Pi_l(D', H)$, which is a result originally due to Schwarz \cite{Schwarz}.
\begin{theorem}[\cite{Berger} Theorem 5.3.20]\label{T:SchwarzCompactHtpyClasses}
    Let $l : H' \to H$ be a linear Fredholm operator and fix a closed disk bundle $D' \subset H'$ with bounding sphere bundle $S'$. The map 
    \begin{align}
        \Psi_{D'} : \pi_0(\Ccal_l(D', H)) &\to \Pi_l(D', H) \nonumber\\
        [f] &\mapsto [\psi_f]
    \end{align}
    is well-defined and injective.
\end{theorem}
\begin{proof}
    Lemma \ref{L:PhiWellDefined1} and \ref{L:PhiWellDefined2} show that $\Psi_{D'} : \pi_0(\Ccal_l(D', H)) \to \Pi_l(D', H)$ is well defined. To prove injectivity, suppose $f = l + c_0$ and $g = l + c_1$ are elements of $\Ccal_l(D', H)$ with $[\psi_f] = [\psi_g]$. After applying $\Sigma$ if necessary, we can assume that there is a compact homotopy $F : (D' \cap V') \times [0,1] \to V$ with $F_0 = f$ and $F_1 = g$ for $V \subset \Ucal$ a finite dimensional subrepresentation that contains $c_0(D') \cup c_1(D')$ and surjects onto $\coker l$. To show that $f$ and $g$ are compactly homotopy, we must extend $F$ to $D' \times [0,1]$.

    Let $v_1,...,v_n$ be an orthonormal basis for $V$ and write $F_t(v) = l(v) + \sum_{i=1}^n c_t^i(v) v_i$ with $c_t^i(v) = \<c_t(v), v_i\>$. Since $(D' \cap V') \times [0,1]$ is a closed subset of $D' \times [0,1]$, the Tietze extension theorem guarantees the existence of a continuous extension $c^i_t : D' \times [0,1] \to \R$ for all $t \in [0,1]$. Define 
    \begin{align*}
        H_t &: D' \times [0,1] \to \Ucal\\
        H_t(v) &= l(v) + \sum_{i=1}^n c^i_t(v) v_i 
    \end{align*}
    It remains to show that $H_t$ is non-vanishing on $S'$ for all $t$. If $H_t(h) = 0$ for $h \in S'$, then $l(h) \in V$. Thus $h \in l\inv(V) = V'$ and $h \in S' \cap V'$. Therefore $H_t(h) = F_t(h) \neq 0$. Thus $H_t$ is a compact homotopy from $f$ to $g$.
\end{proof}

\subsection{Equivalence}\label{S:fdApproxEquivalence}

Let $f = l + c : H' \to H$ be a bounded Fredholm map and $V \subset \Ucal$ an admissible subrepresentation with $V' = l\inv(V)$. Recall that the Bauer-Furuta finite dimensional approximation $\vphi_f$ is given by 
\begin{align*}
    \vphi_f = \rho_V f|_{S_{V'}} : (S_{V'}, B_\infty) \to (S_{V}, \infty)
\end{align*}
This maps factors through the Thom space $TV'$. Let $\Pcal_l(H', H)$ denote the set of equivariant bounded Fredholm maps $f : H' \to H$ relative to $l$. Equip $\Pcal_l(H', H)$ with the topology induced by the uniform metric on $S_H$. Bauer-Furuta approximation defines a map 
\begin{align*}
    \Phi : \Pcal_l(H', H) &\to \pi_{G, \Ucal}^0(B ; \ind l) \\
    f &\mapsto [\vphi_f].
\end{align*}
Alternatively, let $D' \subset H'$ be a closed disk bundle with bounding sphere bundle $S'$ such that $f\inv(0) \subset D'$ and $f\inv(0) \cap S' = \emptyset$, which is guaranteed to exist since $f$ is bounded. Recall that $p_V : H \to V$ is the orthogonal projection and assume for now that $p_V f$ does not vanish on $S' \cap V'$. Then the Schwarz approximation of $f$ is given by
\begin{align*}
    \psi_f &= p_V f|_{D' \cap V'}: (D' \cap V', S' \cap V') \to (S_{V}, S_{V} \setminus \{0\}).
\end{align*}
Schwarz approximation defines another map 
\begin{align}
    \Psi_{D'} : \pi_0(\Ccal_l(D', H)) &\to \Pi_l(D', H) \nonumber\\
    [f] &\mapsto [\psi_f]
\end{align}
We will leverage the properties of Schwarz approximation to prove that $\Phi$ descends to a well defined map from $\pi_0(\Pcal_l(H', H))$ to $\pi_{G, \Ucal}^0(B ; \ind l)$ and that this map is a bijection. At a surface level, it looks as if Schwarz approximation depends on the appropriately chosen disk bundle $D' \subset H'$. However, enlarging $D'$ does not change the Schwarz approximation of $f$ by the following lemma.
\begin{lemma}\label{L:ExtendCptHtpy}
    Two elements $f_0, f_1 \in \Pcal_l(H', H)$ are homotopic through bounded Fredholm maps if and only if they are compactly homotopic on some disk bundle $D' \subset H'$ that contains $f_0\inv(0) \cup f_1\inv(0)$. 
\end{lemma}
\begin{proof}
    Suppose $f_t : [0,1] \to \Pcal_l(H', H)$ is a homotopy so that $f_t$ is a bounded Fredholm map for each $t \in [0,1]$. Compactness of the unit interval and continuity of the homotopy guarantees the existence of a disk $D' \subset H'$ such that $f_t\inv(0) \subset D'$ and $f_t\inv(0) \cap S' = \emptyset$ for all $t \in [0,1]$. Thus $f_0|_{D'}$ and $f_1|_{D'}$ are compactly homotopic. 

    Suppose instead that there is a disk bundle $D' \subset H'$ of radius $R'$ on which $f_0|_{D'}$ and $f_1|_{D'}$ are compactly homotopic. Let $F_t : D' \to H$ be such a homotopy with $F_t\inv(0) \cap S' = \emptyset$ for all $t \in [0,1]$. For any $x \in \overline{H' - D'}$, let $s = \frac{R'}{|x|}x \in S'$ and extend $F_t$ on $\overline{H' - D'}$ by 
    \begin{align*}
        F_t(x) &= \frac{|x|}{R'}F_t(s).
    \end{align*}
    Now for each $t \in [0,1]$, $F_t : H' \to H$ is Fredholm and since $0 \notin F_t(S')$, Lemma \ref{L:FredholmDelta} guarantees that $F_t$ is bounded. Hence $[F_0] = [F_1]$ as elements of $\pi_0(\Pcal_l(H', H))$. We claim that $f_0$ is homotopic to $F_0$ through bounded Fredholm maps. Such a homotopy $h_t : H' \to H$ is given by $h_t|_{D'} = f_0|_{D'}$ and, for $x \in \overline{H' - D'}$,
    \begin{align*}
        h_t(x) &= \left(\frac{|x|}{R'}\right)^t f_0\left(\left(\frac{|x|}{R'}\right)^{-t} x\right).
    \end{align*}
    Similarly, $[F_1] = [f_1]$ in $\pi_0(\Pcal_l(H', H))$ and the result follows.
\end{proof}
To simplify notation, set 
\begin{align*}
    D'_- &= D' \cap V' \\
    D'_+ &= \overline{S_{V'} - D'_-} \\
    S_0' &= S' \cap V'.
\end{align*}
That is, $D'_\pm$ are the two hemispheres of $S_{V'}$ with $S'_0$ the equator. Define an intermediary map 
\begin{align*}
    \phi_f &= \rho_V f|_{S_{V'}} : (S_{V'}, D'_+) \to (S_{V}, S_{V} \setminus \{0\}).
\end{align*}
This definition of $\phi_f$ assumes that $\rho_V f$ does not vanish on $D'_+$. The following lemma shows that the finite dimensional subrepresentation $V$ can be chosen to simultaneously make $\vphi_f, \psi_f$ and $\phi_f$ maps of pairs.
\begin{lemma}\label{L:ThreePhis}
    Let $f = l + c : H' \to H$ be a bounded Fredholm map and fix a disk bundle $D' \subset H'$ such that $f\inv(0) \subset D'$ and $f\inv(0) \cap S' = \emptyset$. There exists a finite dimensional subrepresentation $V \subset \Ucal$ such that: 
    \begin{enumerate}
        \item $V$ is an admissible subrepresentation as in Definition \ref{D:admissibleSubspace},
        \item $p_V f$ is non-vanishing on $S'_0$,
        \item $\rho_V f|_{S_{V'}}$ is non-vanishing on $D'_+$.
    \end{enumerate}
    These properties translate to any finite dimensional subrepresentation $W \supset V$.
\end{lemma}
\begin{proof}
    Since $f$ is bounded, we can assume that $f\inv(D) \subset D'$ where $D \subset \Ucal$ is the closed unit disk. As explained in Remark \ref{R:FredholmDelta}, choose a $\delta > 0$ such that $|f(h)| > \delta$ for $h \in \overline{H' - D'}$. Let $\eps = \min\{\frac{1}{4}, \frac{\delta}{2}\}$. Cover the closure of $c(D')$ by finitely many $\eps$-balls with centres $v_1, ..., v_n$ and set $V = \spann\{v_i\}$. As seen before, we can enlarge $V$ to be a subrepresentation that surjects onto $\coker l$. Now $V$ has the property that $|(1 - p_V)f(h)| < \eps$ for all $h \in D' \cap V'$ and is an admissible subrepresentation by Proposition \ref{P:admissibleSubspace}.  Since $|f(h)| > \delta$ for $h \in S'$, it follows that $|p_V f(h)| > \frac{\delta}{2}$ for $h \in S' \cap V'$.

    Suppose that $\rho_V f(h) = 0$ for some $h \in S_{V'}$. Notice from the definition of $\rho_V$ in (\ref{E:OrthCompDefRet}) that this implies that $f(h)$ is finite with $p_Vf(h) = 0$ and $|(1-p_V)f(h)| < 1$. This means that $|f(h)| < 1$ and $h \in D' \cap V'$. Therefore $|(1-p_V)f(h)| < \eps < \delta$ and $h \notin D'_+$ since $|f(h)| < \delta$. That is, $\rho_V f(h)$ is non-vanishing on $D'_+$. For any finite dimensional $W \supset V$, it is still the case that $|(1-p_W)f(h)| < \eps$ for $h \in D' \cap W'$ and the argument can be repeated.
\end{proof}
Consider the following diagram where $a, b$ and $c$ are the obvious inclusions:
\[
\begin{tikzcd}
    (S_{V'}, B_\infty) \arrow{r}{a} \arrow{d}{\vphi_f} & (S_{V'}, D'_+) \arrow{dr}{\phi_f} & (D'_-, S'_0) \arrow[swap]{l}{b} \arrow[dashed]{d}{\psi_f}\\
    (S_{V}, \infty) \arrow{rr}{c} && (S_{V}, S_{V} \setminus \{0\})
\end{tikzcd}
\]
The dashed arrow $\psi_f$ does not make the diagram commute, but we will show that it does commute up to homotopy. These inclusions induce functions between homotopy classes of maps of pairs:
\[
\begin{tikzcd} \left[(S_{V'}, B_\infty) \, ; (S_{V}, \infty)\right]  \arrow{r}{c_*} & \left[(S_{V'}, B_\infty) \, ; (S_{V}, S_{V} \setminus \{0\})\right]\\
    \left[(D'_-, S'_0) \, ; (S_{V}, S_{V} \setminus \{0\})\right]  &  \arrow{l}[swap]{b^*} \arrow[swap]{u}{a^*}
    \left[(S_{V'}, D'_+) \, ; (S_{V}, S_{V} \setminus \{0\})\right]
\end{tikzcd}
\]
\begin{proposition}\label{P:BF-SEquiv}
    The maps $a^*, b^*$ and $c_*$ induced by inclusions are bijections. The composition $\xi = b^* (a^*)\inv c_*$ defines a bijection 
    \begin{align*}
        \xi : \left[(S_{V'}, B_\infty) \, ; (S_{V}, \infty)\right] \to \left[(D'_-, S'_0) \, ; (S_{V}, S_{V} \setminus \{0\})\right]
    \end{align*}
    which identifies $[\vphi_f]$ with $[\psi_f]$.
\end{proposition}
\begin{proof}
    Contracting $D'_+$ radially to $\infty$ fibrewise defines a homotopy $F_t : S_{V'} \to S_{V'}$ with $F_0 = \id$ and $F_1(D'_+) = B_\infty$. The compositions $aF_1 : (S_{V'}, D'_+) \to (S_{V'}, D'_+)$ and $F_1a : (S_{V'}, B_\infty) \to (S_{V'}, B_\infty)$ are both homotopy equivalent to the identity through maps of pairs, hence $a$ is a homotopy equivalence of pairs and $a^*$ is bijection.

    Since $S_V \setminus \{0\}$ is contractible, any $f : S_{V'} \to S_{V}$ with $f(B_\infty) \subset S_V \setminus \{0\}$ can be composed with a homotopy that contracts $f(B_\infty)$ to $\infty$. Hence $c_*$ is surjective. For injectivity let $g_0, g_1 : S_{V'} \to S_{V}$ be maps with $g_i(B_\infty) = \infty$ and suppose that there is a homotopy $g_t$ from $g_0$ to $g_1$ with $g_t(B_\infty) \subset S_V \setminus \{0\}$. Since $S_{V'} \times I$ is compact, there is an open neighbourhood $U \subset S_{V}$ of $0$ such that $g_t(B_\infty) \subset S_{V} \setminus U$ for all $t$. Thus $[g_0] = [g_1]$ as elements of $[(S_{V'}, B_\infty) ; (S_{V}, S_{V} \setminus U)]$. By the same reasoning as above, the inclusion $(S_{V}, \infty) \to (S_{V}, S_{V} \setminus U)$ is a homotopy equivalence of pairs. Hence $[g_0] = [g_1]$ as elements of $[(S_{V'}, B_\infty), (S_{V}, \infty)]$.

    To see that $b^*$ is surjective, suppose $f : D'_- \to S_{V}$  is a map with $f|_{S'_0}$ valued in $S_{V} \setminus \{0\}$. Locally, $S_{V'}$ is obtained from $D'_-$ by attaching $D'_+$ over $S'_0$. Since $S_{V} \setminus \{0\}$ is contractible, $f|_{S'_0}$ can be extended to $D'_+$ by a null homotopy while remaining valued in $S_{V} \setminus \{0\}$. This construction can be globalised using a partition of unity, thus $f$ extends to $S_{V'}$ with $f(D'_+) \subset S_{V} \setminus \{0\}$.

    For injectivity, let $b' : D'_- \to S_{V'}$ and $b'' : S'_0 \to D'_+$ be inclusions with mapping cones $C_{b'}$ and $C_{b''}$. Recall that the cofibersequence $(D'_-, S'_0) \to (S_{V'}, D'_+) \to (C_{b'}, C_{b''})$ induces an exact sequence \cite[III Prop 3.9]{Adams}
    \begin{align*}
        [(C_{b'}, C_{b''}) ; (S_{V}, S_{V} \setminus \{0\})] \to [(S_{V'}, D'_+) ; (S_{V}, S_{V} \setminus \{0\})] \overset{b^*}{\to} [(D'_-, S'_0) ; (S_{V}, S_{V} \setminus \{0\})].
    \end{align*}
    The cone $C_{b'}$ deformation retracts onto $C_{b''}$, hence 
    \begin{align*} 
        [(C_{b'}, C_{b''}) ; (S_{V}, S_{V} \setminus \{0\})] &\cong [(C_{b''}, C_{b''}) ; (S_{V}, S_{V} \setminus \{0\})]\\
        &= [C_{b''}, S_{V} \setminus \{0\}].
    \end{align*}
    However $[C_{b''}, S_{V} \setminus \{0\}]$ is trivial since $S_{V} \setminus \{0\}$ is contractible. Thus $b^*$ is injective by the exactness of the cofibersequence.

    It remains to show that $[\psi_f] = [b^*(a^*)\inv c_* (\vphi_f)]$. We have that $c_* \vphi_f = a^* \phi_f$, thus it is enough to show that $[\psi_f] = [b^*\phi_f]$. Note that both $\psi_f|_{S'_0}$ and $b^* \phi_f|_{S'_0}$ are valued in $V \setminus \{0\} \subset H \setminus V^\perp$. Recall that $\rho_V f|_{S'_0} = \lambda p_Vf|_{S'_0}$ for some positive continuous function $\lambda : H \setminus V^\perp \to \R$, hence the straight line homotopy from $b^* \phi_f$ to $\psi_f$ never vanishes.
\end{proof}
\begin{corollary}[\cite{RefinedBF} Theorem 2.1]\label{C:PhiBijection}
    Given a choice of trivialisation $H \cong \Ucal \times B$, the map $\Phi$ descends to a bijection 
    \begin{align}
        \Phi : \pi_0(\Pcal_l(H', H)) \to \pi_{G, \Ucal}^0(B ; \ind l).
    \end{align}
\end{corollary}
\begin{proof}
    First suppose that $f_0, f_1 \in \Pcal_l(H', H)$ are homotopic through bounded Fredholm maps. Then by Lemma \ref{L:ExtendCptHtpy}, $f_0$ and $f_1$ are compactly homotopic on an appropriately chosen $D'$ and $[f_0] = [f_1]$ in $\pi_0(\Ccal_l(D', H))$. Thus $\Psi|_{D'} f_0 = \Psi_{D'} f_1$ and applying $\xi\inv$ gives $\Phi f_0 = \Phi f_1$, hence $\Phi$ is well defined. This also proves injectivity since if $\Phi f_0 = \Phi f_1$, then $\Psi|_{D'} f_0 = \Psi_{D'} f_1$ by applying $\xi$. Hence $f_0$ and $f_1$ are compactly homotopic on $D'$ by Theorem \ref{T:SchwarzCompactHtpyClasses}, and $[f_0] = [f_1]$ in $\Pcal_l(H', H)$ by Lemma \ref{L:ExtendCptHtpy}.
    
    For surjectivity, a class $[f] \in \pi_{G,\Ucal}^0(B ; \ind l)$ is represented by a pointed map $f : TV' \to S_{V}$ for $V \subset \Ucal$ an admissible subspace with $V' = l\inv(V)$. After possible suspension we can assume that $f\inv(\infty) = [B_\infty] \in TV'$. For $\pi : S_{V'} \to TV'$ the projection, this means that $(f \circ \pi)\inv(\infty) = B_\infty$ and the restriction $f \circ \pi : V' \to V$ is proper. Since $V$ is admissible, $l$ defines an isomorphism from $(V')^\perp$ to $V^\perp$. Hence $f \circ \pi$ can be extended to a bounded Fredholm map $\tilde{f} : H' \to H$ such that $\Phi \tilde{f} = [f]$.
\end{proof}

    \section{The families Bauer-Furuta invariant}\label{Ch:FamiliesBF}

Let $B$ be a compact, connected smooth manifold. A 4-manifold family is a smooth, locally trivial, oriented fibre bundle $\pi : E \to B$ with each fibre diffeomorphic to a closed, oriented 4-manifold $X$. In particular, $E \to B$ has transition functions valued in $\Diff^+(X)$. For $b \in B$, denote the fibres of $E$ as $X_b = \pi\inv(b)$. 

Let $T(E/B) \to E$ be the vertical tangent bundle $T(E/B) = \ker \pi_*$, which is a 4-dimensional real vector bundle over $E$. Let $g$ be a metric on $T(E/B)$ with $\nabla$ the associated Levi-Civita connection. One can think of $g$ and $\nabla$ as smoothly varying families of metrics $\{g_b\}_{b \in B}$ and connections $\{\nabla_b\}_{b \in B}$ on the fibres $X_b$. Let $\sfrak_E$ be a \spinc structure on $T(E/B)$ with associated spinor bundles $W^\pm \to E$. This induces a smoothly varying family of \spinc structures $\{\sfrak_b\}_{b\in B}$ on the fibres of $E$. Let $\Lcal = \det(W^+)$ be the determinant line bundle of $W^+$, which is a family of $U(1)$-bundles over $B$. A $U(1)$-connection $2A$ on $\Lcal$ defines a family of \spinc connections $\nabla^A$ on $W^+$. 

Let $\Lambda^iT^*(E/B) \to E$ denote the $i$-th exterior power of $T^*(E/B)$. A section of $\Lambda^iT^*(E/B)$ is a family of $i$-forms on the fibres $X_b$. Write $\Omega_B^i(E) = \Cinf(E, \Lambda^iT^*(E/B))$ to denote the set of families of smooth $i$-forms, which has the structure of a vector bundle $\Omega_B^i(E) \to B$. Similarly, $\Cinf(E, W^+) \to B$ denotes the bundle of families of smooth spinors over $B$. We write $\Lambda^2_+T^*(E/B)$ to denote the bundle of self-dual 2-forms determined by the Hodge star.

\subsection{Families with separating necks}

Let $V_0 \to B$ be a rank 4 oriented Riemannian vector bundle equipped with a \spinc structure $\sfrak_{V_0}$. Denote by $S(V_0) \subset V_0$ the unit sphere sub-bundle of $V_0$. When performing a families connected sum, $S(V_0)$ will be obtained as the normal bundle of a section of the vertical tangent bundle of one of the summands. For any $L > 0$, Let $N_B(L)$ denote the family of cylinders
\begin{align*}
    N_B(L) &= S(V_0) \times [-L, L].
\end{align*}
We write $N_b(L)$ to denote the fibre of $N_B(L) \to B$ over $b \in B$. Denote the families of positive and negative fiberwise boundary components by
\begin{align*}
    \p N_B(L)^+ &= S(V_0) \times \{L\}\\
    \p N_B(L)^- &= S(V_0) \times \{-L\}.
\end{align*}
Since the transition maps of $V_0$ are valued in $SO(4)$, the vertical tangent bundle $T(S(V_0)/B)$ can be equipped with a metric $g_{S(V_0)}$ that restricts to the standard round metric on each fibre. Equip the vertical tangent bundle of $N_B(L)$ with the metric $g_{N_B(L)} = g_{S(V_0)} + dt^2$ which on each fibre is the product of the standard round metric on $S^3$ and the standard interval metric on $[-L, L]$. The \spinc structure $\sfrak_{V_0}$ determines a 3-dimensional \spinc structure on the vertical tangent space of $S(V_0)$. Pulling this back to $N_B(L)$ defines a \spinc structure $\sfrak_{N_B(L)}$ on $T(N_B(L)/B)$.
\begin{definition}
    Let $E \to B$ be a family of 4-manifolds with connected fibre $X$ and fix $L > 1$. A separating neck of length $2L$ on a $E \to B$ is an embedding $\iota : N_B(L) \to E$ covering the identity. It is required that the neck complement $M = \overline{E - \iota(N_B(L-1))}$ has fibres $M_b$ which decompose as 
    \begin{align*}
        M_b &= M_b^- \coprod M_b^+
    \end{align*}
    where $\p M_b^- = \iota(\p N_b(L-1)^-)$ and $\p M_b^+ = \iota(\p N_b(L-1)^+)$, both with reversed orientation. It is assumed that $E$ is given a metric and \spinc structure that extends $g_{N_B(L)}$ and $\sfrak_{N_B(L)}$.
\end{definition}
Given a 4-manifold family $E \to B$ with a separating neck of length $2L$, we identify $N_B(L)$ with its image $\iota(N_B(L))$. If $X$ has $n$ connected components, then a separating neck on $E$ is just a separating neck on each component. In this case, the neck is a disjoint union 
\begin{align*}
    N_B(L) = \coprod_{i=1}^n N_B(L)_i.
\end{align*}
Assume for convenience that $L > 2$. For each $1 \leq i \leq n$, define collar subbundles $C_i^\pm \subset N_B(L)_i$ by
\begin{align*}
    C_i^- &= S(V_0) \times [-L, -L+1]\\
    C_i^+ &= S^3(V_0) \times [L-1, L].
\end{align*}
Let $C = \coprod_i (C_i^- \cup C_i^+)$. Each fibre $C_b$ is a collar neighbourhood of the boundary of $N_b(L)$. Removing $N_B(L-1)$ from $E$ gives a family of manifolds $M_b$ with fibres $\overline{X_b - N_b(L-1)}$ and a natural inclusion $\iota : C \to M$. For any other neck length $L' > 2$, there is a natural isometric inclusion $C \to N(L')$ identifying $C$ has a collar neighbourhood of $\p N(L')$. Let $E(L') = M \cup_C N_B(L')$. That is, $E(L')$ is defined by the following pushout
\begin{equation*}
    \begin{tikzcd}
        C \arrow[hook]{r} \arrow[swap]{d}{\iota} & N_B(L') \arrow[dashed]{d} \\
        M \arrow[dashed]{r} & E(L').
    \end{tikzcd}
\end{equation*}
Let $\tau \in S_n$ be an even permutation on $n$ objects. Define a permuted inclusion map $\iota_\tau : C \to M$ such that $\iota_\tau|_{C_i^-} = \iota|_{C_i^-}$ and $\iota_\tau|_{C_i^+} = \iota|_{C_{\tau(i)}^+}$. That is, $C_i^-$ is mapped to $\iota(C_i^-)$ but $C_i^+$ is mapped to $\iota(C_{\tau(i)}^+)$. Define the permuted family $E^\tau$ by the following pushout
\begin{equation*}
    \begin{tikzcd}
        C \arrow[hook]{r} \arrow[swap]{d}{\iota_\tau} & N_B(L) \arrow[dashed]{d} \\
        M \arrow[dashed]{r} & E^\tau.
    \end{tikzcd}
\end{equation*}
Fiberwise, each boundary component of the form $\iota(C_i^-)_b \subset M_b$ has been connected by a cylinder $S^3 \times [-L, L]$ to $\iota(C_{\tau(i)}^+)_b$. We write $X^\tau$ to denote the standard fibre of $E^\tau$.

\subsection{The families Seiberg-Witten monopole map}\label{S:FSWMonoMap}

Fix a reference \spinc connection $A_0$ on $E$. Any other connection $A$ can be written as $A = A_0 + ia$ for some family of one-forms $a \in C^\infty(E, T^*(E/B))$. Let $n$ be the number of connected components of $X$ and fix an integer $k \geq 4$. The metric and orientation of $E$ determines an $L^2$-inner product of spinors and forms through integration. We write $L^2_k(E, -)$ to denote the $L^2_k$-Sobolev space of $k$-times weakly differentiable sections, with weak derivatives in $L^2$.

To define the families monopole map, we follow the construction in \cite[Example 2.1 and 2.4]{BaragliaKonnoBFandSW}. For now assume that $b_1(X) = 0$.  Define Hilbert space bundles $\Acal$ and $\Ccal$ over $B$ by 
\begin{align}\label{E:AcalCcalFamB1=0}
    \Acal &= L^2_k(E, W^+ \oplus T^*(E/B)) \oplus \R^n \nonumber \\
    \Ccal &= L^2_{k-1}(E, W^- \oplus \Lambda^2_+ T^*(E/B) \oplus \R).
\end{align}
The $\R^n$ term in $\Acal$ is identified with the space of locally constant functions $H^0(X ; \R)$ on $X$. Denote by $\bT^n = (S^1)^{\times n}$ the group of locally constant gauge transformations. Let $\bT^n$ act on $\Acal$ and $\Ccal$ in the usual manner, on spinors by multiplication and on forms trivially. This action is fibre-preserving and orthogonal. The monopole map $\mu : \Acal \to \Ccal$ is the $\bT^n$-equivariant bundle map given by the formula
\begin{align*}
    \mu(\psi, a, f) &= (D_{A_0 + ia}\psi, -iF_{A_0 + ia}^+ + i\sigma(\psi), d^*a + f).
\end{align*}
The map $\sigma$ is defined by the equation $\sigma(\psi) = \psi \otimes \psi^* - \frac{1}{2}\text{Id}$ where the traceless, Hermitian endomorphism $\sigma(\psi)$ is identified as an imaginary valued self-dual 2-form. A solution $(\psi, a, f) \in \mu\inv(0)$ must have $f = 0$, hence we will suppress the third component. This solution corresponds to the Seiberg-Witten monopole $(\psi, A_0 + a)$. The gauge fixing condition $d^* a = 0$ determines the with gauge class of $(\psi, A_0 + a)$ up to a harmonic gauge transformation. Since $b_1(X) = 0$, the only harmonic gauge transformations are the locally constant ones.

There is a decomposition $\mu = l + c$ with
\begin{align}\label{E:MiniMuFredholmFamilies}
    l(\psi, a, f) &= (D_{A_0}\psi, d^+a, d^* a + f) \nonumber \\
    c(\psi, a, f) &= (ia \cdot \psi, -iF_{A_0}^+ + i\sigma(\psi), 0).
\end{align}
The map $l$ is linear Fredholm and $c$ is compact, hence $\mu$ is a Fredholm map. There is a somewhat standard argument (e.g \cite[Proposition 3.1]{BF1}) in ordinary Seiberg-Witten theory that shows that $\mu$ is a bounded Fredholm map when $B$ is a point. Assuming that $B$ is compact means that this argument can be extended fibrewise.

In the case that $b_1(X) > 0$, it will be necessary to assume that a smooth section $x : B \to E$ exists. In general the families Bauer-Furuta invariant will depend on the homotopy class of $x$. Let $\Hcal^1(\R) \subset \Cinf(E, T^*(E/B))$ denote the subbundle of real harmonic forms. That is, $\Hcal^1(\R) \to B$ is a vector bundle with fibre $\Hcal^1(X_b ; \R)$ over $b \in B$. Now pull back the bundles defined in (\ref{E:AcalCcalFamB1=0}) to bundles over $\Hcal^1(\R)$:
\begin{align*}
    \tAcal &= L^2_k(E, W^+ \oplus T^*(E/B)) \oplus \R^n \to \Hcal^1(\R) \\
    \tCcal &= L^2_{k-1}(E, W^- \oplus \Lambda^2_+ T^*(E/B) \oplus \R) \oplus \Hcal^1(\R) \to \Hcal^1(\R).
\end{align*}
The tilde notation is used because we are yet to quotient out by harmonic gauge transformations. Let $A_\theta = A_0 + i\theta$ denote the connection associated to $\theta \in \Hcal(\R)$. Note that since $\theta$ is harmonic, $F_{A_\theta} = F_{A_0}$. Define $\tmu : \tAcal \to \tCcal$ by
\begin{align}\label{E:MuDefFamilies}
    \tmu_\theta(\psi, a, f) &= (D_{A_\theta + ia}\psi, -iF_{A_0 + ia} + i\sigma(\psi), d^*a + f, \pr(a)).
\end{align}
This is the monopole map with gauge fixing, before dividing out by the harmonic gauge transformations. The bundle map $\pr : L^2_k(E, T^*(E/B)) \to \Hcal^1(\R)$ is defined as follows. Let $\{U_\beta\} \subset B$ be a trivialising open cover of $B$ with $E|_{U_\beta} \cong U_\beta \times X$. Choose cycles $\alpha^1, ..., \alpha^{b_1(X)}$ that restrict to a homology basis on each fibre of $E|_{U_\beta}$. Define a map $\pr_\beta : \Omega^1_B(E)|_{U_\beta} \to \Hcal^1(\R)|_{U_\beta}$ on each fibre above $b \in U_\beta$ by 
\begin{align}\label{E:ProjectionMap}
    (\pr(a)_b)(\alpha^i_b) &= \int_{\alpha^i_b} a_b.
\end{align}
Extend $\pr(a)_b$ linearly so that $\pr(a)_b \in \Hom(H_1(X_b), \R) = H^1(X_b ; \R)$. Now let $\{\rho_\beta\}$ be a partition of unity subordinate to $\{U_\beta\}$ and define $\pr : \Omega_B(E) \to \Hcal^1(\R)$ by $\pr = \sum_\beta \rho_\beta \pr_\beta$. This map has the property that if $a \in \Omega^1_B(E)$ is a family of closed one forms, then $\pr(a) \in \Hcal^1(\R)$ is the cohomology class of $a$ in each fibre. This extends continuously to a map $\pr : L^2_k(E, T^*(E/B)) \to \Hcal^1(\R)$.

To account for the harmonic gauge transformations, let $\Hcal(2\pi\Z) \to B$ be the bundle of groups over $B$ with fibre $H^1(X_b ; 2\pi\Z)$. For each $\omega \in \Hcal(2\pi\Z)$ and $b \in B$, define a map $g_{\omega, b} : X_b \to S^1$ by 
\begin{align*}
    g_{\omega, b}(y) = \exp\left(i \int_{x(b)}^y \omega\right).
\end{align*}
This map is well defined since the periods of $\omega$ are multiples of $2\pi$. Further, $g_{\omega, b}$ is the unique harmonic gauge transformation with the property that $g_{\omega, b}\inv dg_{\omega,b} = i\omega$ and $g_{\omega, b}(x(b)) = 1$. The gauge transformation $g_\omega$ acts on a connection $A$ by $g_\omega \cdot A = A + i\omega$.

Let the bundle of groups $\Hcal(2\pi\Z)$ act on $\Hcal(\R)$ fiberwise by $\omega \cdot \theta = \omega + \theta$. The quotient bundle $\Jcal = \Hcal(\R)/\Hcal(2\pi\Z)$ is the $b_1(X)$-dimensional Jacobian torus bundle over $B$. That is, each fibre $\Jcal_b$ is the Jacobian torus $\Jcal(X_b) = H(X_b; \R)/H(X_b ; 2\pi \Z)$ of $X_b$. Define an action of $\Hcal(2\pi\Z)$ on elements $(\psi, a, f) \in \tAcal_\theta$ and $(\phi, \eta, g, \alpha) \in \tCcal_\theta$ by
\begin{align*}
    \omega \cdot (\theta, (\psi, a, f)) &= (\theta + \omega, (g_\omega\inv \psi, a, f)) \\
    \omega \cdot (\theta, (\phi, \eta, g, \alpha)) &= (\theta + \omega, (g_\omega\inv\phi, \eta, g, \alpha)). 
\end{align*}
This is the free action of the based harmonic gauge transformations $g_\omega$. Under this action, $\tmu$ is equivariant. The fiberwise quotients $\Acal = \tAcal/\Hcal(2\pi\Z)$ and $\Ccal = \tCcal/\Hcal(2\pi\Z)$ are Hilbert bundles over $\Jcal$ with a residual $\bT^n$-action of the constant gauge transformations. The map $\tmu$ descends to a $\bT^n$-equivariant Fredholm map $\mu : \Acal \to \Ccal$ over $\Jcal$. This is the families monopole map in the setting $b_1(X) > 0$. In a similar fashion to (\ref{E:MiniMuFredholmFamilies}), $\mu = l + c$ is a bounded Fredholm map with 
\begin{align}\label{E:MuFredholmFamilies}
    l_\theta(\psi, a, f) &= (D_{A_\theta}\psi, d^+a, d^* a + f, \pr(a)) \nonumber \\
    c_\theta(\psi, a, f) &= (ia \cdot \psi, -iF_{A_0}^+ + i\sigma(\psi), 0, 0).
\end{align}
Define a $\bT^n$ universe $\Ucal$ by 
\begin{align}\label{E:FamUcalDef}
    \Ucal = L^2_{k-1}(X, W|_{X}^- \oplus \Lambda^2_+(T^*X) \oplus \R) \oplus H^1(X ; \R).
\end{align}
This universe can be identified with each fibre of $\Ccal$. The map $l$ defines a family of linear Fredholm maps over $\Jcal$, so let $\ind_\Jcal l$ denote the corresponding virtual index bundle. Let $H^+ \to \Jcal$ denote the rank $b^+(X)$ trivial bundle with fibre $H^2_+(X; \R)$ so that the relation $\ind_\Jcal l = \ind_\Jcal D - H^+$ holds.
\begin{definition}
    The families Bauer-Furuta invariant of a 4-manifold family $E \to B$ is the cohomotopy class 
    \begin{align}
        [\mu] &\in \pi^0_{\bT^n, \Ucal}(\Jcal, \ind_\Jcal l) \nonumber\\
        &= \pi^{b^+}_{\bT^n, \Ucal}(\Jcal, \ind_\Jcal D). \label{E:FamiliesBFInvariant}
    \end{align}
\end{definition}
Now suppose that $E$ is a family of 4-manifolds $X(L)$ with necks of length $2L$. To construct an appropriate reference connection, let $\{\rho_\beta\}$ be a partition of unity subordinate to a trivialising open cover $\{U_\beta\}$ of $B$. Let $A_0^\beta$ be a flat connection on $N_{U_\beta}(L)$ that is identical on each neck component $N_{U_\beta}(L)_i = U_\beta \times (S^3 \times [-L,L])$. Such a connection exists since $H^2(S^3 \times [-L,L] ; \R) = 0$. Extend $A_0^\beta$ to $E|_{U_\beta}$ and set $A_0 = \sum_\beta \rho_\beta A^\beta_0$. Then $A_0$ defines a connection on both $X$ and $X^\tau$ which is flat on the neck.

Moreover, let $\Gcal_{N(1)} \to B$ be the bundle of Gauge groups with fibre maps $(\Gcal_{N(1)})_b \subset \Cinf(X_b, S^1)$ that fix the short neck $N(1)_b$. Let $\ker d_{N(1)} \subset \Omega^1_B(E)$ be the subset of families of forms $a \in \ker d$ that vanish on $N(1)$. The inclusion $(A_0 + i\ker d_{N(1)})/\Gcal_{N(1)} \to \Jcal_E$ is a smooth bundle map over $B$ that restricts to a diffeomorphism on each fibre, hence we can identify $(A_0 + i\ker d_{N(1)})/\Gcal_{N(1)} = \Jcal_E$. Now for any even permutation $\tau$, $\Jcal_E = \Jcal_{E^\tau}$ which means that $\mu_E$ and $\mu_{E^\tau}$ can be treated as bundle maps over the same space $\Jcal = \Jcal_E = \Jcal_{E^\tau}$.

Denote by $\widehat{W}^+ \to S(V_0) \times [-L, L]$ the restriction of $W^+ \to N_B(L)$ to one of the connected components of $N_B(L)$. Define $F = \oplus_{i=1}^n \widehat{W}^+$ to be the direct sum of $n$-copies of $\widehat{W}^+$ over $S(V_0) \times [-L, L]$. Since $N_B(L)$ has $n$ connected components, a section $\psi : N_B(L) \to W^+$ can be identified with a vector of sections 
\begin{align}\label{E:vecpsi}
    \vecpsi : S(V_0) \times [-L, L] \to F.
\end{align}
That is, the restriction $\psi_i$ to the $i$th component of $N_B(L)$ is identified with the $i$th component of $\vecpsi$. Let $T : S(V_0) \times [-L, L] \to SO(n)$ denote a matrix valued function. For a section $\psi : N_B(L) \to W$ along $N_B(L)$, define an action by $T \cdot \psi = T\vecpsi$ where $T$ acts pointwise on $\vecpsi$ and $T\vecpsi$ is identified with a section of $W^+ \to N_B(L)$. The same process defines an action on forms along the neck $a : N_B(L) \to \Lambda^i(T^*(N_B(L)/B))$.

Let $\gamma : [0,1] \to SO(n)$ be a smooth path from the identity to $\tau$, which exists under the assumption that $\tau$ is even. Let $\vphi : [-L, L] \to [0,1]$ be a smooth map that vanishes on $[-L, 1]$ and is identically equal to 1 on $[1, L]$. Define a matrix valued function $V : S(V_0) \times [-L, L] \to SO(n)$ by
\begin{align}\label{E:VDefinitionFamilies}
    V(x,t) &= \gamma(\vphi(t)).
\end{align}
Note that $V$ is constant along the $S(V_0)$ factor. Let $(\psi, a) : N_B(L) \to W^+ \oplus T^*(N_B(L)/B)$ be a spinor-form pair along $N_B(L)$ and define $(\psi, a)^\tau = (V\cdot \psi, V \cdot a)$ by the action described above. The pair $(\psi, a)^\tau$ has the property that $(\psi, a)^\tau_i = (\psi, a)_i$ on $C^-$ and $(\psi, a)^\tau_i = (\psi, a)_{\tau(i)}$ on $C^+$. Now given a section $(\psi, a) : E \to W^+ \oplus T^*(E/B)$ defined on all of $E$, this permutation process defines a section $(\psi, a)^\tau$ on $E^\tau$ with the property that $(\psi, a)$ and $(\psi, a)^\tau$ agree outside of $N_B(1)$. 

This construction defines an isomorphism $V_\Acal : \Acal_E \to \Acal_{E^\tau}$ of Hilbert bundles over $\Jcal$. Similarly for $\Ccal$, the action of $V$ defines a map $V_\Ccal : \Ccal_E \to \Ccal_{E^\tau}$ that on the $\Hcal^1(X ; \R)$ factor is just the identity. Thus $V_\Acal$ and $V_\Ccal$ identify $\pi_0(\Pcal_l(\Acal, \Ccal)^{\bT^n})$ and $\pi_0(\Pcal_l(\Acal^\tau, \Ccal^\tau)^{\bT^n})$ by the map $[f] \mapsto [V_\Ccal f V_\Acal\inv]$. Moving forward we will suppress the subscripts. Since all the permutation occurs in $N_B(1)$, there is a constant $C_V$ independent of $L$ such that 
\begin{align}\label{E:CVDef}
    \|V(\psi, a)\|_{L^2_k} \leq C_V \|(\psi, a)\|_{L^2_k}.
\end{align}
\begin{theorem}[Families Permutation Theorem]\label{T:FamiliesGluingTheorem}
    Let $E \to B$ be a family of closed 4-manifolds that admits an $n$-component separating neck. Let $\tau \in S_n$ be an even permutation with $E^\tau$ the corresponding permuted family. Then 
    \begin{align}
        [\mu_E] = [\mu_{E^\tau}]
    \end{align}
    as elements of $\pi^{b^+}_{\bT^n, \Ucal}(\Jcal, \ind D)$.
\end{theorem}
\begin{remark}
    In the construction of $E^\tau$ it is assumed that $\tau$ is an even permutation, however Remark \ref{R:tauOdd} explains how this assumption is unnecessary.
\end{remark}

In \cite{BF2}, Bauer gave a proof of Theorem \ref{T:FamiliesGluingTheorem} in the unparameterised case where $B$ is a single point. While the ideas used in the proof of his formula were sound, we were not able to reproduce some of his arguments and have deemed the proof to be incomplete. Instead, we revisit his ideas to formulate a new proof that extends to the families setting.

\section{Monopoles on the neck}\label{S:MonopolesOnTheNeck}

To prove the permutation theorem it is enough to show that $\mu_E$ is homotopic to $V\inv \mu_{E^\tau} V$ through compact perturbations of $l$ (see Corollary \ref{C:PhiBijection}). Such a homotopy is constructed in three stages, and at each stage it is important to check that the boundedness conditions outlined in Definition \ref{D:CompactHomotopy} are satisfied. This is accomplished using techniques from the theory of Sobolev spaces, elliptic operators and monopoles on a cylinder with a varying neck length. 

\subsection{Sobolev estimates}

Two fundamental theorems in the theory of Sobolev spaces are the Sobolev embedding theorem \cite{SobolevSpaces} and the Sobolev multiplication theorem \cite{SobMult}. These theorems give estimates that relate different Sobolev norms on a spinor-form pair $(\psi, a)$ on $X$. For two neck lengths $L_1$ and $L_2$, we require estimates that apply to spinor-form pairs on both $X(L_1)$ and $X(L_2)$. The following results achieve this goal in the situations necessary for Theorem \ref{T:FamiliesGluingTheorem}. 

\begin{lemma}[\cite{BF2} Proposition 3.1]\label{L:SENeckIndependence}
    Let $k$ and $p$ be non-negative integers such that $k - \frac{4}{p} > 0$. There is a constant $C_S$ such that, for any $L \geq 2$,
    \begin{align*}
        |(\psi, a)|_{C^0} \leq C_S \|(\psi, a)\|_{L^p_k}
    \end{align*}
    for any $L^p_k$-pair $(\psi, a)$ on $X(L)$.
\end{lemma}
\begin{proof}
    Fix a neck length $L \geq 2$. For each $x \in X$ let $\delta_x : X \to [0,1]$ be a smooth bump function in a small neighbourhood of $x$. Let $X_0 = X(2)$ and use the Sobolev embedding $L^p_k(X_0, W^+ \oplus T^*X_0) \subset C^0(X_0, W^+ \oplus T^*X_0)$ \cite[Theorem B.2]{SobMult} to choose a constant $C_1$ with 
    \begin{align*}
        |(\psi', a')|_{C^0(X_0)} \leq C_1 \|(\psi', a')\|_{L^p_k(X_0)}
    \end{align*}
    for any $L^p_k$-pair $(\psi', a')$ on $X_0$. Note that such a constant exists since $k - \frac{4}{p} > 0$. For any spinor $\psi$ on $X$, $\delta_x\psi$ can be identified as a spinor on $X_0$. The same is true for $\delta_x a$ for a one-form $a$ on $X$. Now for each $x \in X$,
    \begin{align*}
        |(\delta_x\psi, \delta_x a)|_{C^0(X)} \leq C_1\|(\delta_x \psi, \delta_x a)\|_{L^p_k(X)}.
    \end{align*} 
    Since $\delta_x$ is smooth and defined locally, it has bounded $C_k$ norm which is independent of $L$. Thus there exists a constant $C_2$ such that for all $x \in X$, 
    \begin{align*}
        \|(\delta_x \psi, \delta_x a)\|_{L^p_k(X)} \leq C_2\|(\psi, a)\|_{L^p_k(X)}.
    \end{align*}
    It follows that 
    \begin{align*}
        |(\psi, a)|_{C^0(X)} &= \sup_{x \in X}|(\delta_x \psi, \delta_x a)|_{C^0(X)}\\
        &\leq C_1 \sup_{x \in X} \|(\delta_x \psi, \delta_x a)\|_{L^p_k(X)}\\
        &\leq C_1 C_2 \|(\psi, a)\|_{L^p_k(X)}.
    \end{align*}
    Setting $C_S = C_1 C_2$ gives the result.
\end{proof}
The next lemma demonstrates that Sobolev multiplication bounds only depend linearly on the length of the neck.
\begin{lemma}\label{L:SMNeckIndependence}
    Let $k \geq 0$ and $p \geq 1$ be integers. There is a constant $C_{SM}$ such that, for any neck length $L \geq 2$,
    \begin{align*}
        \|a \cdot \psi\|_{L^p_k} \leq C_{SM}L\|a\|_{L^{2p}_k}\|\psi\|_{L^{2p}_k}
    \end{align*}
    for any $L^{2p}_k$-pair $(\psi, a)$ on $X(L)$.
\end{lemma}
\begin{proof}
    For notational simplicity, assume that $X$ is connected. Recall that $M^\pm$ denotes the two halves of $M = \overline{X - N(L-1)}$ with tubular ends of the form 
    \begin{align*}
        N(L)^- \cap M &= S^3 \times [-L, -L + 1]\\
        N(L)^+ \cap M &= S^3 \times [L - 1, L].
    \end{align*}
    We will cut $N(L)$ into pieces that can be identified on $X_0 = X(2)$, then use Sobolev multiplication on $X_0$. Let $\phi : X \to [0,1]$ be a smooth function such that $\phi \equiv 1$ on $X - N(L-\frac{5}{4})$ and $\phi \equiv 0$ on $N(L-2)$. Define a function $\chi : \R \to [0,1]$ such that $\chi \equiv 1$ on $[0,1]$ and $\chi \equiv 0$ outside $[-\frac{1}{4}, \frac{5}{4}]$. Let $\chi_i$ be $\chi$ shifted by $i$ so that $\chi_i \equiv 1$ on $[i,i+1]$ and $\chi_i \equiv 0$ outside $[i - \frac{1}{4},i + \frac{5}{4}]$. Let $m = \floor{L- \frac{5}{4}}$. For $i$ an integer with $-(m+1) \leq i \leq m$, extend $\chi_i$ to $N(L) = S^3 \times [-L, L]$ by projection onto the interval factor. Let 
    \begin{align*}
        \vphi &= \sqrt{\phi^2 + \sum_{i = -(m + 1)}^m \chi_i^2}.
    \end{align*}
    Notice that $\vphi$ is positive on $X$. Let $\vphi_i = \frac{\chi_i}{\vphi}$ for $-(m+1) \leq i \leq m$ with $\vphi_{m+1} = \frac{\phi}{\vphi}$. By construction,
    \begin{align*}
        \sum_{i=-(m+1)}^{m+1} \vphi_i^2 = 1.
    \end{align*}
    For each $-(m+1) \leq i \leq m+1$, set $\psi_i = \vphi_i\psi$ and $a_i = \vphi_i a$. Both $\psi_i$ and $a_i$ can be identified as sections on $X_0$. For $-(m+1) \leq i \leq m$, this is accomplished by shifting the interval $[i-\frac{1}{4}, i + \frac{5}{4}]$ to $[-\frac{1}{4}, \frac{5}{4}]$. We can assume that the $C^k$ norm of $\vphi_i$ is bounded, which implies that there exists a constant $C_1$, independent of $L$, such that
    \begin{align}\label{E:SMLemma1}
        \|\psi_i\|_{L^{2p}_k(X_0)} &\leq C_1\|\psi\|_{L^{2p}_k(X)} \nonumber \\
        \|a_i\|_{L^{2p}_k(X_0)} &\leq C_1\|a\|_{L^{2p}_k(X)}.
    \end{align} 
    For the purposes of elliptic bootstrapping, the $L^p_k$-Sobolev norm on $X_0$ is defined as 
    \begin{align*}
        \|(\psi_i, a_i)\|_{L^p_k(X_0)} &= \sum_{j=0}^k \|(\Dcal^j \psi_i, (d^* + d^+)^j a_i)\|_{L^p(X_0)}.
    \end{align*}
    Equivalently, the $L^p_k$-norm on $X_0$ can instead be defined by differentiating spinors with the \spinc connection $\nabla_{A_0}$ and forms with the Levi-Civita connection $\nabla$. Thus there are constants $0 < c \leq C$ such that 
    \begin{align*}
        c\|(\psi_i, a_i)\|_{L^p_k(X_0)} \leq \sum_{j=0}^k \|(\nabla_{A_0}^j \psi_i, \nabla^j a_i)\|_{L^p(X_0)} \leq C\|(\psi_i, a_i)\|_{L^p_k(X_0)}.
    \end{align*}
    Calculating with repeated applications of the Leibniz rule gives 
    \begin{align*}
        \|a_i \cdot \psi_i\|_{L^p_k(X_0)} &\leq \frac{1}{c}\sum_{j=0}^k \|\nabla_A^j(a_i \cdot \psi_i)\|_{L^p(X_0)}\\
        &\leq \frac{1}{c}\sum_{j=0}^k \sum_{l=0}^j K_{j,l}\|\Gamma(\nabla^l a_i)\cdot (\nabla_A^{j-l} \psi_i)\|_{L^p(X_0)}
    \end{align*}
    for some non-negative constants $K_{j,l}$. Here $\Gamma(\nabla^l a_i) \in \End(W)$ is the matrix corresponding to spinor multiplication by the $(l+1)$-form $\nabla^l a_i$. The operator norm of $\Gamma(\nabla^l a_i)$ is equal to $|\nabla^l a_i|$, hence applying Sobolev multiplication \cite[Lemma B.3]{SobMult} it follows that  
    \begin{align}\label{E:SMLemma2}
        \|a_i \cdot \psi_i\|_{L^p_k(X_0)} &\leq C_2\|a_i\|_{L^{2p}_k(X_0)} \|\psi_i\|_{L^{2p}_k(X_0)}
    \end{align}
    for some constant $C_2$. This constant depends on $c$, $K_{j,l}$ and Sobolev multiplication on $X_0$, hence is independent of $L$. Combining (\ref{E:SMLemma1}) and (\ref{E:SMLemma2}) produces the result.
    \begin{align*}
        \|a \cdot \psi\|_{L^p_k(X)} &\leq \sum_{i=-m-1}^{m+1} \|a_i \cdot \psi_i\|_{L^p_k(X_0)}\\
        &\leq C_2\sum_{i=-m-1}^{m+1} \|a_i\|_{L^{2p}_k(X_0)} \|\psi_i\|_{L^{2p}_k(X_0)}\\
        &\leq C_2C_1^2\sum_{i=-m-1}^{m+1} \|a\|_{L^{2p}_k(X)} \|\psi\|_{L^{2p}_k(X)}\\
        &\leq C_{SM}L\|a\|_{L^{2p}_k(X)}\|\psi\|_{L^{2p}_k(X)}.
    \end{align*}
\end{proof}
The same argument applied to $\sigma(\psi)$ instead gives the following result.
\begin{lemma}\label{L:SigmaPsiSobolevBound}
        Let $k \geq 0$ and $p \geq 1$ be integers. There is a constant $C_{\sigma}$ such that, for any neck length $L \geq 2$,
    \begin{align*}
        \|\sigma(\psi)\|_{L^p_{k}} \leq C_\sigma L\|\psi\|_{L^{2p}_{k}}^2
    \end{align*}
    for and $\psi \in L^{2p}_k(X(L), W^+)$.
\end{lemma}

\subsection{Elliptic inequality}

To analyse the properties of monopoles on a neck of varying length, it is useful to apply Yang Mills theory on cylinders as in Chapter 2 of \cite{FloerYM}. Fix a neck length $L$ with $X = X(L)$. For notational simplicity, assume that $X$ only has one connected component. Recall that $M^+$ and $M^-$ are the two halves of $M = \overline{X - N(L-1)}$. Attach infinite tubes to $M^+$ and $M^-$ to get manifolds with tubular ends $Y^\pm$ of the form 
\begin{align*}
    Y^- &= M^- \cup S^3 \times [-L + 1, \infty)\\
    Y^+ &= S^3 \times (-\infty, L-1] \cup M^+.
\end{align*}
One-forms on the tubular component of $Y^\pm$ can be analysed by studying forms on the product $S^3 \times \R$. Let $\pi : S^3 \times \R \to S^3$ be projection onto the $S^3$ factor. All elements of $\Omega^1(S^3 \times \R)$ are of the form $\omega_t + fdt$ for $\omega_t \in \Omega^1(S^3)$ a smooth family of one-forms on $S^3$ and $f : S^3 \times \R \to \R$ a smooth function. That is, we can identify
\begin{align*}
    \Omega^1(S^3 \times \R) &= \Cinf(S^3 \times \R, \R \oplus \pi^* T^* S^3).
\end{align*}
Similarly, self-dual 2-forms $\Omega_+^2(S^3 \times \R)$ can be identified with time-dependent 1-forms $\xi \in C^\infty(S^3 \times \R, \pi^* T^* S^3)$ by the isomorphism
\begin{align*}
    \xi \mapsto \xi \wedge dt + *_3 \xi.
\end{align*}
Here $*_3$ is the hodge star operator on $S^3$. Thus we can interpret the elliptic operator $d^* + d^+ : \Omega^1(S^3 \times \R) \to \Omega^0(S^3 \times \R) \oplus \Omega^2_+(S^3 \times \R)$ as 
\begin{align}\label{E:FormIdentification}
    d^* + d^+ : C^\infty(S^3 \times \R, \R \oplus \pi^* T^* S^3) \to C^\infty(S^3 \times \R, \R \oplus \pi^* T^* S^3).
\end{align}
Consider the operator $\Lcal : \Omega^0(S^3) \oplus \Omega^1(S^3) \to \Omega^0(S^3) \oplus \Omega^1(S^3)$ defined by 
\begin{align}\label{E:Lcal}
    \Lcal &= \begin{pmatrix}
        0 & d^*\\
        d & *d
    \end{pmatrix}.
\end{align}
This is a self-adjoint elliptic operator that squares to the Laplacian $\Lcal^2 = dd^* + d^*d$ on $\Omega^0(S^3) \oplus \Omega^1(S^3)$. It can be shown by direct calculation that under the identification (\ref{E:FormIdentification}), 
\begin{align*}
    d^* + d^+ = \frac{\p }{\p t} + \Lcal
\end{align*}
where $\frac{\p}{\p t}$ is the derivative in the $\R$ direction. 

Since the tubular ends of $Y$ are not compact, solutions to the operator $\frac{\p}{\p t} + \Lcal$ will be studied in weighted Sobolev spaces. Weighted Sobolev spaces consist of functions that have a controlled exponential increase towards the tubular ends. To define them, fix a parameter $\alpha < 0$ and let $f^-_\alpha$ be a smooth function on $S^3 \times [-L, \infty)$ that is zero on $S^3 \times [-L, -L + 2]$ and decreases with slope $\alpha$ on $S^3 \times [-L + 3, \infty)$. Similarly, define $f^+_\alpha$ on $S^3 \times (-\infty, L]$ to be zero on $S^3 \times [L-2, L]$ and decrease with slope $\alpha$ on $S^3 \times (-\infty, L-3]$. Since $\alpha < 0$, both functions $f^\pm_\alpha$ are non-positive. Define the weighted Sobolev space $L^{p,\alpha}_k(Y^\pm)$ to be the completion of $L^p(Y^\pm)$ with respect to the norm 
\begin{align*}
    \|g\|_{L^{p,\alpha}_k} &= \|\exp(f^\pm_\alpha) g\|_{L^p_k}.
\end{align*}
Note that $\exp(f^\pm_\alpha)$ is decreasing exponentially towards the infinite end of $Y^\pm$. Moreover, the spaces $L^{p,\alpha}_k(Y^\pm)$ are independent of the original neck length $L$.

It is shown in \cite{FloerYM} that $d^* + d^+ = \frac{\p}{\p t} + \Lcal$ is a linear Fredholm operator on $L^{p,\alpha}_1$-forms if $\alpha$ is not in the spectrum of $L$. Since $L$ is self-adjoint and elliptic it has discrete spectrum away from infinity, so choose $\alpha < 0$ to be greater than the maximal negative eigenvalue of $L$. As in (\ref{E:ProjectionMap}), define a harmonic projection map $\pr^\pm : \Omega^1(Y^\pm) \to \Omega^1(Y^\pm)$ by integrating a homology basis of curves away from the neck. The image of $\pr^\pm$ is $H^1(X ; \R)$, identified as the space of harmonic forms $\Hcal^1(Y^\pm) \subset \Omega^1(Y^\pm)$. Fix $p > 4$ so that $L^{p}_1(Y^\pm, T^*Y) \subset C^0(Y^\pm, T^*Y)$ by Sobolev embedding and extend $\pr^\pm$ continuously to a map on $L^{p, \alpha}_1$ forms. The operator 
\begin{align*}
    d^* + d^+ : L^{p,\alpha}_1(Y^\pm, T^*Y^\pm) \to L^{p,\alpha}(Y^\pm, \R \oplus \Lambda^2_+ T^* Y^\pm)
\end{align*}
is Fredholm with kernel $\Hcal^1(Y^\pm)$ and cokernel $H^0(Y^\pm ; \R) \oplus H^2_+(Y^\pm ; \R)$. Let $H^\pm = \ker \pr^\pm$, which is a complement of $\ker(d^* + d^+)$. Thus the restriction of $d^* + d^+$ to $H^\pm$ is a linear bijection onto the closed $L^{p,\alpha}$-image of $d^* + d^+$. The bounded inverse theorem guarantees that there are constants $C^\pm > 0$ such that for $b \in L^{p,\alpha}_1(Y^\pm, T^* Y^\pm)$,
\begin{align}\label{E:EllipticIneqAlpha}
    \|b\|_{L^{p,\alpha}_1} \leq C^\pm\left(\|(d^* + d^+)b\|_{L^{p,\alpha}} + \|\pr^\pm(b)\|\right).
\end{align}
Importantly, the constants $C^\pm$ are independent of the neck length $L$. That is, for another choice of neck length $L'$ and manifolds with tubular ends $(Y')^\pm$, there is an isometry from $L^{p,\alpha}_k(Y^\pm, T^* Y^\pm)$ to $L^{p,\alpha}_k((Y')^\pm, T^*(Y')^\pm)$ defined by shifting the interval component by $L' - L$.

To analyse the behaviour of forms away from the middle of the neck, define smooth cut-off functions $\beta^\pm : X \to [0,1]$ which vanish on $X^\mp \cup N(2)$ and are equal to $1$ on $M^\pm$. To ensure such $\beta$ exist, we will assume that $L \geq 3$. 
\begin{lemma}[\cite{BF2} Proposition 3.1]\label{L:PoUBound}
    Let $\beta^\pm$ be cutaway functions as described above and fix $p > 4$. There exists a constant $C$ such that, for any neck-length $L > 3$,
    \begin{align*}
        |a|_{C^0(M)} \leq C\left(\|(d^* + d^+)\beta^+ a\|_{L^p(X)} + \|(d^* + d^+)\beta^- a\|_{L^p(X)} + \|\pr(a)\|\right)
    \end{align*}
    for any $L^p_1$-form $a$ on $X(L)$. 
\end{lemma}
\begin{proof}
    The Sobolev embedding $L^p_1(X, T^*X) \subset C^0(X, T^*X)$ guarantees the existence of a constant $C_S$ such that
\begin{align}\label{E:SobBound}
    |a|_{C^0(X)} \leq C_S \|a\|_{L^p_1(X)}.
\end{align}
for all $a \in L^p_1(X, T^*X)$. Lemma \ref{L:SENeckIndependence} ensures that $C_S$ can be chosen independently of $L$. To apply the elliptic bound, let $b_\pm = \beta^\pm a$ and notice that $e^{f^\pm_\alpha}b_\pm = a$ on $M^\pm$.
\begin{align}\label{E:PoUHalfBound}
    |a|_{C^0(M^\pm)} &= |e^{f^\pm_\alpha}b_\pm|_{C^0(M^\pm)} \nonumber \\
    &\leq |e^{f^\pm_\alpha}b_\pm|_{C^0(Y^\pm)} \nonumber \\
    &\leq C_S\|e^{f^\pm_\alpha}b_\pm\|_{L^{p}_1(Y^\pm)} \nonumber \\
    &= C_S\|b_\pm\|_{L^{p,\alpha}_1(Y^\pm)}
\end{align}
Note that $b_\pm$ is compactly supported in $M^\pm \cup N(L-1) \subset Y^\pm$, so the Sobolev bound (\ref{E:SobBound}) applies to $|e^{f^\pm_\alpha}b_\pm|_{C^0(Y^\pm)}$. Now (\ref{E:EllipticIneqAlpha}) gives
\begin{align*}
    |a|_{C^0(M^\pm)} &\leq C_SC^\pm\left(\|(d^* + d^+)b_\pm\|_{L^{p,\alpha}(Y^\pm)} + \|\pr(b_\pm)\|\right) \\
    &\leq C_SC^\pm\left(\|(d^* + d^+)b_\pm\|_{L^{p}(X)} + \|\pr(b_\pm)\|\right).
\end{align*}
This inequality follows since $f_\alpha^\pm \leq 0$ and $b^\pm$ is compactly supported on $M^\pm \cup N(L-1) \subset Y^\pm$. Putting this together with $C = \max\{C^sC^+, C^sC^-\}$ yields
\begin{align*}
    |a|_{C^0(M)} &\leq |a|_{C^0(M^+)} + |a|_{C^0(M^-)} \\
    &\leq C\left(\|(d^* + d^+)b_+\|_{L^{p}(X)} + \|(d^* + d^+)b_-\|_{L^{p}(X)} + (\|\pr(b_+)\| + \|\pr(b_-)\|)\right).
\end{align*}
Recall that $\pr(b_\pm)$ is defined by integration over an orthonormal basis of curves contained in $M$. Since $b_\pm$ vanishes on $M^\mp$ we have
\begin{align*}
    \|\pr(b_+)\| + \|\pr(b_-)\| &= \|\pr(b_+ + b_-)\| \\
    &= \|\pr(a)\|.
\end{align*}
It follows that 
\begin{align*}
    |a|_{C^0(M)} &\leq C\left(\|(d^* + d^+)b_+\|_{L^{p}(X)} + \|(d^* + d^+)b_-\|_{L^{p}(X)} + \|\pr(a)\|\right).
\end{align*}
\end{proof}
\begin{proposition}[Adapted from \cite{BF2} Lemma 3.3]\label{P:HarmonicNeckBound}
    Fix $p > 4$. There exists a neck length $L_0$ and a constant $C_E$ such that the following holds: For any $L \geq L_0$, let $a \in L^p_1(X, T^* X)$ be an $L^p_1$-form on $X(L)$ such that $\pr(a) = 0$. If $(d^* + d^+)a$ vanishes on $N(L-1)$, then
    \begin{align*}
        |a|_{C^0(M)} \leq C_E|(d^* + d^+)a|_{C^0(M)}.
    \end{align*}
\end{proposition}
\begin{proof}
    Let $\beta^\pm$ be cut-off functions as described in Lemma \ref{L:PoUBound}. Assume without loss of generality that $|d\beta^\pm|_{C^0(X)} < \frac{2}{L}$, which is possible when $L > 6$. Lemma \ref{L:PoUBound} gives a constant $C_1$, independent of $L$, such that
    \begin{align}\label{E:PoUBound2}
        |a|_{C^0(M)} \leq C_1(\|(d^* + d^+)\beta^+ a\|_{L^p(X)} + \|(d^* + d^+)\beta^- a\|_{L^p(X)}).
    \end{align}
    Calculating with the Leibniz rule yields
    \begin{align*}
        \|(d^* + d^+)\beta^\pm a\|_{L^p(X)} &\leq \|\beta^\pm (d^* + d^+) a\|_{L^p(X)} + \|d\beta^\pm \wedge a\|_{L^p(X)}.
    \end{align*}
    The product $\beta^\pm (d^* + d^+) a$ is supported inside $M^\pm$, thus
    \begin{align*}
        \|\beta^\pm (d^* + d^+) a\|_{L^p(X)} &= \|(d^* + d^+) a\|_{L^p(M^\pm)}.
    \end{align*}
    Since $N(L-1)$ has non-negative Ricci curvature, the \Weitzenbock formula \cite[Ex 2.31]{Salamon} implies that $|a|$ is a harmonic function when restricted to $N(L-1)$. Thus the maximum principle holds and $\sup_{N(L-1)}|a| = \sup_{\p N(L-1)}|a|$. Let $N(2, L-1)$ denote $\overline{N(L-1) - N(2)}$. Then $d\beta^\pm$ is supported inside $X^\pm \cap N(2, L-1)$ and
    \begin{align}\label{E:dbetaWedgea}
        \|d\beta^+ \wedge a\|_{L^p(X)} + \|d\beta^- \wedge a\|_{L^p(X)} &\leq \|d\beta^+ + d\beta^-\|_{L^p(N( L-1))} \sup_{N(2, L-1)}|a| \nonumber\\
        &\leq 4 L^{\frac{1}{p}-1}\vol(S^3)^\frac{1}{p} \sup_{\p N(L-1)}|a|.
    \end{align}
    Combining this with (\ref{E:PoUBound2}) gives
    \begin{align*}
        |a|_{C^0(M)} &\leq C_1\|(d^* + d^+)a\|_{L^p(M)} + 4C_1 L^{\frac{1}{p}-1}\vol(S^3)^\frac{1}{p} |a|_{C^0(\p N(L))} \\
        &\leq C_1\vol(M)^\frac{1}{p}|(d^* + d^+)a|_{C^0(M)} + 4C_1 L^{\frac{1}{p}-1}\vol(S^3)^\frac{1}{p} |a|_{C^0(M)}.
    \end{align*}
    Set $C_2 = C_1 \vol(M)^{\frac{1}{p}}$ and $C_3 = 4 C_1 \vol(S^3)^{\frac{1}{p}}$ to obtain
    \begin{align*}
        |a|_{C^0(M)}(1 - C_3L^{\frac{1}{p}-1}) \leq C_2|(d^* + d^+)a|_{C^0(M)}.
    \end{align*}
    Since $p > 4$, $\frac{1}{p}-1 < 0$ and $L \geq L_0$ implies $L^{\frac{1}{p}-1} \leq L_0^{\frac{1}{p} - 1}$. Set $L_0 = \left(2C_3\right)^{-\frac{p}{1-p}}$, which we can assume is larger than 6, so that $L \geq L_0$ implies 
    \begin{align*}
        (1 - C_3 L^{\frac{1}{p}-1}) &\geq (1 - C_3 L_0^{\frac{1 - p}{p}}) = \frac{1}{2}. 
    \end{align*}
    When $L \geq L_0$ it follows that
    \begin{align*}
        |a|_{C^0(M)} &\leq C_2|(d^* + d^+)a|_{C^0(M)}(1 - C_3L^{\frac{1}{p}-1})\inv\\
        &\leq 2C_2|(d^* + d^+)a|_{C^0(M)}.
    \end{align*}
    Let $C_E = 2C_2$, which is independent of $L$. 
\end{proof}
\begin{remark}\label{R:HarmonicNeckBound}
    Suppose instead that $(d^* + d^+)a$ only vanishes on $N(2, L-1)$. Then the maximum of $(d^* + d^+)a$ could be obtained on $\p N(2)$ instead of $\p N(L - 1)$. To overcome this, assume that there is a constant $C$, independent of $L$, such that
    \begin{align*}
        \sup_{N(2, L-1)} |a| \leq C \sup_{\p N(L-1)} |a|.
    \end{align*}
    Since $(d^* + d^+)a = 0$ on $N(2, L-1)$ and $\beta^\pm$ is supported in $X^\pm - N(2)$, the product $\beta^\pm (d^* + d^+) a$ is supported in $M^\pm$. We can still execute the above argument with (\ref{E:dbetaWedgea}) becoming
    \begin{align*}
        \|d\beta^+ \wedge a\|_{L^p(X)} + \|d\beta^- \wedge a\|_{L^p(X)} &\leq \|d\beta^+ + d\beta^-\|_{L^p(N( L-1))} \sup_{N(2, L-1)}|a| \nonumber\\
        &\leq 4 C L^{\frac{1}{p}-1}\vol(S^3)^\frac{1}{p}\sup_{\p N(L-1)}|a|.
    \end{align*}
    Setting $C_3 = 4C_1C\vol(S^3)^{\frac{1}{p}}$, there still exists constants $C_E$ and $L_0$ such that, when $L \geq L_0$, 
    \begin{align*}
        |a|_{C^0(X)} \leq C_E|(d^* + d^+)a|_{C^0(M)}.
    \end{align*}
\end{remark}

\subsection{Elliptic bootstrapping}

For a fixed connection $A \in \Jcal(X)$ on $X(L)$, an elliptic bootstrapping argument can be used to produce a polynomial $L^2_k$-bound on a monopole $(\psi, a)$ of the form
\begin{align*}
    \|(\psi, a)\|_{L^2_k} \leq C_B(1 + |(\psi, a)|)^d_\Co.
\end{align*}
The constant $C_B$ depends on the curvature of $A$ and the length of the neck $L$. To cooperate with neck stretching, we show that $C_B$ only increases polynomially in $L$.

\begin{lemma}\label{L:EllipticBootstrappingNeck}
    Let $A \in \Jcal_X$ be a connection on $X(L)$ and fix an integer $k \geq 2$. There are positive constants $C_B$ and $d$ such that, for any $L \geq 2$,  if $(\psi, a)$ is an $L^2_k$-pair with   
    \begin{align}\label{E:EBS-SWeqs}
        D_A\psi &= - ia \cdot \psi  \nonumber\\
        d^+ a &= iF^+_A - i\sigma(\psi),
    \end{align}
    then 
    \begin{align*}
        \|(\psi, a)\|_{L^2_k} \leq C_B L^d(1 + |(\psi, a)|_{C^0})^d.
    \end{align*}
\end{lemma}
\begin{proof}
    Use the first order differential operators $D_A$ and $d^+$ to define the $L^p_k$-norm so that 
    \begin{align*}
        \|(\psi, a)\|^p_{L^p_i} - \|(\psi, a)\|^p_{L^p} &= \|(D_A\psi, d^+a)\|^p_{L^p_{i-1}}.
    \end{align*}
    For any $0 \leq i \leq k$ and $2 \leq p \leq 2^{k+1}$, (\ref{E:EBS-SWeqs}) ensures that
    \begin{align*}
        \|(D_A\psi, d^+a)\|^p_{L^p_{i-1}} &\leq \|a\cdot \psi\|^p_{L^p_{i-1}} + (\|\sigma(\psi)\|_{L^p_{i-1}} + \|F^+_A\|_{L^p_{i-1}})^p
    \end{align*}
    By Lemma \ref{L:SMNeckIndependence} and \ref{L:SigmaPsiSobolevBound}, there are constants $C_{SM}$ and $C_\sigma$ independent of $L$ such that 
    \begin{align*}
        \|(\psi, a)\|_{L^p_i} &\leq C_{SM}L\|a\|_{L^{2p}_{i-1}}\|\psi\|_{L^{2p}_{i-1}} + C_\sigma L\|\psi\|^2_{L^{2p}_{i-1}} + \|F^+_A\|_{L^p_{i-1}} + \|(\psi, a)\|_{L^p}.
    \end{align*}
    Since $A$ is flat on the neck, $\|F^+_A\|_{L^p_{i-1}}$ is a constant independent of $L$. Thus there is a constant $C_1$ such that
    \begin{align*}
        \|(\psi, a)\|_{L^p_i} &\leq C_1L(\|(\psi, a)\|^2_{L^{2p}_{i-1}} + \|(\psi, a)\|_{L^p})
    \end{align*}
    for all $0 \leq i \leq k$ and $2 \leq p \leq 2^{k+1}$. Starting with $i = k$ and $p = 2$, inductively applying this inequality gives a bound 
    \begin{align*}
        \|(\psi, a)\|_{L^2_k} &\leq L^{d_1} f(\|(\psi, a)\|_{L^2},...,\|(\psi, a)\|_{L^{2^{k+1}}})
    \end{align*}
    for some natural number $d_1$ and polynomial $f$, both independent of $L$. Letting $d_2$ be the degree of $f$, there is a constant $C_2$ such that 
    \begin{align*}
        |f(x_1,...,x_k)| \leq C_2(1 + |x_1| + ... + |x_k|)^{d_2}.
    \end{align*}
    Since $\vol(X(L))$ increases linearly with $L$, there is a bound 
    \begin{align*}
        \|(\psi, a)\|_{L^p} &\leq \vol(X(L))^{\frac{1}{p}} |(\psi, a)|_{C^0}\\
        &\leq C_3L|(\psi, a)|_{C^0}.
    \end{align*}
    Here $C_3$ is a constant independent of $L$ and $p$. Letting $d = d_1 + d_2$, it follows that 
    \begin{align*}
        \|(\psi, a)\|_{L^2_k} &\leq C_2L^{d_1}(1 + \|(\psi, a)\|_{L^2} + ... + \|(\psi, a)\|_{L^{2^{k+1}}})^{d_2}\\
        &\leq C_2L^{d_1}L^{d_2}(1 + C_3|(\psi, a)|_{C^0} + ... + C_3|(\psi, a)|_{C^0})^{d_2}\\
        &\leq C_BL^d(1 + |(\psi, a)|_{C^0})^d
    \end{align*}
    for some constant $C_B$ independent of $L$.
\end{proof}
\begin{remark}\label{R:EllipticBootstrappingNeck}
    Assume that there is a smooth function $\rho : X \to \R$ and constant $C$ such that the pair $(\psi, a)$ instead satisfies
    \begin{align*}
        D_A\psi &= - i\rho a \cdot \psi  \\
        \|d^+ a\|_{L^p_i} &\leq C(\|\sigma(\psi)\|_{L^p_i} + \|F^+_A\|_{L^p_i})  \\
        \|\rho a\|_{L^p_i} &\leq C \|a\|_{L^p_i}
    \end{align*}
    for all $0 \leq i \leq k$, $2 \leq p \leq 2^{k+1}$. The same argument can be repeated, the only difference being that the constant $C_1$ now depends on $C$. Thus there still exists positive constants $C_B$ and $d$ such that 
    \begin{align*}
        \|(\psi, a)\|_{L^2_k} \leq C_B L^d(1 + |(\psi, a)|_{C^0})^d.
    \end{align*} 
    These constants depend on $C$, but are independent of $L$ so long as $C$ is.
\end{remark}

\subsection{Exponential decay}

Since $X(L)$ is compact, there are $L^p$-bounds on spinors and one-forms of the form 
\begin{align}\label{E:CpBound}
    \|(\psi, a)\|_{L^p} \leq C_p|(\psi, a)|_{C^0}
\end{align}
with $C_p = \vol(X(L))^\frac{1}{p}$. This constant $C_p$ grows linearly with the length of the neck. However, we will demonstrate that monopoles decay exponentially towards the middle of the neck, which will counteract this and other polynomial growth. The following work is adapted from Chapter 3 of \cite{FloerYM}.

Let $E \to S^3$ be a vector bundle over $S^3$, equipped with a metric $g_E$ and compatible connection $\nabla_E$. For notational simplicity, we will assume that $N(L) = S^3 \times [-L, L]$ has one connected component. Let $\pi : N(L) \to S^3$ be projection onto the $S^3$ component. Fix $k > 2$ and let $A : \Cinf(S^3, E) \to \Cinf(S^3, E)$ be a first order, self-adjoint, elliptic pseudo-differential operator on $E$. By spectral theory of elliptic operators, there is an orthonormal basis of eigenvectors $\{\phi_n\}_{n=-N}^\infty \subset L^2(S^3, E)$ for $A$ with discrete real eigenvalues $\{\lambda_n\}$. Label the eigenvalues so that the non-zero eigenvalues have a positive index and the zero eigenvalues (of multiplicity $N+1$) have a non-positive index. Thus there is a $\delta > 0$ such that $|\lambda_n| > \delta$ for all $n \geq 1$. Also ensure that the labeling is chosen so that $|\lambda_n| \geq |\lambda_m|$ when $n \geq m$.

Let $f_0 \in C^\infty(S^3, E)$ be a smooth section with eigen decomposition $f_0 = \sum_n f^n_0 \phi_n$ convergent in $L^2$ for $f^n_0 \in \R$. Then $Af_0$ is also smooth and its eigen decomposition is $Af_0 = \sum_n \lambda_n f^n_0 \phi_n$ since $A$ is self-adjoint. A smooth section $f$ of $\pi^* E \to N(L)$ also has a decomposition $f_t = \sum_n f^n(t) \phi_n$ for some functions $f^n : [-L, L] \to \R$. The smoothness of $f_t$ implies the smoothness of the component functions $f^n$ by the Leibniz integral rule.

Define a pseudo-differential operator by 
\begin{align}
    D : C^\infty(N(L), \pi^* E) &\to C^\infty(N(L), \pi^* E) \nonumber\\
    D = \frac{\p}{\p t} &+ A.
\end{align}
Assume that $D$ is elliptic and extend $D$ to an operator on $L^2$ sections. Recall that $C^+ = S^3 \times [L-1, L]$ and $C^- = S^3 \times [-L, -L + 1]$ denote collar neighbourhoods of the boundary of $N(L)$.
\begin{proposition}[Adapted from \cite{FloerYM} Lemma 3.2]\label{P:ExponentialDecayMiddleNeck}
    Fix constants $r \geq 1$ and $L \geq 2r $. Suppose $f \in L^2(N(L), \pi^* E)$ such that $f_t$ is orthogonal to $\ker A$ for all $t \in [-L, L]$. If $Df = 0$ then
    \begin{align}\label{E:ExpDecayIntegralInequality}
        \int_{N(2r)} |f|^2 &\leq \left(\frac{e^{-2\delta(L-2r )}}{1-e^{-2\delta}}\right) \left(\int_{C^-} |f|^2 + \int_{C^+} |f|^2\right)
    \end{align}
    and
    \begin{align}
        \sup_{N(r)}|f| \leq C_\delta e^{-\delta (L - 2r )}\sup_{N(L)} |f|.
    \end{align}
    where $\delta$ and $C_\delta$ are positive constants independent of $L$ and $r$.
\end{proposition}
\begin{proof}
    Note that since $D$ is assumed to be elliptic, $Df = 0$ implies that $f$ is smooth by elliptic regularity. 
    Write $Af_t = \sum_{n} \lambda_n f^n(t)\phi_n$ so that
    \begin{align*}
        \p_t f + \sum_n \lambda_n f^n\phi_n = 0.
    \end{align*}
    Taking the $L^2-$inner product with $\phi_n$ yields  
    \begin{align*}
        \p_t f^n(t) + \lambda_n f^n(t) &= 0.
    \end{align*}
    Since $f_t$ is orthogonal to $\ker A$ it can be assumed that $n \geq 1$ and $\lambda_n \neq 0$ so that 
    \begin{align*}
        f^n(t) &= e^{-\lambda_n t}f^n(0).
    \end{align*}    
    Notice that if $\lambda_n > 0$ then $f^n$ decays exponentially as $t$ increases and if $\lambda_n < 0$ then $f^n$ decays exponentially as $t$ decreases. To capture this behaviour, split $f^n = f^n_- + f^n_+$ defined by 
    \begin{align*}
        f^n_- = \begin{cases}
            0 & \text{if } \lambda_n > 0\\
            f^n & \text{if } \lambda_n < 0.
        \end{cases} \qquad
        f^n_+ &= 
        \begin{cases}
            f^n & \text{if } \lambda_n > 0\\
            0 & \text{if } \lambda_n < 0
        \end{cases}
    \end{align*}
    Also let $f_\pm = \sum_{n=1}^\infty f^n_\pm \phi_n$ so that $f = f_- + f_+$. Each half of $|f|^2$ is integrated separately.
    \begin{align}
        \int_{N(2r)} |f_+|^2 &= \int_{-2r }^{2r } \|f_+(t)\|^2_{L^2} dt  \nonumber\\
        &= \int_{-2r }^{2r } \sum_{n=1}^\infty e^{-2\lambda_n t}|f^n_+(0)|^2 dt  \nonumber\\
        &= \sum_{n=1}^\infty \frac{\sinh(4r \lambda_n)}{\lambda_n} |f_+^n(0)|^2 \label{E:sinhLemma1}
    \end{align}
    Here the monotone convergence theorem has been used to swap the sum and the integral. Integrating instead over the band $C^-$ gives 
    \begin{align}
        \int_{C^-} |f_+|^2 &= \sum_{n=1}^\infty \left(\frac{e^{2\lambda_n L} - e^{2\lambda_n (L-1)}}{2\lambda_n}\right) |f_+^n(0)|^2 \label{E:sinhLemma2}
    \end{align}
    Choose a $\delta > 0$ such that $|\lambda_n| > \delta$ for $n \geq 1$. When $\lambda_n > 0$, notice that
    \begin{align}
        \frac{\sinh(4r \lambda_n)}{\lambda_n} &\leq \frac{e^{4r \lambda_n}}{2\lambda_n} \nonumber\\
        &= \frac{e^{2\lambda_n L}}{2\lambda_n} \left(\frac{e^{-2\lambda_n(L-2r )}}{1-e^{-2\lambda_n}}\right)(1-e^{-2\lambda_n}) \nonumber\\
        &= \left(\frac{e^{-2\lambda_n(L-2r )}}{1-e^{-2\lambda_n}}\right)\left(\frac{e^{2\lambda_n L} - e^{2\lambda_n (L-1)}}{2\lambda_n}\right) \nonumber \\
        &\leq \left(\frac{e^{-2\delta(L-2r )}}{1-e^{-2\delta}}\right) \left(\frac{e^{2\lambda_n L} - e^{2\lambda_n (L-1)}}{2\lambda_n}\right) \label{E:sinhLemmaInequality}
    \end{align}
    The last line follows since $\lambda_n > \delta > 0$ and $L - 2r \geq 0$. Combining (\ref{E:sinhLemma1}), (\ref{E:sinhLemma2}) and (\ref{E:sinhLemmaInequality}) gives 
    \begin{align*}
        \int_{N(2r)} |f_+|^2 &\leq \sum_{n=1}^\infty \left(\frac{e^{-2\delta(L-2r )}}{1-e^{-2\delta}}\right) \left(\frac{e^{2\lambda_n L} - e^{2\lambda_n (L-1)}}{2\lambda_n}\right) |f_+^n(0)|^2\\
        &= \left(\frac{e^{-2\delta(L-2r )}}{1-e^{-2\delta}}\right) \int_{C^-} |f_+|^2  \\
        & \leq \left(\frac{e^{-2\delta(L-2r )}}{1-e^{-2\delta}}\right) \int_{C^-} |f|^2.
    \end{align*}
    Similarly when $\lambda_n < 0$,
    \begin{align*}
        \int_{C^+} |f_-|^2 &= \sum_{n=1}^\infty \left(\frac{e^{-2\lambda_n (L-1)} - e^{-2\lambda_n L}}{2\lambda_n}\right) |f_-^n|^2\\
        &= \sum_{n=1}^\infty \left(\frac{e^{2|\lambda_n| L} - e^{2|\lambda_n| (L-1)}}{2|\lambda_n|}\right) |f_-^n|^2.
    \end{align*}
    Now (\ref{E:sinhLemmaInequality}) can be applied to get 
    \begin{align*}
        \int_{N(2r)} |f_-|^2 &= \sum_{n=1}^\infty \frac{\sinh(4r |\lambda_n|)}{|\lambda_n|} |f_-^n(0)|^2\\
        &\leq \left(\frac{e^{-2\delta(L-2r )}}{1-e^{-2\delta}}\right) \int_{C^+} |f_-|^2\\ 
        &\leq \left(\frac{e^{-2\delta(L-2r )}}{1-e^{-2\delta}}\right) \int_{C^+} |f|^2.
    \end{align*}
    It follows that 
    \begin{align}\label{E:IntegralBandBound}
        \int_{N(2r)} |f|^2 &\leq \left(\frac{e^{-2\delta(L-2r )}}{1-e^{-2\delta}}\right) \left(\int_{C^-} |f|^2 + \int_{C^+} |f|^2\right).
    \end{align}
    This proves the first inequality (\ref{E:ExpDecayIntegralInequality}).

    The supremum and essential supremum of $|f|$ agree because $f$ is continuous. Since the sequence $(\sum_{i=1}^N f^n(t)\phi_n)_{N=1}^\infty$ converges to $f_t$ in $L^2$ as $N \to \infty$, there is a subsequence that converges to $f_t$ pointwise almost everywhere. Let $(x_0,t_0) \in S^3 \times [-r, r]$ be any point such that 
    \begin{align*}
        f_{t_0}(x_0) &= \sum_{n=1}^\infty e^{-\lambda_n t_0}f^n(0) \phi_n(x_0).
    \end{align*}
    Since $t_0 \in [-r, r]$ it follows that 
    \begin{align*}
        |f_{t_0}(x_0)| &\leq \sum_{n=1}^\infty e^{r|\lambda_n|} |f^n(0)||\phi_n(x_0)|.
    \end{align*}
    The Sobolev embedding $L^2_2(S^3, E) \to C^0(S^3, E)$ gives a constant $C_S$ such that $|\phi_n|_{C^0} \leq C_S \|\phi_n\|_{L^2_2}$ for all $n$. Further the second order elliptic operator $A^2 : L^2_2(S^3, E) \to C^0(S^3, E)$ provides an elliptic inequality 
    \begin{align*}
        \|\phi_n\|_{L^2_2} &\leq C_E(\|A^2\phi_n\|_{L^2} + \|\phi_n\|_{L^2}) \\
        &= C_E(\lambda_n^2 + 1)\|\phi_n\|_{L^2}.
    \end{align*}
    Note that $C_S$ and $C_E$ are independent of $L$. Since $\|\phi_n\|_{L^2} = 1$, we have $|\phi_n|_{C_0} \leq C_EC_S(\lambda_n^2 + 1)$ and 
    \begin{align*}
        |f_{t_0}(x_0)| &\leq \sum_{n=1}^\infty C_E C_S (\lambda_n^2 + 1)e^{r|\lambda_n|}|f^n(0)|.
    \end{align*}
    Lemma \ref{L:WeylsLemma} provides a bound 
    \begin{align*}
        \left(\sum_{n=1}^\infty (\lambda_n^2 + 1)e^{r|\lambda_n|}|f^n(0)|\right)^2 &\leq C'\sum_{n=1}^\infty \frac{\sinh(4r |\lambda_n|)}{|\lambda_n|}|f^n(0)|^2
    \end{align*}
    for some constant $C'$ which depends only on $\{\lambda_n\}$. Combining this with (\ref{E:sinhLemma1}) produces 
    \begin{align*}
        |f_{t_0}(x_0)|^2 &\leq C \sum_{n=1}^\infty \frac{\sinh(4r |\lambda_n|)}{|\lambda_n|}|f^n(0)|^2 \\
        &= C\int_{S^3 \times [-2r , 2r ]} |f|^2
    \end{align*}
    where $C = C' C_S^2 C_E^2$. Applying (\ref{E:IntegralBandBound}) and taking the essential supremum over $N(r)$ yields
    \begin{align*}
        \sup_{N(r)} |f|^2 &\leq C\left(\frac{e^{-2\delta(L-2r )}}{1-e^{-4\delta}}\right) \left(\int_{C^-} |f|^2 + \int_{C^+} |f|^2\right)\\
        &\leq \left(\frac{2C\vol(S^3)}{1 - e^{-4\delta}}\right)e^{-2\delta (L - 2r )} \sup_{N(L)} |f|^2.
    \end{align*}
    Let $C_\delta = \sqrt{\frac{2C\vol(S^3)}{1 - e^{-4\delta}}}$ so that
    \begin{align*}
        \sup_{N(r)} |f| &\leq C_\delta e^{-\delta (L - 2r )}\sup_{N(L)} |f|.
    \end{align*}
\end{proof}
\begin{corollary}\label{C:ExponentialDecayMiddleNeck}
    Suppose that $a \in L^2(N(L-1), T^* N(L-1))$ is a 1-form such that $(d^* + d^+)a = 0$. Then for any $r \geq 1$ and $L \geq 2r + 1$,
    \begin{align}\label{E:CExpDecayMiddleNeck}
        \sup_{N(r)} |a \wedge dt| \leq C_\delta e^{-\delta (L - 2r )} \sup_{N(L-1)} |a|
    \end{align}
    for some positive constants $\delta$ and $C_\delta$ independent of $L$ and $r$.
\end{corollary}
\begin{proof}
    It is shown in (\ref{E:FormIdentification}) that $d^* + d^+$ can be identified as an operator on $\Cinf(N(L-1), \R \oplus \pi^* T^* S^3)$ and that $d^* + d^+ = \frac{\p}{\p t} + \Lcal$. Here $\Lcal$ is a self-adjoint, elliptic operator on $\Omega^0(S^3) \oplus \Omega^1(S^3)$ with $\Lcal^2 = dd^* + d^*d$. Note that $d^* + d^+$ is also self-adjoint and elliptic. Since $b_1(S^3) = 0$, the kernel of $\Lcal$ is one dimensional consisting of only constant functions. Thus there is an eigenbasis $\{\phi_n\}_{n=0}^\infty$ of $\Lcal$ with eigenvalues $\{\lambda_n\}_{n=0}^\infty$ such that $\phi_0$ is a non-zero constant function on $S^3$, $\lambda_0 = 0$ and $\lambda_n \neq 0$ for $n \geq 1$. Write
    \begin{align*}
        a_t = a_0(t)\phi_0dt + \sum_{n=1}^\infty a_n(t)\phi_n
    \end{align*}
    for some smooth functions $a_n : [-L+1,L-1] \to \R$. As in Proposition \ref{P:ExponentialDecayMiddleNeck}, $\p_t a_0 + \lambda_0 a_0 = 0$ and therefore $a_0$ is a constant function. Now $a' = a - a_0\phi_0dt$ is $L^2$-orthogonal to $\ker L$ for all $t$. Since $(d^* + d^+)(a_0\phi_0dt) = 0$ we have $(d^* + d^+)a' = 0$. Proposition \ref{P:ExponentialDecayMiddleNeck} gives constants $C_1$ and $\delta$, independent of $L$ and $r$, such that
    \begin{align*}
        \sup_{N(r)} |a'| &\leq C_1 e^{-\delta (L-2r -1)}\sup_{N(L-1)}|a'| \\
        &\leq C_1' e^{-\delta (L-2r )}\left(\sup_{N(L-1)}|a| + \sup_{N(L-1)}|a_0\phi_0dt|\right).
    \end{align*}
    Since $a_0$ and $\phi_0$ are constants, we can calculate 
    \begin{align*}
        \|a_0\phi_0 dt\|^2_{L^2} &= \int_{N(L-1)} |a_0 \phi_0 dt|^2  \\
        &= 2\vol(S^3)(L-1) |a_0|^2 |\phi_0|^2.
    \end{align*}
    The decomposition $a = a' + a_0\phi_0dt$ is $L^2$-orthogonal, hence $\|a_0\phi_0dt\|_{L^2}^2 = \|a\|^2_{L^2} - \|a'\|^2_{L^2}$. It follows that 
    \begin{align*}
        2\vol(S^3)(L-1) |a_0|^2 |\phi_0|^2 &= \|a_0\phi_0 dt\|^2_{L^2} \\
        &\leq \|a\|_{L^2}^2 \\
        &\leq 2(L-1)\vol(S^3)\sup_{N(L-1)} |a|^2.
    \end{align*}
    Thus $|a_0| \leq \frac{1}{|\phi_0|}\sup_{N(L-1)} |a|$ and there is a constant $C_\delta$ with 
    \begin{align*}
        \sup_{N(r)} |a'| \leq C_\delta e^{-\delta (L-2r )}\sup_{N(L-1)}|a|.
    \end{align*}
    Finally, $|a \wedge dt| = |a' \wedge dt| \leq |a'|$ and (\ref{E:CExpDecayMiddleNeck}) follows.
\end{proof}
\begin{corollary}\label{C:ExponentialDecayMiddleNeckSpinor}
    Let $A_0$ be a flat reference connection on $N(L)$. Suppose $\psi \in L^2(N(L), W^+)$ is a spinor such that $D_{A_0} \psi = 0$. Then for any $r \geq 1$ and $L \geq 2r $,
    \begin{align}
        \sup_{S^3 \times [-r,r]} |\psi| \leq C_{\delta'} e^{-\delta' (L - 2r )} \sup_{S^3 \times [-L,L]} |\psi|
    \end{align}
    for some positive constants $\delta'$ and $C_{\delta'}$ independent of $L$ and $r$.
\end{corollary}
\begin{proof}
    The \spinc structure on $X$ is defined so that, on the neck, Clifford multiplication $\Gamma : TN(L) \to \End(W)$ is induced by the Clifford multiplication $\gamma : TS^3 \to \End(W_{S^3})$ on $S^3$. 
    \begin{align}
        \Gamma(\p_{x_i}) &= 
        \begin{pmatrix}
            0 & \gamma(\p_{x_i})\\
            -\gamma(\p_{x_i})^* & 0
        \end{pmatrix}, \qquad 
        \Gamma(\p_t) = \begin{pmatrix}
            0 & \id\\
            -\id & 0
        \end{pmatrix}. \label{E:SpincStructureS3}
    \end{align}
    Here $\{\p_t, \p_{x_1}, \p_{x_2}, \p_{x_3}\}$ is a basis for $TN(L)$ corresponding to local coordinates $(x_1, x_2, x_3, t)$ of $N(L)$. The \spinc connection $\nabla_{A_0}$ for the reference connection $A_0$ is given by the formula
    \begin{align}\label{E:flatRefConnection}
        \nabla_{A_0} = dt \otimes \frac{\p}{\p t} + \nabla^{S^3}.
    \end{align}
    Here $\nabla^{S^3}$ is a \spinc connection on $W_{S^3} \to S^3$. Since $b_2(S^3) = 0$, it can be assumed that $\nabla^{S^3}$ is flat. This equation is understood by treating a spinor $\psi \in C^\infty(N(L), W^+)$ as a time-dependent family of spinors $\{\psi_t\}$ on $S^3$. Over the neck $N(L)$, the Dirac operator $D_{A_0} : C^\infty(X, W^+) \to C^\infty(X, W^-)$ takes the form
    \begin{align}\label{E:DiracIdentification}
        D_{A_0} &= \Gamma(\p_t) \cdot \frac{\p}{\p t} + \sum_{i=1}^3 \Gamma(\p_{x_i}) \cdot \nabla^{S^3}_{x_i} \nonumber \\ 
            &= \Gamma(\p_t) \cdot \frac{\p}{\p t} - \sum_{i=1}^3 \Gamma(\p_t) \cdot \gamma(\p_{x_i})\nabla^{S^3}_{x_i} \nonumber\\
            &= \Gamma(\p_t) \left(\frac{\p}{\p t} - D^{S^3}\right).
    \end{align}
    Here $D^{S^3}$ is the self-adjoint Dirac operator associated to $\nabla^{S^3}$. Note that both $D_{A_0}$ and $D^{S^3}$ are elliptic. Since $A_0$ is flat and $S^3$ has positive scalar curvature, the \Weitzenbock formula implies that $\ker D^{S^3} = 0$. Therefore $\psi$ is automatically orthogonal to $\ker D^{S^3}$ and the result follows from Proposition \ref{P:ExponentialDecayMiddleNeck}.
\end{proof}
To complete the proof of Proposition \ref{P:ExponentialDecayMiddleNeck}, it remains to prove the following lemma.
\begin{lemma}\label{L:WeylsLemma}
    Let $A : C^\infty(S^3, E) \to C^\infty(S^3, E)$ be an elliptic, self-adjoint, pseudo-differential operator of positive order. Let $0 < |\lambda_1| \leq |\lambda_2| \leq ...$ denote the non-zero eigenvalues of $A$, ordered by magnitude. There exists a constant $C$ such that, for any $r \geq 1$,
    \begin{align}
        \left(\sum_{n=1}^\infty (\lambda_n^2 + 1)e^{r|\lambda_n|}|a_n|\right)^2 \leq C\sum_{n=1}^\infty \frac{\sinh(4r |\lambda_n|)}{|\lambda_n|}|a_n|^2
    \end{align}
    for any real number sequence $\{a_n\}$.
\end{lemma}
\begin{proof}
    First, apply the Cauchy-Schwarz inequality to obtain
    \begin{align*}
        \left(\sum_{n=1}^\infty (\lambda_n^2 + 1)e^{r|\lambda_n|}|a_n|\right)^2 &= \left(\sum_{n=1}^\infty \left(\frac{(\lambda_n^2 + 1)\sqrt{|\lambda_n|}e^{r|\lambda_n|}}{\sqrt{\sinh(4r |\lambda_n|)}}\right)\left(\frac{\sqrt{\sinh(4r |\lambda_n|)}}{\sqrt{|\lambda_n|}}|a_n|\right)\right)^2 \\
        &\leq \left(\sum_{n=1}^\infty \frac{(\lambda_n^2 + 1)^2|\lambda_n|e^{2r |\lambda_n|}}{\sinh(4r |\lambda_n|)}\right)\left(\sum_{n=1}^\infty\frac{\sinh(4r |\lambda_n|)}{|\lambda_n|}|a_n|^2\right)
    \end{align*}
    It suffices to bound $\sum_{n=1}^\infty \frac{(\lambda_n^2 + 1)^2|\lambda_n|e^{2r |\lambda_n|}}{\sinh(4r |\lambda_n|)}$. Fix $0 < \delta < |\lambda_1|$. The function $\frac{e^{4x}}{\sinh(4x)}$ is bounded on $[\delta, \infty)$, therefore there is a constant $C_1$ such that, for all $x \geq \delta$, 
    \begin{align*}
        \frac{e^{2x}}{\sinh(4x)} \leq C_1 e^{-2x}
    \end{align*}
    Apply this to $r|\lambda_n|$ to produce 
    \begin{align*}
        \sum_{n=1}^\infty \frac{e^{2r |\lambda_n|}(\lambda_n^2 + 1)^2|\lambda_n|}{\sinh(4r |\lambda_n|)} &\leq \sum_{n=1}^\infty C_1(\lambda_n^2 + 1)^2|\lambda_n|e^{-2r |\lambda_n|} \\
        &\leq \sum_{n=1}^\infty C_1(\lambda_n^2 + 1)^2|\lambda_n|e^{-2|\lambda_n|}.
    \end{align*}
    Similarly, there exists a constant $C_2$ such that $x(x^2+1)^2 e^{-x} \leq C_2$ for all $x \geq 0$. It follows that
    \begin{align}\label{E:WeylSum1}
        \sum_{n=1}^\infty C_1(\lambda_n^2 + 1)^2|\lambda_n|e^{-2|\lambda_n|} &\leq \sum_{n=1}^\infty C_1C_2 e^{-|\lambda_n|}.
    \end{align}
    Since $A$ is elliptic and self-adjoint, Weyl's law \cite[Lemma 1.6.3]{Gilkey} implies that there exists a constant $C_3$ and an exponent $\alpha > 0$ such that $|\lambda_n| \geq C_3 n^\alpha$ for large enough $n$. Thus to show that (\ref{E:WeylSum1}) is finite, it is enough to show that 
    \begin{align*}
        \sum_{n=1}^\infty e^{-n^a} < \infty.
    \end{align*}
    This follows from the integral test. Let $u = x^\alpha$ so that
    \begin{align*}
        \int_1^\infty e^{-x^\alpha}dx &= \frac{1}{\alpha}\int_1^\infty u^{\frac{1 - \alpha}{\alpha}}e^{-u}du\\
        &\leq \frac{1}{\alpha}\Gamma\left(\frac{1}{\alpha}\right)\\
        &< \infty.
    \end{align*}
    Therefore $C = \sum_{n=1}^\infty C_1C_2 e^{-|\lambda_n|}$ is a suitable constant.
\end{proof}

    \section{Proof of the Families Permutation Theorem}

Now we construct a homotopy from $\mu_X$ to $V\inv \mu_{X^\tau} V$ that, after restricting to a suitably chosen disk bundle, is a homotopy through compact perturbations of $l$. Such a homotopy proves Theorem \ref{T:FamiliesGluingTheorem} because of Corollary \ref{C:PhiBijection}. The final homotopy is a concatenation of three compact homotopies, each dealing with problematic quadratic terms of $\mu_X$ separately. The idea to use these particular homotopies comes from Bauer's proof in \cite{BF2}, however great care is taken to ensure that these homotopies satisfy the necessary boundedness conditions and that these conditions are compatible with stretching the neck length. 

Fix $L > 2$ and let $E = E(L) \to B$ be a family of closed 4-manifolds $X$ with a separating neck of length $2L$. Fix a reference connection $A_0$, which can be assumed to be flat on the neck $N_B(L)$. Recall that for $\theta \in \Jcal$, $A_\theta$ denotes the associated connection $A_0 + i\theta$. Note that $A_\theta$ is also flat on the neck. For a given $R \leq L$, let $\rho_R : E \to [0,1]$ be a smooth function that vanishes on $N_B(R-1)$ and is identically $1$ on $E - N_B(R)$. Along $N_B(R) - N_B(R-1)$, we require that $\rho_R$ only depends on the interval coordinate. For $s \in [0,1]$, let $\rho_R^s$ be a linear homotopy ending at $\rho_R$ of the form 
\begin{align*}
    \rho_R^s &= (1-s) + s\rho_R.
\end{align*}
Since $\rho_R^s$ is constant outside of $N_B(R) - N_B(R-1)$, the $C^k$-norm of $\rho_R^s$ is independent of $L$ for all $R$ and $s$.

\subsection{The first homotopy}

To define the first homotopy $F : \Acal \to \Ccal$ fiberwise, let $\theta \in H^1(X_b ; \R)$ for some $b \in B$ and set
\begin{align*}
    F^\theta_s(\psi, a) &= (D_{A_\theta}\psi + ia \cdot \psi, d^+ a - i F^+_{A_\theta} + i\rho_L^s\sigma(\psi), d^*a, \pr(a)).
\end{align*}
Notice that $F_0 = \mu_X$ and that the quadratic term in the second factor of $F_1$ vanishes on $N(L-1)$. The proof that $(F_s)\inv(0)$ is $L^2_k$-bounded uses variations on techniques that show compactness of the moduli space in ordinary Seiberg-Witten theory.
\begin{proposition}\label{P:FirstHomotopy}
    Fix a connection $A_\theta$ for $X_b(L)$ with $\theta \in \Jcal_b$ for some $b \in B$. For $s \in [0,1]$, the preimage $(F^\theta_s)\inv(0)$ is uniformly $L^2_k$-bounded.
\end{proposition}
\begin{proof}
    Let $(\psi, a) \in (F^\theta_s)\inv(0)$ so that $D_{A_\theta+ia}\psi = 0$ and $F^+_{A_\theta+ia} = \rho_L^s\sigma(\psi)$. The \Weitzenbock formula \cite[Theorem 6.19]{Salamon} applied to the connection $A_\theta + ia$ gives a pointwise bound
    \begin{align*}
        \Delta_g|\psi|^2 + \frac{s_X}{2}|\psi|^2 + \<F^+_{A_\theta+ia}\psi, \psi\> &\leq 2\<D^*_{A_\theta+ia}D_{A_\theta+ia}\psi, \psi\>\\
        \Delta_g|\psi|^2 + \frac{s_X}{2}|\psi|^2 + \frac{1}{2}\rho_L^s|\psi|^4 &\leq 0.
    \end{align*}
    Here $s_X$ is the scalar curvature of $X = X_b(L)$ and $\Delta_g$ is the positive definite Laplace-beltrami operator, which is non-negative at a maximum. Let $S = \sup_{X}\{0, -\sX\}$ and note that $\sX$ is positive along the neck. Thus $\Delta_g |\psi|^2 \leq 0$ on $N(L)$ and $|\psi|^2$ achieves a maximum on $M = \overline{X - N(L-1)}$. At such a maximum $x \in M$, we have 
    \begin{align*}
        |\psi(x)|^2(s_X(x) + |\psi(x)|^2) &\leq 0.
    \end{align*}
    It follows that $|\psi|_\Co^2 \leq S$. To bound $|a|_{C_0}$, notice that $d^+a = -i\rho_L^s\sigma(\psi) + iF^+_{A_\theta}$ and
    \begin{align*}
        |d^+ a| \leq |\sigma(\psi)| + |F^+_{A_\theta}|.
    \end{align*}
    Fix some $p \geq 4$ so that the Sobolev embedding $L^p_1(X, T^* X) \subset C^0(X, T^*X)$ gives a constant $C_S$ with $|a|_{C^0} \leq C_S\|a\|_{L^p_1}$. Since $d^* + d^+$ is a self-adjoint elliptic operator, \cite[Theorem 4.12]{WellsGarcia} guarantees the existence of a constant $C_e$ such that
    \begin{align*}
        |a|_\Co \leq C_S C_E(\|\sigma(\psi)\|_\Lp + \|F^+_{A_\theta}\|_\Lp).
    \end{align*}
    This shows that $|a|_\Co$ is bounded by a constant since $|\psi|_\Co$ is. For bootstrapping, $D_{A_\theta} \psi = -ia \cdot \psi$ and $\|d^+a\|_{L^p_i} = \|-\rho_L^s\sigma(\psi) + F^+_{A_\theta}\|_{L^p_i}$. The $C^k$-norm of $\rho_L^s$ determines a constant $C$ such that, for any $0 \leq i \leq k$ and $2 \leq p \leq 2^{k+1}$
    \begin{align*}
        \|d^+a\|_{L^p_i} \leq C\|\sigma(\psi)\|_{L^p_i} + \|F^+_{A_\theta}\|_{L^p_i}.
    \end{align*}
    From Proposition \ref{L:EllipticBootstrappingNeck} and Remark \ref{R:EllipticBootstrappingNeck} there a constant $C_B$ and integer $d \geq 1$ such that 
    \begin{align*}
        \|(\psi, a)\|_{L^2_k} &\leq C_B L^d (1 + |(\psi, a)|_\Co)^d.
    \end{align*}
    The norm $|(\psi,a)|_\Co$ is bounded by a constant, hence so is $\|(\psi, a)\|_{L^2_k}$. This bound is independent of $s$, but depends on the connection ${A_\theta}$ and neck length $L$.
\end{proof}

\begin{proposition}\label{P:FDiskBundleFamilies}
    The map $F_s : \Acal \to \Ccal$ is a homotopy through compact perturbations of $l$.
\end{proposition}
\begin{proof}
    For each $s \in [0,1]$, it is clear that $F_s = l + c_s$ with $c_s$ compact. Proposition \ref{P:FirstHomotopy} gives for each $[\theta] \in \Jcal$ a radius $R^\theta > 0$ such that, for any $(\psi, a) \in (F_s^\theta)\inv(0)$,
    \begin{align*}
        \|(\psi, a)\|_{L^2_k} \leq R^\theta.
    \end{align*}
    This bound does not depend on $s \in [0,1]$. Let $R$ be the supremum of $R^\theta$ over $\Jcal$, which exists since $\Jcal$ is compact. Let $D \subset \Acal$ be a disk bundle over $\Jcal$ with $L^2_k$-radius $2R$. This shows in fact that each preimage $F\inv_s(0)$ is contained in a bounded disk bundle, a stronger result than required.
\end{proof}

\subsection{The second homotopy}

The second homotopy $G_s$ for $s \in [0,3]$ is constructed in three stages. For $s \in [0,1]$ define
\begin{align*}
    G_s^\theta(\psi, a) &= (D_{A_\theta}\psi+ i\rho_r^s a \cdot \psi, d^+ a -iF^+_{A_\theta} + i\rho_L\sigma(\psi), d^* a,\pr(a))
\end{align*}
This homotopy eliminates the other quadratic term $ia \cdot \psi$ from $N_B(r-1)$. The constant $r \geq 3$ will be defined later. It is assumed without loss of generality that $L \geq 2r + 1$.

To define the second stage of $G$, let $P = G_1$. This stage will transform $P$ to $P^\tau = V\inv P V$ where the action of $V$ was defined in equation \ref{E:VDefinitionFamilies}. Restricting to $N_B(r-1)$, $P$ is a first order linear differential operator given by the formula
\begin{align*}
    P^\theta(\psi, a) &= (D_{A_\theta}\psi, d^+a, d^* a,\pr(a)).
\end{align*}
Note that $F_{A_\theta}^+ = 0$ since $A_\theta$ is flat on the neck. For $s \in [0,1]$, let
\begin{align*}
    V_s(x, t) = \gamma((s - 1) \cdot \vphi(t)) : S(V_0) \times [-L, L] \to SO(n).
\end{align*}
Define $Q_s : \Acal \to \Ccal$ by
\begin{align*}
    Q^\theta_s(\psi, a) &= V_s\inv \p_t V_s( dt \cdot \psi, (dt \wedge a)^+, *(*\veca \wedge dt), 0).
\end{align*}
Here $V\inv \p_t V$ is a matrix functions which acts on each vector $dt \cdot \vecpsi$, $(dt \wedge \veca)^+$ and $*(*\veca \wedge dt)$. Notice that $Q$ vanishes outside of $N(1)$ since $\p_t V = 0$ away from the short neck. Applying the Leibniz rule, it follows that
\begin{align*}
    V_s\inv P V_s(\psi, a) &= P(\psi, a) + Q_s(\psi, a).
\end{align*}
For $s \in [1,2]$, define $G_s$ by 
\begin{align}\label{E:GHomotopyMiddleFamilies}
    G_{s} = P + Q_{s}.
\end{align}
Each $Q_s$ has the property that $Q_s = 0$ outside of $N(1)$, hence this formula is well defined globally. Restricted to the neck $N(L-1)$, equation (\ref{E:GHomotopyMiddleFamilies}) is equivalent to $G_s = V_{s}\inv P V_{s}$. For the final stage $s \in [2,3]$, let $G_s = V\inv G_{3-s} V$. Now $G$ is a homotopy from $G_0 = F_1$ to $G_3 = V\inv F_1 V$.

Since $G$ alters the $D_{A+ia}\psi = 0$ equation, the previous argument fails to bound $G_s\inv(0)$. However to show that $G$ is a compact homotopy, it is only necessary to find an $L^2_k$-disk bundle containing $G_0\inv(0)$ and $G_3\inv(0)$ for which its bounding sphere bundle does not intersect $G_s\inv(0)$ for any $s \in [0,3]$. The following results help accomplish this by proving similar results for the $C^0$-norm of zeroes of $G_s$. For any $[\theta] \in \Jcal$ with $\theta \in H^1(X_b ; \R)$, we set $X = X_b(L)$.
\begin{lemma}\label{L:WeitzenbockArgument}
    Let $(\psi, a) \in (G^\theta_s)\inv(0)$ for some $s \in [0,3]$ and $[\theta] \in \Jcal$. If $\sup_X |\psi|$ is achieved at some $x \in M$, then $|\psi|_{\Co(X)}^2 \leq S$ for $S = \sup_{X}\{0, -s_X\}$.
\end{lemma}
\begin{proof}
    Restricted to $M = \overline{X - N(L-1)}$, the pair $(\psi, a)$ satisfies $D_{A+ia}\psi = 0$ and $F^+_{A+ia} = \sigma(\psi)$. As in Proposition \ref{P:FirstHomotopy}, the \Weitzenbock formula on $M$ gives 
        \begin{align*}
            \Delta_g|\psi|^2 + \frac{s_X}{2}|\psi|^2 + \frac{1}{2}|\psi|^4 &\leq 0
        \end{align*}
        Since $X$ is a closed 4-manifold, $\Delta_g |\psi|^2 \geq 0$ at $x$. Since $x \in M$, it follows that 
        \begin{align*}
            |\psi(x)|^2(s_X(x) + |\psi(x)|^2) \leq 0.
        \end{align*}
        Therefore $|\psi|^2 \leq S$ since $|\psi(x)| = |\psi|_{\Co(X)}$.
\end{proof}

\begin{lemma}\label{L:VsNormChange}
    Let $(\psi, a)$ be a spinor-from pair along the $n$-component neck $N(L)$. For any $0 \leq R \leq L$, we have 
    \begin{align*}
        \sup_{N(R)}|\psi| &\leq n\sup_{N(R)} |V_s \psi| \leq n^2\sup_{N(R)}|\psi| \nonumber\\
        \sup_{N(R)}|a| &\leq n\sup_{N(R)} |V_s a| \leq n^2\sup_{N(R)}|a|.
    \end{align*}
\end{lemma}
\begin{proof}
We prove only the spinor case. Let $\vecpsi$ be the vectorised version of $\psi$ as in (\ref{E:vecpsi}). That is, $\vecpsi : S^3 \times [-L, L] \to \oplus_{i=1}^n W^+$ with the $i$-th component $\vecpsi_i$ corresponding to the restriction of $\psi$ to the $i$th connected component of $N(L)$. The restriction of $V_s\psi$ to the $i$th connected component of $N(L)$ is given by the $i$th component of $V_s \vecpsi$. Inside $N(R)$, we have
\begin{align*}
    |(V_s \vecpsi)_i| &= \left|\sum_j  (V_s)_{ij} \vecpsi_j\right| \nonumber\\
    &\leq \sum_j  |(V_s)_{ij}| |\vecpsi_j| \nonumber\\
    &\leq \left(\sum_j  |(V_s)_{ij}|\right) \sup_{N(R)} |\psi| \nonumber\\
    &= n\sup_{N(R)} |\psi|.
\end{align*}
The last line follows since $V_s$ is valued in $SO(n)$, hence the absolute value of each of its entries is less than 1. Therefore $\sup_{N(R)} |V_s \psi| \leq n\sup_{N(R)} |\psi|$. The same calculation shows that $\sup_{N(R)} |\psi| = \sup_{N(R)} |V_s\inv V_s \psi| \leq \sup_{N(R)} n|V_s \psi|$.
\end{proof}
\begin{remark}
    For any $R \leq L$, the same calculation can be used to show that 
    \begin{align*}
        \sup_{\p N(R)}|\psi| &\leq n\sup_{\p N(R)} |V_s \psi| \leq n^2\sup_{\p N(R)}|\psi| \nonumber\\
        \sup_{\p N(R)}|a| &\leq n\sup_{\p N(R)} |V_s a| \leq n^2\sup_{\p N(R)}|a|.
    \end{align*}
\end{remark}

\begin{lemma}\label{L:GHomotopyHelper}
    There exists positive constants $L_0, C_E, \delta$ and $C_\delta$ such that the following holds. For any $s \in [0,3]$, let $(\psi, a) \in (G_s^\theta)\inv(0)$ be a spinor-form pair on $X_b(L)$. If $L \geq L_0$, then
    \begin{align}\label{E:TwoBounds}
        |a|_{\Co(X)} &\leq C_E|(d^* + d^+)a|_{\Co(M)} \nonumber\\
        \sup_{N(r)} |a \wedge dt| &\leq C_\delta e^{- \delta (L-2r)}\sup_{N(L-1)}|a|.
    \end{align}
\end{lemma}
\begin{proof}
    For $s \in [0,1]$, we have
    \begin{align*}
        d^+a &= iF^+_{A_\theta} - i\rho_L\sigma(\psi) \nonumber \\
        d^*a &= 0 \nonumber \\
        \pr(a) &= 0.
    \end{align*}
    Along $N(L-1)$, $d^+ a = iF^+_{A_\theta}$ and therefore $d^+ a = 0$ since $A_\theta$ is flat on the neck. Thus $(d^* + d^+)a$ vanishes on $N(L-1)$. Hence Proposition \ref{P:HarmonicNeckBound} gives constants $C_E'$ and $L_1$ such that, if $L \geq L_1$ then
    \begin{align*}
        |a|_{C^0(X)} &\leq C_E'|(d^* + d^+)a|_{\Co(M)}.
    \end{align*}
    Further, Corollary \ref{C:ExponentialDecayMiddleNeck} applies to $a \wedge dt$ yielding, for some $\delta > 0$ and $C_\delta'$ independent of $L$,
    \begin{align*}
        \sup_{N(r)} |a \wedge dt| &\leq C_\delta'e^{- \delta (L-2r)}|a|_{\Co(N(L))}.
    \end{align*}
    If $s \in [1,2]$, the condition $\pr(a) = 0$ still holds. Restricting to $N(L-1)$ we have $V_s\inv P V_s(\psi, a) = 0$ and therefore $V_s(\psi, a)$ is a solution to $P$. Note that $V_s(\psi, a)$ is only defined on the neck when $s \in (0,1)$ and that $(d^* + d^+)V_s a = 0$ on $N(L-1)$. This means that $\sup_{N(L-1)}|V_s a| = \sup_{\p N(L-1)}|V_s a|$ by the maximum principle. Lemma \ref{L:VsNormChange} implies that 
    \begin{align}\label{E:aAlmostHarmonic}
        \sup_{N(L-1)}|a| &\leq n\sup_{N(L-1)}|V_s a| \nonumber\\
        &= n\sup_{\p N(L-1)}|V_s a| \nonumber\\
        &\leq n^2 \sup_{\p N(L-1)}|a|.
    \end{align}
    Thus $|a|_{\Co(X)} \leq n^2|a|_{\Co(M)}$. Restricting to $X - N(1)$ instead, we have $P(\psi, a) = 0$. This means that $(d^* + d)a = 0$ along $N(2, L)$. Now (\ref{E:aAlmostHarmonic}) with Remark \ref{R:HarmonicNeckBound} implies the existence of constants $L_2$ and $C_E''$ such that, if $L \geq L_2$,
    \begin{align}\label{E:GHomotopyHelper1}
        |a|_{\Co(X)} &\leq n^2|a|_{\Co(M)} \nonumber \\
        &\leq n^2C_E''|(d^* + d^+)a|_{\Co(M)}.
    \end{align}
    To obtain the exponential bound on $a \wedge dt$, note that $V_s(a \wedge dt) = (V_s a) \wedge dt$. We have $(d^* + d^+)V_s a = 0$ on $N(L-1)$ and Corollary \ref{C:ExponentialDecayMiddleNeck} applies to $V_s a \wedge dt$, yielding
    \begin{align*}
        \sup_{N(r)} |V_s a \wedge dt| &\leq C_\delta'e^{- \delta (L-2r)}\sup_{N(L-1)}|V_s a|.
    \end{align*}
    By Lemma \ref{L:VsNormChange}, it follows that 
    \begin{align}\label{E:GHomotopyHelper2}
        \sup_{N(r)} |a \wedge dt| &\leq n\sup_{N(r)} |V_s a \wedge dt| \nonumber\\
        &\leq n C_\delta'e^{- \delta (L-2r)}\sup_{N(L-1)}|V_s a| \nonumber\\
        &\leq n^2C_\delta'e^{- \delta (L-2r)}\sup_{N(L-1)}|a|.
    \end{align}
    For the third stage $s \in [2,3]$, we have $V\inv G_{3-s} V(\psi, a) = 0$. Thus $V(\psi, a)$, which is defined globally, is a solution of $G_{3-s}$. The argument for the second stage can be repeated to establish (\ref{E:GHomotopyHelper1}) and (\ref{E:GHomotopyHelper2}). Setting $C_E = \max\{C_E', n^2 C_E''\}$, $L_0 = \max\{L_1, L_2\}$ and $C_\delta = n^2 C_\delta'$ ensures that (\ref{E:TwoBounds}) is satisfied for any $s \in [0,3]$.
\end{proof}

\begin{proposition}\label{P:C0BoundReducer}
    Let $[\theta] \in \Jcal_b$ for some $b \in B$. There exists positive constants $U_0, L_0, C, \delta$ and $r$ such that the following holds. If $L \geq L_0$, then for any $s \in [0,3]$, there are no solutions $(\psi, a) \in (G^\theta_s)\inv(0)$ with $C^0$-norm in the interval $[U_0, U(L)]$, where
    \begin{align}\label{E:U(L)}
        U(L) &= C e^{\delta (L-2r)}.
    \end{align}
\end{proposition}
\begin{proof}
    Let $(\psi, a) \in (G^\theta_s)\inv(0)$ for some $s \in [0,3]$. Notice that for any stage of $G_s$, on $X - N(r)$ the pair $(\psi, a)$ satisfies 
    \begin{align*}
        D_{A_\theta+ia}\psi &= 0 \nonumber \\
        d^+a &= iF^+_{A} - i\rho_L\sigma(\psi) \nonumber \\
        d^*a &= 0 \nonumber \\
        \pr(a) &= 0.
    \end{align*}
    Lemma \ref{L:GHomotopyHelper} gives constants $C_E$ and $L_0$ such that, for $L \geq L_0$,
    \begin{align*}
        |a|_{\Co(X)} &\leq C_E|(d^* + d^+)a|_{\Co(M)}.
    \end{align*}
    Applying the Seiberg-Witten style equations above gives
    \begin{align*}
        |a|_{C^0(X)} &\leq C_E(|F^+_{A_\theta}|_{\Co} + |\sigma(\psi)|_{\Co(M)}) \nonumber \\
        &= C_E(|F^+_{A_\theta}|_{\Co} + \frac{1}{2}|\psi|^2_{\Co(M)}).
    \end{align*}
    Recall that $S = \sup_{X}\{-\sX, 0\}$ where $\sX$ is the scalar curvature of $X$. Let
    \begin{align*}
        U_0' = 1 + \sqrt{S} + C_E(|F^+_{A_\theta}|_\Co + \frac{1}{2}S).
    \end{align*}
    Note that $|F^+_{A_\theta}|_\Co$ and $S$ do not depend on $L$. To show that $|(\psi, a)|_{\Co(X)} < U_0'$ it is enough to show that $|\psi|_{\Co(X)}^2 \leq S$. By Lemma \ref{L:WeitzenbockArgument}, it suffices to show that $\sup_{X}|\psi| = \sup_{M}|\psi|$.

    For now assume that $s \in [0,1]$ so that $\psi$ satisfies $D_{A_\theta + i\rho_r^s a}\psi = 0$ and $d^+ a = iF_{A_\theta}^+ - i\rho_L \sigma(\psi)$. Inside $N(L-1)$, the \Weitzenbock formula applied to the connection $A' = A_\theta + i\rho_r^sa$ gives 
    \begin{align*}
        \Delta_g|\psi|^2 &\leq \<D_{A'}^* D_{A'} \psi - \frac{\sN}{2}\psi - F_{A'}^+\psi, \psi\>.
    \end{align*}
    Here $\sN$ is the scalar curvature of the neck, which is a positive constant. Since $A_\theta$ is flat on the neck, $F_{A'}^+ = d^+(i\rho_r^sa)$. But $d^+ a = 0$ on $N(L-1)$, so it follows that
    \begin{align}\label{E:NegativeLaplacianN(L)}
        \Delta_g |\psi|^2 &\leq -\frac{\sN}{2}|\psi|^2 + \|(d\rho_2^s \wedge a)^+\| |\psi|^2 \nonumber \\
        &= |\psi|^2\left(\sqrt{2}|(d\rho_r^s \wedge a)^+| - \frac{\sN}{2}\right).
    \end{align}
    Here $\|(d\rho_r^s \wedge a)^+\|$ is the operator norm of $d^+(\rho_r^s a) = (d\rho_r^s \wedge a)^+$ identified as an element of $\End_0(W^+)$ and $|(d\rho_r^s \wedge a)^+|$ is the norm of $(d\rho_r^s \wedge a)^+$ as a 2-form. The relation $\|(d\rho_r^s \wedge a)^+\| = \sqrt{2}|(d\rho_r^s \wedge a)^+|$ is shown in \cite[Lemma 7.4]{Salamon}. 
    
    Since $d\rho_r^s$ is supported in $N(r)$, (\ref{E:NegativeLaplacianN(L)}) guarantees that $\Delta_g |\psi|^2 < 0$ on $N(L-1) - N(r)$. It remains to show that $\Delta_g |\psi|^2 < 0$ on $N(r)$. Since $\rho_r^s$ is constant on spheres, $d\rho_r^s = \p_t \rho_r^s dt$. Define
    \begin{align*}
        R = \sqrt{2}\sup_{s \in [0,1]} |\p_t \rho_r^s|_{N(r)}.
    \end{align*} 
    If follows that 
    \begin{align}\label{E:NegativeLaplacian}
        \Delta_g |\psi|^2 &\leq |\psi|^2\left(R|a \wedge dt| - \frac{\sN}{2}\right).
    \end{align}
    Lemma \ref{L:GHomotopyHelper} provides constants $\delta, C_\delta$ such that if $L \geq L_0$, then
    \begin{align*}
        \sup_{N(r)} |a \wedge dt| &\leq C_\delta e^{- \delta (L-2r)}\sup_{N(L-1)}|a|.
    \end{align*}
    Define the constant $C > 0$ by 
    \begin{align}\label{E:CDef}
        C &= \frac{\sN}{4RC_\delta}.
    \end{align}
    This is positive since $\sN$, $R$ and $C_\delta$ are. Define $U'(L)$ by 
    \begin{align*}
        U'(L) &= C e^{\delta (L-2r)}.
    \end{align*}
    Note that the definition of $C$ is independent of $L$ and $A_\theta$. Further, it can be assumed that $L$ is large enough to ensure that $U'(L) > U'_0$. When $|(\psi,a)|_\Co \leq U'(L)$ and $L \geq L_0$, inside $N(r)$ we have
    \begin{align}\label{E:NeckStretchCalc}
        R|a \wedge dt| &\leq R C_\delta e^{- \delta (L-2r)} \sup_{N(L-1)}|a| \nonumber \\
        &\leq R C_\delta e^{- \delta (L-2r)} U'(L) \nonumber \\
        &\leq \frac{s_N}{4}.
    \end{align}
    From (\ref{E:NegativeLaplacian}) it follows that $\Delta_g|\psi|^2 < 0$ on all of $N(L-1)$. Therefore $\sup_{N(L-1)}|\psi| = \sup_{\p N(L-1)} |\psi|$ because $\Delta_g|\psi|^2$ is non-negative at an interior local maximum. Consequently $\sup_X|\psi| = \sup_{M}|\psi|$, thus $|\psi|_\Co \leq S$ and $|(\psi, a)|_\Co < U_0$. It remains to shows that $|\psi|_\Co \leq S$ for $s \in [1,3]$.

    Now suppose $(\psi, a) \in G_s\inv(0)$ for some $s \in [1,2]$ with $|(\psi, a)|_\Co \leq U'(L)$. Recall that $G_s = P + Q_s$ and $Q_s = 0$ outside of $N(1)$, hence $P(\psi, a) = 0$ on $X - N(1)$. Alternatively, $G_s = V_s\inv P V_s$ on the neck so $V_s(\psi , a)$ is a solution to $P$ on $N(L-1)$. Again we prove that $|\psi|_\Co^2 \leq S$ by showing that $\sup_X |\psi| = \sup_{M}|\psi|$.
    
    Restricting to $N(1, L-1) = \overline{N(L-1) - N(1)}$, the \Weitzenbock formula as before for the connection $A' = A_\theta + i\rho_r a$ gives 
    \begin{align*}
        \Delta_g |\psi|^2 &\leq |\psi|^2\left(R|a \wedge dt| - \frac{\sN}{2}\right).
    \end{align*}
    For $L \geq L_0$, Lemma \ref{L:GHomotopyHelper} still applies to $(\psi, a)$ yielding
    \begin{align}\label{E:NeckStretching2}
        \sup_{N(r)} |a \wedge t| &\leq C_\delta e^{- \delta (L-2r)}\sup_{N(L-1)}|a|.
    \end{align}
    Thus the calculation in (\ref{E:NeckStretchCalc}) guarantees $\Delta_g|\psi|^2 < 0$ on $N(1, L-1)$. This implies that 
    \begin{align}\label{E:MaxPsiNorms}
        \sup_X |\psi| &= \max\lbrace \sup_{N(1)}|\psi|, \sup_{M}|\psi|\rbrace.
    \end{align}
    Notice that $D_{A_\theta} V_s\psi = 0$ on $N(r-1)$. Thus Corollary \ref{C:ExponentialDecayMiddleNeckSpinor} implies the existence of constants $\delta', C_\delta' > 0$ such that 
    \begin{align}\label{E:supVs}
        \sup_{N(1)} |V_s \psi| \leq C_\delta'e^{-\delta' (r-2)} \sup_{N(r-1)} |V_s \psi|.
    \end{align}
    Fix a large enough $r$ to ensure that 
    \begin{align}\label{E:MDef}
        C_\delta'e^{-\delta' (r-2)} \leq \frac{1}{n^2}.
    \end{align}
    Note that this definition of $r$ is independent of $L$, and we can assume that $L_0 \geq 2r$. Since $V_s\psi$ is a solution to $P$ along $N(L-1)$, we have that 
    \begin{align*}
        \sup_{N(L-1)}|V_s \psi| = \sup_{\p N(L-1)}|V_s \psi|.
    \end{align*} 
    This follows from the the argument presented in the $s \in [0,1]$ case. It follows from Lemma \ref{L:VsNormChange}, (\ref{E:supVs}) and (\ref{E:MDef}) that 
    \begin{align*}
        \sup_{N(1)} |\psi| &\leq n\sup_{N(1)} |V_s \psi| \nonumber \\
        &\leq  nC_\delta'e^{-\delta' (r-2)} \sup_{N(r-1)}|V_s \psi| \nonumber \\
        &\leq  \frac{1}{n} \sup_{N(r-1)}|V_s \psi| \nonumber \\
        &\leq  \frac{1}{n} \sup_{\p N(L-1)}|V_s \psi| \nonumber \\
        &\leq  \sup_{\p N(L-1)}|\psi|.
    \end{align*} 
    That is, $\sup_{N(1)}|\psi| \leq \sup_{M}|\psi|$ and therefore $\sup_X|\psi| = \sup_{M}|\psi|$ by (\ref{E:MaxPsiNorms}). Thus Lemma \ref{L:WeitzenbockArgument} guarantees $|\psi|^2 \leq S$ and $|(\psi, a)| < U'_0$.

    For the third stage $s \in [2,3]$, we have $G_s(\psi, a) = V\inv G_{3-s} V(\psi, a) = 0$. Note that $V(\psi, a)$ is defined globally and thus $V(\psi, a)$ is a solution of $G_{3-s}$. Further, by the same calculation as Lemma \ref{L:VsNormChange}, $|V(\psi, a)|_\Co \leq n|(\psi, a)|_\Co \leq n^2 |V(\psi, a)|_\Co$. This implies that if $|(\psi, a)|_\Co \leq \frac{1}{n} U'(L)$, then $|V(\psi, a)|_\Co \leq U'(L)$ and $|(\psi, a)| \leq n U'_0$. The result follows by taking $U(L) = \frac{1}{n}U'(L)$ and $U_0 = n U'_0$, ensuring that $L_0$ is large enough so that $U(L) > U_0$ for $L \geq L_0$.
\end{proof}

The above lemma shows that given a neck length $L$ and a connection $A_\theta$, there are no elements of $(G^\theta_s)\inv(0)$ with $C^0$-norm in the interval $[U_0, U(L)]$. This will be used to find an $L^2_k$-disk in $\Acal_\theta$ with boundary that does not intersect $(G^\theta_s)\inv(0)$ for any $s \in [0,3]$. The $L^2_k$-norm of a pair $(\psi, a) \in (G^\theta_s)\inv(0)$ can be bounded by a polynomial in $|(\psi, a)|_\Co$ and $L$. The exponential increase of $U(L)$ counteracts this polynomial growth. First we show that the endpoints $(G_0^\theta)\inv(0)$ and $(G_3^\theta)\inv(0)$ are contained in an $L^2_k$-disk with radius that increases polynomially with $L$.

\begin{lemma}\label{L:G0DiskBound}
Let $[\theta] \in \Jcal_b$ for some $b \in B$. There exists positive constants $C, d$ and $L_0$ such that, for any $L \geq L_0$,
\begin{align*}
    \|(\psi, a)\|_{L^2_k} \leq CL^d
\end{align*}
for any solution $(\psi, a) \in (G^\theta_0)\inv(0) \cup (G^\theta_3)\inv(0)$ on $X_b(L)$. 
\end{lemma}
\begin{proof}
For $(\psi, a) \in (G^\theta_0)\inv(0)$ we have
\begin{align*}
    D_{A_\theta+ia}\psi &= 0 \nonumber\\
    d^+ a &= iF^+_{A_\theta} - i\rho_L\sigma(\psi) \nonumber\\
    d^* a &= 0.
\end{align*}
As in Proposition \ref{P:FirstHomotopy}, the \Weitzenbock formula gives 
\begin{align*}
    |\psi|^2_{C^0} \leq S.
\end{align*}
Since $(d + d^*)a = 0$ on $N(L-1)$, Proposition \ref{P:HarmonicNeckBound} provides constants $L_0$ and $C'$ such that $L \geq L_0$ implies
\begin{align*}
    |a|_\Co &\leq C'|(d^* + d^+)a|_\Co\\
    &\leq C'(|F^+_A|_\Co + \frac{1}{2}S).
\end{align*}
Let $U = 1 + \sqrt{S} + C'(|F^+_A|_\Co + \frac{1}{2}S)$ so that $|(\psi, a)|_\Co < U$. Notice that $|\rho_L \sigma(\psi)| \leq |\sigma(\psi)|$ and that $d\rho_L$ is supported on $N(L) - N(L-1)$. Therefore the $C^k$-norm of $\rho$ can be used to obtain a constant $C_\rho$ such that $\|\rho_L \sigma(\psi)\|_{L^p_i} \leq C_\rho\|\sigma(\psi)\|_{L^p_i}$ with $C_\rho$ independent of $L$. Now applying elliptic bootstrapping as in Remark \ref{R:EllipticBootstrappingNeck}, there are constants $C_B$ and $d$ such that 
\begin{align*}
    \|(\psi, a)\|_{L^2_k} &\leq C_B L^d(1 + U)^d\\
    &\leq C_1 L^d.
\end{align*}
The constant $C_1$ is independent of $L$ since $C_B, d$ and $U$ are. 

The argument for $(\psi, a) \in G_3\inv(0)$ is similar. Recall $G_3 = V\inv G_0 V$ so that $V(\psi, a)$ is a solution to $G_0$ and therefore
\begin{align*}
    \|V(\psi, a)\|_{L^2_k} &\leq C_1L^d.
\end{align*}
Applying $V\inv$ gives 
\begin{align*}
    \|(\psi, a)\|_{L^2_k} &= \|V\inv V(\psi, a)\|_{L^2_k}\\
    &\leq C_{V\inv} \|V(\psi, a)\|_{L^2_k}\\
    &\leq C_1C_{V\inv}(1+L)^d.
\end{align*}
Here $C_{V\inv}$ is a constant from (\ref{E:CVDef}) that is independent of $L$. The result follows with $C = \max\{C_1, C_{V\inv} C_1\}$.
\end{proof}

It remains to find an $L^2_k$-disk bundle $D$ with bounding sphere bundle $S$ that does not intersect $G_s\inv(0)$ for any $s \in [0,3]$. This is done by combining Proposition \ref{P:C0BoundReducer} with the following elliptic bootstrapping result.

\begin{lemma}\label{L:EllipticBootstrappingG}
Let $\theta \in \Jcal_b$ for some $b \in B$. There are constants $C_B$ and $d$ such that, for any $L \geq 2$, if $(\psi, a) \in (G^\theta_s)\inv(0)$ for some $s \in [0,3]$ then 
\begin{align*}
    \|(\psi, a)\|_{L^2_k} \leq C_B L^d(1+|(\psi, a)|_\Co)^d.
\end{align*}
\end{lemma}
\begin{proof}
First assume that $s \in [0,1]$ so that $(\psi, a) \in (G^A_s)\inv(0)$ implies
\begin{align*}
    D_A\psi &= -i\rho_r^s a\\
    d^+ a &= iF_A^+ - i\rho_{L} \sigma(\psi).
\end{align*}
For any $0 \leq i \leq k$ and $2 \leq p \leq 2^{k+1}$, there is a constant $C_1$ such that
\begin{align}\label{E:RhoC}
    \|\rho_r^s a\|_\Lpi &\leq C_1\|a\|_\Lpi.
\end{align}
This constant comes from the $C^k$-norm of $\rho_r^s$. Since $a$ and $\rho_r^s a$ only differ on $N(r) - N(r-1)$, $C_1$ is independent of $L$. Taking the supremum over $s \in [0,1]$, we can assume that (\ref{E:RhoC}) holds for any $s$. Similarly,
\begin{align*}
    \|d^+ a\|_\Lpi &\leq \|F_A^+\|_\Lpi + \|\rho_{L} \sigma(\psi)\|_\Lpi\\
    &\leq C_2(\|F_A^+\|_\Lpi + \|\sigma(\psi)\|_\Lpi).
\end{align*}
Once again $C_2$ can be chosen independent of $L$. Now apply bootstrapping as in Remark \ref{R:EllipticBootstrappingNeck} to obtain
\begin{align*}
    \|(\psi, a)\|_{L^2_k} \leq C_B' L^d(1+|(\psi, a)|_\Co)^d 
\end{align*}
for some constants $C_B' > 0$ and $d \geq 1$, both independent of $L$. This proves the result for $s \in [0,1]$.

If $s \in [1,2]$, we have $P(\psi, a) = 0$ on $X - N(1)$ and $PV_s(\psi, a) = 0$ on $N(1)$. On $N(1)$, the fact that $D_{A_\theta} V_s\psi = 0$ and $(d^+ + d^*) V_s a = 0$ implies that 
\begin{align*}
    \|V_s(\psi, a)\|^2_{L^2_k(N(1))} &= \|V_s(\psi, a)\|^2_{L^2(N(1))} \nonumber\\
    &\leq 2\vol(S^3) |V_s(\psi, a)|_{\Co(N(1))}^2.
\end{align*}
From Lemma \ref{L:VsNormChange} and (\ref{E:CVDef}) it follows that 
\begin{align}\label{E:EllipticBootstrappingG1}
    \|(\psi, a)\|_{L^2_k(N(1))}^2 &\leq C_{V_s\inv} \|V_s (\psi, a)\|^2_{L^2_k(N(1))} \nonumber\\
    &\leq 2C_{V_s\inv} \vol(N(1)) \cdot \sup_{N(1)} |V_s(\psi, a)|^2 \nonumber \\
    &\leq C_3 |(\psi, a)|_\Co^2.
\end{align}
The elliptic bootstrapping argument of Lemma \ref{L:EllipticBootstrappingNeck} can be applied to $(\psi, a)$ over $X - N(1)$ to obtain 
\begin{align*}
    \|(\psi, a)\|_\Lk^2 &= \|(\psi, a)\|_{L^2_k(X - N(1))}^2 + \|(\psi, a)\|_{L^2_k(N(1))}^2 \nonumber\\
    &\leq C_4 L^d (1+|(\psi, a)|_\Co)^d + C_3 |(\psi, a)|_\Co^2 \nonumber\\
    &\leq C_B''L^d (1+|(\psi, a)|_\Co)^d.
\end{align*}
Here we have assumed without loss of generality that $d \geq 2$. For $s \in [2,3]$, we have $G_s(\psi, a) = V\inv G_{3-s}V(\psi, a) = 0$. Thus $G_{3-s}V(\psi, a) = 0$ globally and Lemma \ref{L:EllipticBootstrappingNeck} applies to $V(\psi, a)$. Lemma \ref{L:VsNormChange} and (\ref{E:CVDef}) imply
\begin{align*}
    \|(\psi, a)\|_\Lk &\leq C_{V\inv} \|V (\psi, a)\|_\Lk \nonumber\\
    &\leq C_{V\inv} C_5 L^d(1 + |V(\psi, a)|_\Co)^d \nonumber\\
    &\leq C_B''' L^d(1 + |(\psi, a)|_\Co)^d.
\end{align*} 
Hence the result follows with $C_B = \max\{C_B', C_B'', C_B'''\}$.
\end{proof}

\begin{proposition}
    There are constants $r$ and $L_0$ such that, if $L \geq L_0$, then $G_s : \Acal \to \Ccal$ is a homotopy through compact perturbations of $l$.
\end{proposition}
\begin{proof}
    For any $[\theta] \in \Jcal$, Lemma \ref{L:G0DiskBound} provides constants $C_1^\theta$ and $d$ such that, for large enough $L$,
    \begin{align*}
        \|(\psi, a)\|_{L^2_k} \leq C_1^\theta L^d
    \end{align*}
    for any $(\psi, a) \in (G^\theta_0)\inv(0) \cup (G^\theta_3)\inv(0)$. The constant $d$ from the bootstrapping argument only depends on $k$, hence the same $d$ can be used for each $\theta$. Let $C_1 = \sup_{\theta \in \Jcal} C_1^\theta$ so that 
    \begin{align}\label{E:GDiskEndpointsFams}
        \|(\psi, a)\|_{L^2_k} \leq C_1L^d
    \end{align}
    for $(\psi, a)$ in any fibre of $G_0\inv(0) \cup G_3\inv(0)$.

    Again for each $[\theta] \in \Jcal$, Proposition \ref{P:C0BoundReducer} provides constants $U_0^\theta, C, \delta$ and $r$ such that, for large enough $L$,
    \begin{align*}
        |(\psi, a)|_\Co \leq U(L) \Rightarrow |(\psi, a)|_{C^0} < U^\theta_0
    \end{align*}
    so long as $(\psi, a) \in (G_s^\theta)\inv(0)$ for some $s \in [0,3]$. Recall that $U(L) = Ce^{-\delta(L-2r)}$. The constant $\delta$ is chosen based on the eigenvalues of the first order elliptic operator $\Lcal$ on $S^3$ defined in (\ref{E:Lcal}). Thus the same $\delta$ can be used for any $\theta$ on any fibre $X_b(L)$ of $E$. Further, from (\ref{E:CDef}) we can see that $C$ only depends on $\delta$, the scalar curvature of $S^3 \times [-L,L]$, and the derivative of $\rho$. Hence $C$ is also independent of $\theta$ and $b$. By similar reasoning, $r$ can also be chosen independently from $\theta$ and $b$ by (\ref{E:MDef}).
    
    Letting $U_0 = \sup_{\theta \in \Jcal} U^\theta_0$, it follows that 
    \begin{align}\label{E:GBoundReducerFams}
        |(\psi, a)|_\Co \leq U(L) \Rightarrow |(\psi, a)|_{C^0} < U_0
    \end{align}
    so long as $(\psi, a)$ is an element of some fibre of $G_s\inv(0)$ for some $s \in [0,3]$. 

    By taking a supremum over fibrewise Sobolev embeddings, there is a constant $C_S = \sup_{b \in B} C^b_S$ such that, for any $\Lk$-pair $(\psi, a)$ on any fibre $X_b(L)$,
    \begin{align}\label{E:SobEmbFams}
        |(\psi, a)|_\Co \leq C_S \|(\psi, a)\|_\Lk.
    \end{align}
    Lemma \ref{L:SENeckIndependence} ensures that $C_S$ can be chosen independently from $L$. Finally, to facilitate bootstrapping, for each $[\theta] \in \Jcal$ Lemma \ref{L:EllipticBootstrappingG} gives a constant $C_B^\theta$ such that 
    \begin{align*}
        \|(\psi, a)\|_\Lk \leq C_B^\theta L^d(1 + |(\psi, a)|_\Co)^d
    \end{align*}
    This holds so long as $(\psi, a) \in (G_s^\theta)\inv(0)$ for some $s \in [0,3]$. Once again let $C_B = \sup_{\theta \in \Jcal} C_B^\theta$ so that 
    \begin{align}\label{E:GBootstrappingFams}
        \|(\psi, a)\|_\Lk \leq C_B L^d(1 + |(\psi, a)|_\Co)^d
    \end{align}
    so long as $(\psi, a)$ is an element of some fibre of $G_s\inv(0)$ for some $s \in [0,3]$.

    Set $R(L) = \frac{U(L)}{C_S}$ and let $D \subset \Acal$ be a disk bundle with $L^2_k$-radius $R(L)$. Let $S$ denote the bounding sphere bundle of $D$. Choose $L_0$ large enough so that $L \geq L_0$ implies 
    \begin{align*}
        R(L) \geq \max\{ C_1L^d, 2C_BL^d(1+U_0)^d\}.
    \end{align*}
    This is achievable since $R(L)$ increases exponentially. By (\ref{E:GDiskEndpointsFams}), $R(L)$ contains $G_0\inv(0) \cup G_3\inv(0)$. Further, suppose $(\psi, a) \in (G^\theta_s)\inv(0) \cap D$ for some $s \in [0,3]$ and $[\theta] \in \Jcal$. Then $\|(\psi, a)\|_\Lk \leq R(L)$ and by (\ref{E:SobEmbFams}), $|(\psi, a)|_\Co \leq U(L)$. Thus $|(\psi, a)|_\Co < U_0$ by (\ref{E:GBoundReducerFams}) and (\ref{E:GBootstrappingFams}) implies that 
    \begin{align*}
        \|(\psi, a)\|_\Lk &\leq C_BL^d(1+U_0)^d\\
        &\leq \frac{1}{2}R(L).
    \end{align*}
    That is, $(G^\theta_s)\inv(0)$ does not intersect $S$ for any $\theta \in \Jcal$ and $s \in [0,3]$.
\end{proof}

\subsection{The third homotopy}

The third homotopy $H_s$ for $s \in [0,1]$ is given by 
\begin{align*}
    H_s = V\inv F_{1-s} V.
\end{align*}
This homotopy starts at $H_0 = G_3 = V\inv F_1 V$ and ends at $H_1 = V\inv \mu_{E^\tau} V$.

\begin{proposition}\label{P:HDiskBundleFam}
    The homotopy $H_s$ is a homotopy through compact perturbations of $l$.
\end{proposition}
\begin{proof}
    A solution $(\psi, a) \in (H_s)\inv(0)$ satisfies $F_{1-s}^\theta V(\psi, a) = 0$ for some $b \in B$ and $[\theta] \in \Jcal_b$. Proposition \ref{P:FDiskBundleFamilies} provides a constant $R > 0$, independent of $s$ and $\theta$, such that  
    \begin{align*}
        \|V(\psi, a)\|_{L^2_k} \leq R.
    \end{align*}
    It follows from (\ref{E:CVDef}) that 
    \begin{align*}
        \|(\psi, a)\|_{L^2_k} &= \|V\inv V(\psi, a)\|_{L^2_k}\\
        &\leq C_{V\inv} R.
    \end{align*}
    The constant $C_{V\inv}$ can be chosen independently of $\theta \in \Jcal$. The disk bundle $D \subset \Acal_k$ with fibres of $2L^2_k$-radius $C_{V\inv}R$ contains $H_s\inv(0)$ for all $s \in [0,1]$.
\end{proof}
\begin{proof}[Proof of Theorem \ref{T:FamiliesGluingTheorem}.]
    The concatenation $F \cdot G \cdot H$ is a homotopy from $\mu_E$ to $V\inv \mu_{E^\tau} V$ through compact perturbations of $l$. By Corollary \ref{C:PhiBijection}, the Bauer-Furuta classes $[\mu_E]$ and $[\mu_{E^\tau}]$ are equal in $\pi^{b^+}_{\bT^n, \Ucal}(\Jcal_E, \ind D)$, where the class $[\mu_{E^\tau}]$ is represented by the bounded Fredholm map $V\inv \mu_{E^\tau} V$.
\end{proof}
\begin{remark}
    The definition of the separating neck $N_B(L)$ required that the fibres of the neck components are of the form $S^3 \times [-L,L]$, with the application to connected sums in mind. However in Section \ref{S:MonopolesOnTheNeck}, no particularly special properties of $S^3$ were used. We only used that fact that $S^3$ has a positive scalar curvature metric and that $b_1(S^3) = 0$. Thus Theorem \ref{T:FamiliesGluingTheorem} will extend to the case that the fibres of the neck are a product $M^3 \times [-L,L]$ with $M^3$ any spherical 3-manifold.
\end{remark}

    \section{The Families Bauer-Furuta Connected Sum Formula}\label{Ch:FamConSum}

For $j \in \{1,2\}$, let $E_j \to B$ be a family of closed, oriented 4-manifolds $X_j$. To define the families connected sum, it is necessary to have sections $i_j : B \to E_j$ with normal bundles $V_j \to B$ and an orientation reversing isomorphism $\vphi : V_1 \to V_2$. Since the fibre of $E_j$ is 4-dimensional, $V_j$ is a real 4-dimensional vector bundle. Fix a metric on $V_j$ and identify the open unit disk bundle $D(V_j)$ as a tubular neighbourhood of $i_j$ with $S(V_j)$ the bounding unit sphere bundle. Let $U_j = \overline{E_j - D(V_j)}$ so that
\begin{align}\label{E:E1E2Union}
    E_1 &= U_1 \cup_{S(-V_1)} D(V_1) \nonumber\\
    E_2 &= D(V_2) \cup_{S(V_2)} U_2.
\end{align}
Here we are interpreting $S(V_2)$ as the outgoing boundary of $D(V_2)$ and $S(-V_1)$ as the ingoing boundary of $D(V_1)$, hence the negative sign. Thus $\vphi$ identifies $S(-V_1)$ with $S(V_2)$. Topologically the families connected sum $E = E_1 \#_B E_2$ is defined as 
\begin{align}\label{E:FamConSumUnion}
    E = U_1 \cup_{S(-V_1)} U_2.
\end{align}
We write $S(V) \subset E$ to denote $S(-V_1) \subset U_1$, which has been identified with $\vphi(S(-V_1)) = S(V_2) \subset U_2$. To define a metric on $E$, attach cylinders to $E_1$ and $E_2$ to get 
\begin{align*}
    \Ehat_1 &=  U_1 \cup_{S(-V_1)} \left(S(V_1) \times [0, \infty)\right) \\
    \Ehat_2 &=  \left(S(V_2) \times (\infty, 0]\right) \cup_{S(V_2)} E_2.
\end{align*}
Let $g_1$ be the metric on $S(V_1) \times [0, \infty)$ which restricts to a product of the standard round metric and interval metric on the fibres. The metric $g_1$ can be smoothly extended to $\Ehat_1$ using a collar neighbourhood. Repeat the same process to get a metric $g_2$ on $\Ehat_2$. For $L > 0$, let 
\begin{align*}
    \Ehat_1(L) &= \Ehat_1 -  \left(S(V_1) \times (L+1, \infty)\right)\\
    \Ehat_2(L) &= \Ehat_2 -  \left(S(V_2) \times (-\infty, -L-1)\right).
\end{align*}
For gluing along the cylindrical ends, define a smooth map 
\begin{align*}
    f : S^3 \times [L-1, L+1] &\to S^3 \times [-L-1, -L+1]\\
    f(x, t) &= (x, t - 2L).
\end{align*}
Now let $E(L) = E_1(L) \cup_f E_2(L)$ with metric $g_{E(L)} = g_1 \cup_f g_2$. By construction $E(L)$ is a 4-manifold family with standard fibre $X(L) = X_1 \# X_2$ that has a separating neck of length $2L$. Up to diffeomorphism, the families connected sum $E(L)$ depends only on the given sections $i_1$ and $i_2$ and the orientation reversing diffeomorphism of the normal bundles $\vphi$. 

To get a \spinc structure on $E = E(L)$, let $\sfrak_j$ be a \spinc structure on the vertical tangent bundle $T(E_j/B)$ for $j \in \{1,2\}$. Write $\Scal(E)$ to denote the set of isomorphism classes of \spinc structures on $E$. There is a restriction map defined by 
\begin{align*}
    r : \Scal(E) &\to \Scal(E_1) \times \Scal(E_2) \\
    r(\sfrak) &= (\sfrak|_{E_1}, \sfrak|_{E_2})
\end{align*}
\begin{lemma}
    The restriction map $r : \Scal(E) \to \Scal(E_1) \times \Scal(E_2)$ is a bijection onto the subset $T \subset \Scal(E_1) \times \Scal(E_2)$ defined by 
    \begin{align*}
        T &= \{(\sfrak_1, \sfrak_2) \in \Scal(E_1) \times \Scal(E_2) \mid \sfrak_1|_{S(V)} \cong \sfrak_2|_{S(V)}\}.
    \end{align*}
\end{lemma}
\begin{proof}
    From (\ref{E:FamConSumUnion}) is it clear that the image of $r$ is contained in $T$. Given $(\sfrak_1, \sfrak_2) \in T$, a \spinc structure $\sfrak$ on $E$ can be obtained from gluing, hence $r$ is surjective. It remains to prove injectivity. Suppose $\sfrak, \sfrak'$ are \spinc structures on $E$ with $r(\sfrak) = r(\sfrak')$. That is, there are isomorphisms $\vphi_j : \sfrak|_{E_j} \to \sfrak'|_{E_j}$ for $j \in \{1,2\}$. If $\vphi_1|_{S(V)} = \vphi_2|_{S(V)}$, then $\vphi_1$ and $\vphi_2$ would glue to give an isomorphism $\sfrak \to \sfrak'$. 

    Let $\psi = \vphi_1\inv|_{S(-V)} \circ \vphi_2|_{S(V)}$ so that $\vphi_2|_{S(V)} = \vphi_1|_{S(V)} \circ \psi$. The map $\psi$ is an automorphism of \spinc structures over $S(V)$ and therefore is determined by a smooth map $f : S(V) \to S^1$. We claim that $f$ extends to a smooth map $\tilde{f} : E_1 \to S^1$. Assuming this claim implies that $\psi$ extends to an automorphism $\tilde{\psi}$ of $\sfrak|_{E_1}$. Setting $\vphi_1' = \vphi_1 \circ \tilde{\psi} : \sfrak|_{E_1} \to \sfrak'|_{E_1}$ gives an isomorphism of \spinc structures with the property that $\vphi_1'|_{S(V)} = \vphi_2|_{S(V)}$ and the result follows by gluing.

    To prove the claim, recall that the set of homotopy class of maps $[S(V), S^1]$ are in bijection with $H^1(S(V) ; \Z)$. The Serre spectral sequence implies that $H^1(S(V) ; \Z)$ is isomorphic to $H^1(B ; \Z)$ by pullback. That is, the homotopy class of $f$ corresponds to the pullback of an element $\alpha \in H^1(B ; \Z)$. Pulling back $\alpha$ to $H^1(E_1 ; \Z)$ corresponds to a homotopy class of $[E_1, S^1]$ and we can choose a representative $\tilde{f}$ that restricts to $f$ on $S(V)$.
\end{proof}
\begin{corollary}\label{C:FamSpincExt}
    For $j \in \{1,2\}$, let $E_j \to B$ be a 4-manifold family equipped with a \spinc structure $\sfrak_j$ on the vertical tangent bundle. Let $i_j : B \to E_j$ be a section with normal bundle $V_j$ and assume that an orientation reversing isomorphism $\vphi : V_1 \to V_2$ is given. An extension of $\sfrak_1$ and $\sfrak_2$ to the families connected sum $E = E_1 \#_B E_2$ exists if and only if 
    \begin{align*}
        \vphi(i_1^*(\sfrak_{E_1})) \cong i_2^*(\sfrak_{E_2}).
    \end{align*}
\end{corollary}

\subsection{Families Bauer-Furuta formula}\label{Ch:FBF-CSF}

The families Bauer-Furuta connected sum formula follows from the Theorem \ref{T:FamiliesGluingTheorem} by the following observations. For a disjoint union of families $E = \coprod_{i=1}^n E_i$ the monopole map $\mu_E : \Acal \to \Ccal$ is the direct sum
\begin{align*}
    \mu_E = \bigoplus_{i=1}^n \mu_{E_i} : \bigoplus_{i=1}^n \Acal_{E_i} \to \bigoplus_{i=1}^n \Ccal_{E_i}.
\end{align*}
Assume that each $E_i$ is connected and let $\Ucal_i$ be an $S^1$-universe for $E_i$ as in (\ref{E:FamUcalDef}). Then $\Ucal = \oplus_i \Ucal_i$ is a $\bT^n$-universe with $\bT^n$ acting component-wise and the Bauer-Furuta class of $\mu_E$ is an element of $\pi_{\bT^n, \Ucal}(\Jcal; \ind l)$.
\begin{proposition}\label{P:MuSmashProduct}
    If $E = \coprod_{i=1}^n E_i$ is a disjoint union of families of 4-manifolds over $B$, then the Bauer-Furuta class $[\mu_E] \in \pi_{\bT^n, \Ucal}(\Jcal; \ind l)$ is given by the fibrewise smash product 
    \begin{align*}
        [\mu_E] &= [\mu_{E_1}] \wedge_\Jcal \cdots \wedge_\Jcal [\mu_{E_n}].
    \end{align*}
\end{proposition}
The above proposition follows directly from the definition of $[\mu_E]$ outlined in Definition \ref{T:EquivCohomotopyClass}. The next observation demonstrates a method for calculating the Bauer-Furuta invariant in the simplest cases. Recall that $H^+ \to \Jcal$ is the rank $b^+(X)$ trivial bundle with fibre $H^2_+(X ; \R)$ and that $S_{H^+} \to \Jcal$ denotes the unit sphere bundle in $H^+ \oplus \R$. In the case that $b_1(X) = 0$, the Jacobian torus $J(X)$ is just a point and $H^+$ is a bundle over $B$.
\begin{proposition}\label{P:MuTrivial}
    Let $E \to B$ be a 4-manifold family with fibre $X$ such that $b_1(X) = 0$ and assume a \spinc structure on $T(E/B)$ is given. Suppose there exists a family of metrics $\{g_b\}_{b\in B}$ on $E$ with positive scalar curvature and that $E$ admits a family of flat \spinc connections $\{A_b\}_{b\in B}$. Then the class $[\mu_E]$ is stably homotopic to the inclusion
    \begin{align*}
        \iota : B \times S^0 \to S_{H^+}.
    \end{align*}
\end{proposition}
\begin{proof}
    Let $n$ be the number of connected components of $E$. For $t \in [0,1]$ define a homotopy 
    \begin{align*}
        \mu_t : L^2_k(E, W^+ \oplus T^*(E/B)) \oplus \R^n \to L^2_{k-1}(E, W^- \oplus \Lambda^2_+ T^*(E/B) \oplus \R)
    \end{align*}
    by the formula 
    \begin{align*}
        \mu_t(\psi, a, f) = (D_{A+ta}\psi, d^+a - t\sigma(\psi), d^*a + f).
    \end{align*}
    Since $b_1(X) = 0$ and $F_A = 0$, we have $\mu_1 = \mu_E$. Further, $\mu_0$ is the linearised monopole map $l = D_A \oplus d^+ \oplus d^*$. We show that $\mu_t$ is a homotopy through compact perturbations of $l$. Suppose that $\mu_t(\psi, a, f) = 0$ for some $t \in [0,1]$. This implies that 
    \begin{align*}
        D_{A+ta}\psi &= 0\\
        d^+ a &= t\sigma(\psi)\\
        d^*a &= 0\\
        f &= 0.
    \end{align*}
    It follows from the \Weitzenbock formula that 
    \begin{align*}
        \Delta_g |\psi|^2 + \frac{s}{2}|\psi|^2 + \frac{t^2}{2}|\psi|^4 \leq 0.
    \end{align*}
    At a maximum of $|\psi|$ we obtain 
    \begin{align*}
        \frac{s}{2}|\psi|^2_{C^0} + \frac{t^2}{2}|\psi|^4_{C^0} \leq 0.
    \end{align*}
    Since $s > 0$ we have $\psi = 0$. This in turn implies that $d^+ a = 0$. Since $d^*a = 0$ and $b_1(X) = 0$, $a$ is harmonic and therefore $a = 0$. Thus $\mu\inv(0)$ contains only one point and certainly is bounded. That is, $\mu$ is a compact homotopy.

    Recall that $\ind l = \ind D_A - b^+(X)$. The positive scalar curvature and the fact that $F_A = 0$ implies that both $\ker D_A = 0$ and $\coker D_A = 0$. Thus $D_A$ is an isomorphism and therefore the Bauer-Furuta finite dimensional approximation of $l$ is stably homotopic to the inclusion $\iota$.
\end{proof}
Let $V \to B$ be an $SO(4)$-vector bundle with a \spinc structure $\sfrak$ on the vertical tangent space $T(V/B)$. This induces a \spinc structure on $S_V = S(\R \oplus V)$ in the following way. Let $\Fr(V)$ denote the vertical oriented frame bundle of $V$. The \spinc structure on $V$ determines a principle $\Spinc(4)$-bundle $\Pcal_V \to \Fr(V)$ which pulls back to a principle $\Spinc(5)$-bundle $\Pcal_{\R \oplus V} \to \Fr(\R \oplus V)$. Let $i : \Fr(S(V)) \to \Fr(\R \oplus V)$ be the inclusion map of frames defined by the outward normal first convention. Then $i^*(\Pcal_{\R \oplus V}) \to \Fr(S(V))$ is the \spinc structure on $S_V$ induced by $\sfrak$.

\begin{corollary}\label{C:FamMuTriv}
    Let $V \to B$ be an $SO(4)$-bundle with a \spinc structure and give $\pi : S_V \to B$ the induced \spinc structure on the vertical tangent bundle $T(S_V/B)$. Then the class $[\mu_{S_V}]$ is stably homotopic to the identity $\id : B \times S^0 \to B \times S^0$.
\end{corollary}
\begin{proof}
    Since $b_1(S^4) = b_2(S^4) = 0$, the pullback map $\pi^* : H^2(B ; \Z) \to \pi^*(S_V ; \Z)$ is an isomorphism by the Serre spectral sequence. Let $\Lcal \to S_V$ be the canonical line bundle of the induced \spinc structure on $T(S_V/B)$. Then the first chern class $c_1(\Lcal) \in H^2(S_V ; \Z)$ is in the image of $\pi^*$. Thus there exists a connection $A$ on $\Lcal$ with curvature $F_A = \pi^*(\omega)$ for some 2-form $\omega \in \Omega^2(B)$. Let $i_b : \pi\inv(b) \to S_V$ be the inclusion of the fibre over $b \in B$. Then the restriction $A_b = i_b^* A$ is flat since $F_{A_b} = i_b^* \pi^* \omega = 0$.

    Since the structure group of $V$ is $SO(4)$, the fibres of $S_V$ can be equipped with the standard round metric which has positive scalar curvature. By Proposition \ref{P:MuTrivial}, $[\mu_{S_V}] = [\id]$.
\end{proof}
Finally, we have all the necessary tools to derive Bauer-Furuta connected sum formula. We begin with the unparameterised case, which was first formulated by Bauer in \cite{BF2}. Afterwards, we prove the families formula which is a new result.
\begin{theorem}[\cite{BF2} Theorem 1.1]\label{T:BFConnectedSumFormula}
    Let $X = \#_{i} X_i$ be a connected sum of $n$ closed, oriented, 4-manifolds. The Bauer-Furuta invariant $[\mu_X]$ is given by the formula
    \begin{align}
        [\mu_X] &= \bigwedge_{i=1}^n [\mu_{X_i}].
    \end{align}
\end{theorem}
\begin{proof}
    It is enough to prove the result for a connected sum of two 4-manifolds. Define 
    \begin{align}\label{E:YSums}
        Y_1 &= X_1 \# S^4 \nonumber\\
        Y_2 &= S^4 \# X_2 \nonumber\\
        Y_3 &= S^4 \# S^4.
    \end{align}
    Set $Y = \coprod_i Y_i$. By the connected sum construction outlined in \ref{Ch:FamConSum}, we can choose a metric that gives $Y$ the structure of a separating neck. The negative components of $Y$ are given by the left summands of (\ref{E:YSums}) and the positive components by the right summands. Further, any choice of \spinc structure on $X_1$ and $X_2$ extends uniquely to a \spinc structure on $Y$. Now $[\mu_{Y_1}] = [\mu_{X_1}]$, $[\mu_{Y_2}] = [\mu_{X_2}]$ and Proposition \ref{P:MuTrivial} implies that $[\mu_{Y_3}] = [\id]$. By Proposition \ref{P:MuSmashProduct} we have 
    \begin{align*}
        [\mu_Y] &= [\mu_{X_1}] \wedge [\mu_{X_2}].
    \end{align*}
    Let $\tau$ be the even permutation $\tau = (123)$ so that
    \begin{align*}
        Y^\tau = (X_1 \# X_2) \amalg (S^4 \# S^4) \amalg (S^4 \# S^4).
    \end{align*}
    Applying Propositions \ref{P:MuSmashProduct} and \ref{P:MuTrivial} again yields
    \begin{align*}
        [\mu_{Y^\tau}] &= [\mu_{X_1 \# X_2}].
    \end{align*}
    Thus Theorem \ref{T:FamiliesGluingTheorem} implies that $[\mu_X] = [\mu_{X_1}] \wedge [\mu_{X_2}]$.
\end{proof}
\begin{remark}\label{R:tauOdd}
    In the construction of $X^\tau$ it is assumed that $\tau$ is an even permutation, however this assumption is unnecessary for Theorem \ref{T:FamiliesGluingTheorem}. If $\tau$ happens to be odd, then replace $X$ with the disjoint union 
    \begin{align*}
        X' &= X \amalg (S^4 \# S^4) \amalg (S^4 \# S^4).
    \end{align*} 
    Now include an extra transposition in $\tau$ that swaps the last two $S^4$ components. As shown in the argument above, $[\mu_X] = [\mu_{X'}]$.
\end{remark}

\begin{theorem}[Families Bauer-Furuta Connected Sum Formula]\label{T:FamBFCSF}
    For $j \in \{1,2\}$, let $E_j \to B$ be a 4-manifold family equipped with a \spinc structure $\sfrak_j$ on the vertical tangent bundle. Let $i_j : B \to E_j$ be a section with normal bundle $V_j$ and assume that $\vphi : V_1 \to V_2$ is an orientation reversing isomorphism satisfying 
    \begin{align*}
        \vphi(i_1^*(\sfrak_{E_1})) \cong i_2^*(\sfrak_{E_2}).
    \end{align*} 
    Then the families Bauer-Furuta class of the fiberwise connected sum $E = E_1 \#_B E_2$ is
    \begin{align}
        [\mu_{E}] &= [\mu_{E_1}] \wedge_\Jcal [\mu_{E_2}].
    \end{align}
\end{theorem}
\begin{proof}
    By Corollary \ref{C:FamSpincExt}, there is a unique \spinc structure on the vertical tangent space of $E$ that extends $\sfrak_1$ and $\sfrak_2$. Let $U_j = \overline{E_j - D(V_j)}$ as in (\ref{E:E1E2Union}) so that 
    \begin{align*}
        E_1 &= U_1 \cup_{S(-V_1)} D(V_1) \\
        E_2 &= D(V_2) \cup_{S(V_2)} U_2.
    \end{align*}
    Recall that $S(V) \subset E$ denotes $S(-V_1) \subset E_1$ and $S(V_2) = \vphi(S(-V_1)) \subset E_2$. For any $L > 0$, we can choose a metric on $E_1$ and $E_2$ that gives both of them a separating neck of length $2L$. Let $F = E_1 \amalg E_2$ so that $[\mu_F] = [\mu_{E_1}] \wedge_\Jcal [\mu_{E_2}]$ by Proposition \ref{P:MuSmashProduct}. Let $\tau$ be the transposition $(12)$ so that 
    \begin{align*}
        F^\tau &= \left(U_1 \cup_{S(V)} U_2\right) \amalg \left(D(V_2) \cup_{S(V)} D(V_1)\right).
    \end{align*}
    That is, $F^\tau = E \amalg S_{V_2}$. The \spinc structure on $S_{V_2}$ is induced by $\sfrak_2$ and therefore $[\mu_{S_{V_2}}] = [\id]$ by Corollary \ref{C:FamMuTriv}. Thus $[\mu_{F^\tau}] = [\mu_E]$ by Proposition \ref{P:MuSmashProduct}. Theorem \ref{T:FamiliesGluingTheorem} implies that $[\mu_F] = [\mu_{F^\tau}]$ and therefore
    \begin{align*}
        [\mu_E] &= [\mu_{E_1}] \wedge_\Jcal [\mu_{E_2}].
    \end{align*}
    Note that the fact that $\tau$ is an odd permutation is not an issue by Remark \ref{R:tauOdd}. 
\end{proof}
Of course, this formula extends to a connected sum of arbitrarily many families. Further, the diffeomorphism type of the connected sum $E = E_1 \#_B E_2$ depends on the sections $i_1$, $i_2$ and the isomorphism $\vphi$, however the class $[\mu_{E_1}] \wedge_\Jcal [\mu_{E_2}]$ does not. That is, if $E'$ is obtained as a connected sum of $E_1$ and $E_2$ for different $i_1, i_2$ and $\vphi$, then $[\mu_E] = [\mu_{E'}]$. 

\bibliography{Bibliography}{}
\bibliographystyle{plain}

\end{document}